\documentclass[leqno]{article}
\usepackage{amssymb,amsmath,amsthm}

\newtheorem{proposition}{Proposition}[section]
\newtheorem{lemma}[proposition]{Lemma}
\newtheorem{corollary}[proposition]{Corollary}
\newtheorem*{thma}{Theorem A}
\newtheorem*{thmb}{Theorem B}
\newtheorem*{thmc}{Theorem C}
\newtheorem*{thmd}{Theorem D}
\newtheorem*{thme}{Theorem E}
\newtheorem*{thmf}{Theorem F}
\newtheorem*{thmg}{Theorem G}
\newtheorem*{thmh}{Theorem H}
\newtheorem*{thmi}{Theorem I}
\newtheorem*{thmj}{Theorem J}
\newtheorem*{thmk}{Theorem K}
\newtheorem*{thml}{Theorem L}
\newtheorem*{IPCcplthm}{\IPC-Completeness Theorem}
\newtheorem*{WEMcplthm}{\WEM-Completeness Theorem}
\newtheorem*{embedthm}{Embedding Theorem}
\newtheorem*{dembedthm}{Dual Embedding Theorem}

\theoremstyle{definition}
\newtheorem{definition}[proposition]{Definition}
\newtheorem{question}[proposition]{Open Question}
\newtheorem{remark}[proposition]{Remark}

\newif\ifautoproof\autoprooffalse

\def\mathtext#1#2{\ifmmode#1{#2}\else$#1{#2}$\fi}
\def\mathdefzero#1#2#3{\relax\def#1{\relax\mathtext{#3}{#2}}}
\def\mathdefone#1#2#3{\relax\def#1##1{\relax\def\1{{##1}}\mathtext{#3}{#2}}}
\def\mathdeftwo#1#2#3{\relax\def#1##1##2{\relax\def\1{{##1}}
   \def\2{{##2}}\mathtext{#3}{#2}}}
\def\mathdefthree#1#2#3{\relax\def#1##1##2##3{\relax\def\1{{##1}}
   \def\2{{##2}}\def\3{{##3}}\mathtext{#3}{#2}}}
\def\mmdef#1#2#3#4{\ifcase#1\mathdefzero{#2}{#3}{#4}
   \or\mathdefone{#2}{#3}{#4}   \or\mathdeftwo{#2}{#3}{#4}
   \or\mathdefthree{#2}{#3}{#4} \or\mathdeffour{#2}{#3}{#4}\fi}

\def\mdef#1#2[#3]{\mmdef{#1}{#2}{#3}{}}
\def\reldef#1#2[#3]{\mmdef{#1}{#2}{#3}{\mathrel}}
\def\bindef#1#2[#3]{\mmdef{#1}{#2}{#3}{\mathbin}}
\def\opdef#1#2[#3]{\mmdef{#1}{#2}{#3}{\mathop}}

\def\ssfdef#1[#2]{\def#1{{\ssf {#2}}}}
\mdef3\lis[\1_{\2}\ldots\1_{\3}]
\mdef3\llist[\1_{\2},\ldots,\1_{\3}]
\mdef2\zlist[\1_0,\ldots,\1_{\2-1}]
\mdef2\zlis[\1_0\ldots\1_{\2-1}]
\mdef3\andlist[\1_{\2}\land\cdots\land\1_{\3}] 
\mdef3\orlist[\1_{\2}\lor\cdots\lor\1_{\3}]    
\def\lAnd{\bigwedge}
\def\lOr{\bigvee}

\mdef1\:[\dot{\overbrace{\1}}] 
\mdef1\.[\dot{\1}]
\mdef1\modelof[\>\models{\kern-9pt}_%
   {\lower5pt\hbox{$\scriptstyle\1$}}\>]
\def\notmodels{\mathrel|\joinrel\not=}
\mdef1\notmodelof[\>\notmodels{\kern-9pt}_%
   {\lower5pt\hbox{$\scriptstyle\1$}}\>]

\def\incl{\subseteq}
\mdef1\set[\left\{\mkern1.5mu\1\mkern1.5mu\right\}]           
\mdef1\bigset[\bigl\{\mkern1.5mu\1\mkern1.5mu\bigr\}]           
\mdef3\setlist[\left\{\mkern2mu{\1}_{\2},\ldots,
     {\1}_{\3}\mkern2mu\right\}]
\mdef2\setzlist[\left\{\mkern2mu{\1}_0,\ldots,
     {\1}_{\2-1}\mkern2mu\right\}]
\mdef2\setof[\left\{\,\1:\2\,\right\}]                     
\mdef2\bigsetof[\bigl\{\,\1:\2\,\bigr\}]                     
\mdef2\Bigsetof[\Bigl\{\,\1:\2\,\Bigr\}]                     
\def\power{\wp}

\def\poweromm{\wp_\omega^{-}}

\def\iff{\,\leftrightarrow\,}
\def\Iff{\,\Longleftrightarrow\,}

\def\implies{\,\rightarrow\,}
\def\Implies{\,\Longrightarrow\,}

\def\qand{\quad\hbox{and}\quad}
\def\rand{\;\hbox{and}\;}
\def\qqand{\qquad\hbox{and}\qquad}

\def\qor{\quad\hbox{or}\quad}

\def\qIff{\quad\Iff\quad}
\def\rIff{\;\Iff\;}

\def\qImplies{\quad\Implies\quad}
\def\rImplies{\;\Implies\;}

\def\lAnd{\bigwedge}
\def\DDegree#1#2{{\ssbf#1}\ifcat#20
   \ifcase#2\or'\or''\or'''\else^{(#2)}\fi\else
   ^{(#2)}\fi} 
\def\degree#1#2{\ifmmode\DDegree{#1}{#2}\else
   $\DDegree{#1}{#2}$\fi} 
\def\lnot{\neg\>}
\mdef1\alef[\aleph_{\1}]

\def\eqw{\equiv_{\ssf w}}
\def\eqs{\equiv_{\ssf s}}
\def\eqt{\equiv_T}
\def\zerow{{\bf 0}_{\ssf w}}
\def\zeros{{\bf 0}_{\ssf s}}
\def\zerot{{\bf 0}_T}
\def\onew{{\bf 1}_{\ssf w}}
\def\ones{{\bf 1}_{\ssf s}}
\def\onet{{\bf 1}_T}
\def\inftyw{{\bf \infty}_{\ssf w}}
\def\inftys{{\bf \infty}_{\ssf s}}

\mdef2\prf[\{\1\}^{\2}]
\mdef3\pprf[\{\1\}^{\2}_{\3}]

\def\lbrac{[}
\def\rbrac{]}
\mdef3\mm[{\ssf m}^{\1}_{\2,\3}]

\mdef1\tree[{\ssf T}_{\1}]
\mdef2\treelg[{\ssf T}_{\1}^{\2}]
\mdef2\ttree[{\ssf T}_{\1,\2}]
\mdef2\ttreelg[{\ssf T}_{\1,\2}^{\2}]
\mdef3\ttreelgg[{\ssf T}_{\1,\2}^{\3}]

\mdef1\trst[{\ssf V}^{\1}]
\mdef2\ptrst[{\ssf V}^{\1}_{\2}]
\mdef2\pltrst[{\ssf V}^{\1}_{<\2}]

\ssfdef\ext[Ext]

\mdef1\gs[{\ssf S}_{\1}]
\mdef1\gsf[{\ssf S}^F_{\1}]
\mdef2\gsd[{\cal S}^{\1}_{\2}]
\mdef3\gsdd[{\cal S}^{\1,\,\2}_{\3}]
\mdef2\gsdb[\overline{\cal S}^{\1}_{\2}]
\mdef0\deg[{\ssf dg}]
\mdef0\degs[{\ssf dg}_{\ssf s}]
\mdef0\degt[{\ssf dg}_T]
\mdef0\degw[{\ssf dg}_{\ssf w}]
\mdef0\Dgt[{\mathbb D}_T]
\mdef0\Dgs[{\mathbb D}_{\ssf s}]
\mdef0\Dgw[{\mathbb D}_{\ssf w}]
\mdef0\Dgpt[{\mathbb P}_T]
\mdef0\Dgps[{\mathbb P}_{\ssf s}]
\mdef0\Dgpw[{\mathbb P}_{\ssf w}]
\reldef0\leqs[\leq_{\ssf s}]
\reldef0\leqw[\leq_{\ssf w}]
\reldef0\geqw[\geq_{\ssf w}]
\reldef0\les[<_{\ssf s}]
\reldef0\lew[<_{\ssf w}]
\mdef1\om[{}^\omega\1]
\mdef1\dnr[{\ssf DNR}_{\1}]

\mdef3\zslist[\1_0^{\3},\ldots,\1_{\2-1}^{\3}]
\mdef2\pre[{}^{\1}\2]

\def\mm{{\ssf m}}
\def\Impliess{\;\Implies\;}

\catcode`\@=11
\def\cases#1{\left\{\,\vcenter
   {\m@th\ialign{$##\hfil$&\quad##\hfil\crcr#1\crcr}}\right.}
\catcode`\@=12

\def\join{\mathchoice{\mathrel{\rlap{$\vee$}\mskip1.8 mu\vee}}{\mathrel{\rlap{$\vee$}\mskip1.8 mu\vee}}{\mathrel{\rlap{$\scriptstyle\vee$}\mskip1.8 mu\vee}}{\mathrel{\rlap{$\scriptscriptstyle\vee$}\mskip1.8 mu\vee}}}
\def\meet{\mathchoice{\mathrel{\rlap{$\wedge$}\mskip1.8 mu\wedge}}{\mathrel{\rlap{$\wedge$}\mskip1.8 mu\wedge}}{\mathrel{\rlap{$\scriptstyle\wedge$}\mskip1.8 mu\wedge}}{\mathrel{\rlap{$\scriptscriptstyle\wedge$}\mskip1.8 mu\wedge}}}
\opdef0\latneg[\rlap{$\lnot$}\mskip2mu\lnot\kern-2pt]
\opdef0\bigjoin[\rlap{$\displaystyle\bigvee$}\displaystyle\bigvee]
\opdef0\bigmeet[\rlap{$\displaystyle\bigwedge$}\displaystyle\bigwedge]
\bindef0\latimpl[\rlap{\kern.4pt$\rightarrow$}\rightarrow]

\def\djoin{\buildrel\lower3pt\hbox{$\scriptscriptstyle\dualmarker$}\over\join}
\def\dmeet{\buildrel\lower3pt\hbox{$\scriptscriptstyle\dualmarker$}\over\meet}
\reldef0\dimbeds[\buildrel\lower5pt\hbox{$\scriptscriptstyle\dualmarker$}\over\hookrightarrow]
\let\dembeds=\dimbeds
\def\qdembeds{\quad\dembeds\quad}
\reldef0\dlatneg[\,\buildrel\lower2pt\hbox{$\scriptscriptstyle\dualmarker$}\kern1pt\over\latneg]

\def\dlatimpl{\buildrel\lower2pt\hbox{$\scriptscriptstyle\dualmarker$}\over\latimpl}
\def\dleq{\buildrel\lower2pt\hbox{$\scriptscriptstyle\dualmarker$}\over\leq}
\let\dualmarker=\circ

\def\latz{{\bf 0}}
\def\lato{{\bf 1}}

\def\dga{{\bf a}}
\def\dgaw{{\bf a}_{\ssf w}}
\def\dgawstar{{\bf a}_{\ssf w}^*}
\def\dgb{{\bf b}}

\def\dgbwstar{{\bf b}_{\ssf w}^*}

\def\dgd{{\bf d}}

\def\dgp{{\bf p}}
\def\dgq{{\bf q}}
\def\dgr{{\bf r}}
\def\dgs{{\bf s}}
\def\dgu{{\bf u}}

\mdef0\calp[{\cal P}]
\mdef0\calq[{\cal Q}]
\mdef0\calr[{\cal R}]
\mdef0\cals[{\cal S}]
\mdef0\calt[{\cal T}]
\mdef0\calu[{\cal U}]
\mdef0\calv[{\cal V}]
\mdef0\calw[{\cal W}]
\mdef0\calx[{\cal X}]
\mdef0\caly[{\cal Y}]
\mdef0\calz[{\cal Z}]
\mdef0\truth[{\ssf T}]
\mdef0\falsity[{\ssf F}]

\def\Ifff{\,\Iff\,}
\def\dIff{\;:\Longleftrightarrow\;}

\def\dff#1{{\it#1}}
\newfont{\bfit}{cmbxti10}
\def\dff#1{{\bfit #1}}
\let\ssf=\sf

\def\bframe#1{\begin{frame}\frametitle{#1}}

\def\frdg#1{\left[\mkern-2.5mu\left[\,#1\,
       \right]\mkern-2.5mu\right]}  
\def\dgipc#1{\left[\mkern-2.5mu\left[\,#1\,
       \right]\mkern-2.5mu\right]_\IPC}  
\let\ipcdg=\dgipc
\let\sdg=\frdg     
\def\frdgs{\frdg{{\cal S}}} 
\def\frdgt{\frdg{{\cal T}}} 
\def\frdgu{\frdg{{\cal U}}}

\let\iso=\simeq
\reldef0\submeet[\rlap{$\scriptscriptstyle\wedge$}\mskip1.5 mu{\scriptscriptstyle\wedge}]
\reldef0\subjoin[\rlap{$\scriptscriptstyle\vee$}\mskip1.5 mu{\scriptscriptstyle\vee}]
\mdef0\freedist[\mathbb{F}_{\submeet}]
\mdef0\free[\mathbb{F}]
\bindef0\sublatimpl[\rlap{\kern.5pt$\scriptscriptstyle\rightarrow$}{\scriptscriptstyle\rightarrow}]
\def\subdlatimpl{\,\buildrel\lower2pt\hbox{$\scriptscriptstyle\dualmarker$}\over\sublatimpl\,}

\def\Th{{\ssf Th}}
\def\Thd{{\ssf Th}^\dualmarker}
\def\CPC{{\ssf CPC}}
\def\IPC{{\ssf IPC}}
\def\WEM{{\ssf WEM}}

\def\At{{\ssf At}}
\def\Mod{{\ssf Mod}}
\def\lg{{\ssf lg}}
\reldef0\pripc[\vdash_{\kern-4pt\lower.5pt\hbox{$\scriptscriptstyle\IPC$}}]
\reldef0\prx[\vdash_{\kern-4pt\lower.5pt\hbox{$\scriptscriptstyle X$}}]
\reldef0\prwem[\vdash_{\kern-4pt\lower.5pt\hbox{$\scriptscriptstyle\WEM$}}]
\reldef0\notpripc[\not\vdash_{\kern-2pt\lower.5pt\hbox{$\scriptscriptstyle\IPC$}}]
\reldef0\notprwem[\not\vdash_{\kern-2pt\lower.5pt\hbox{$\scriptscriptstyle\WEM$}}]

\let\mathstr=\mathfrak
\def\forces{\mathrel{\raise 7pt
   \hbox{$\mathchar "330C\mkern-3.5mu \mathchar  "330C$}
   \mkern-7.5mu
   {\kern 1.5pt\hbox{$\raise .5 pt\hbox{}
   \over\kern6pt$}}}}
\mdef1\random[{\ssf R}_#1]

\newif\ifendpf
\def\squarebox#1{\vbox{\hrule\hbox{\vrule height#1\kern#1\vrule}\hrule}}
\def\endpf{\ifendpf\ifmmode\enspace\fi\squarebox{6pt}\fi\global\endpffalse}
\def\noproof{\ \qed}
\def\p#1{{\ssf p}_{#1}}
\def\q#1{{\ssf q}_{#1}}
\def\pp{{\ssf p}}
\def\sepset{{\ssf S}}

\mdef0\l[{\mathstr L}]
\mdef0\lz[{\mathstr L}_{\ssf 0}]
\mdef0\lzo[{\mathstr L}_{\ssf 0}^{\ssf 1}]
\mdef0\lo[{\mathstr L}^{\ssf 1}]

\newif\ifautoqed
\autoqedtrue
\def\QED{\ifautoqed\else\ \qed\fi}

\begin{document}

\title{A survey of Mu\v cnik and Medvedev degrees}

\author{Peter G. Hinman}

\date{Draft of \today}

\maketitle 

\begin{abstract}We survey the theory of Mu\v cnik (weak) and Medvedev (strong) degrees of subsets of $\pre\omega\omega$ with particular attention to the degrees of $\Pi^0_1$ subsets of $\pre\omega2$. Later sections present proofs, some more complete than others, of the major results of the subject.
\end{abstract}

\begin{section}{Introduction}
The formal introduction of the notion of Medvedev reducibility and the associated notion of Medvedev degree was in \cite{Med55} in 1955, but the idea had been around already for some time; it may have been first suggested by Kolmogorov as a possible way of modeling the Intuitionistic Propositional Calculus. Similarly, the notion of Mu\v cnik reducibility dates formally from \cite{Mu} but was probably known earlier. Medvedev reducibility was  treated briefly in Hartley Rogers' influential text \cite{Ro}, and both notions were studied by a small number of Soviet mathematicians in the period 1955-1990, but they produced only some 10 articles and the subject seemed firmly in the backwater of logic. 

Attention picked up with Sorbi's Ph.D. thesis and the series of papers emanating from it in the early 1990's. The most important event in revitalizing the subject was Simpson's suggestion in a posting \cite{SiFOM} in the online forum FOM in 1999 that the degrees of $\Pi^0_1$ subsets of $\pre\omega2$ constitute a natural generalization of the Turing degrees of recursively enumerable sets of natural numbers. In the 11 years since  there has been a steady growth in the field to the point that it seems worthwhile to collect together some of the most important results in a unified way for the benefit of a researcher new to the area. 

We structure this survey as follows. Sections 2-6 describe important aspects of the theory with full definitions and proofs of some of the simpler results. Major results are designated Theorem A etc., and in Sections 7-16 we give relatively complete proofs of many of these. For the others we provide outlines of proofs and/or references. 

Mu\v cnik and Medvedev reducibilities are built on Turing reducibility and for the most part we will assume that the reader knows the basics of Recursion (Computability) Theory as contained, for example, in the first six chapters of \cite{Soa} or the first two sections of Chapters IV and VIII of \cite{Hi}. Our notation will generally conform with this standard, but we mention here some of our most important conventions. ${}^\omega\omega$ is the set of all functions $f:\omega\to\omega$ and ${}^\omega k$ is the set of those with all values in \set{0,\ldots,k-1}. ${}^{<\omega}\omega$  (${}^{<\omega}k$) is the set of finite sequences of natural numbers ($<k$), and ${}^m\omega$ (${}^m k$) is the subset of these of length $m$. For $\sigma\in\pre{<\omega}\omega$, $|\sigma|$ and $\lg(\sigma)$ both denote the length of $\sigma$, $\sigma=(\sigma(0),\ldots,\sigma(|\sigma|-1))$, and $\sigma^\frown f\in\pre\omega\omega$ is defined by
\[(\sigma^\frown f)(i)=\cases{\sigma(i),&if $i<|\sigma|$;\cr f(i-|\sigma|),&otherwise.\cr}\]
Similarly, $\sigma^\frown\tau$ is the obvious finite sequence of length $|\sigma|+|\tau|$. For $P\incl\pre\omega\omega$, 
\[\sigma^\frown P:=\setof{\sigma^\frown f}{f\in P}\]
and $\sigma\incl f$ ($\sigma\incl\tau$) express that $\sigma$ is an initial segment of $f$ (of $\tau$).

We assume given an indexing  $\functionof{\set a}{a\in\omega}=\functionof{\Phi_a}{a\in\omega}$ of the partial recursive functionals; if $\Phi=\set a$, then $\Phi(f)$ is the partial function also denoted $\set a^f$ such that  for all $m\in\omega$, $\Phi(f)(m)\simeq\set a^f(m)$. We assume also some precise notion of a computation $\set a^f_s(m)$ being completed in at most $s$ steps and therefore depending on at most the finite initial segment $f\restrict s:=\bigl(f(0),\ldots,f(s-1)\bigr)$ of $f$. $\set a^f_s$ denotes the longest finite sequence $\bigl(\set a^f_s(0),\ldots,\set a^f_s(n-1)\bigr)$ such that all of the indicated values are defined. For $\sigma\in\pre{<\omega}\omega$, $\set a^\sigma:=\set a^f_{|\sigma|}$ for some (any) $f\supseteq\sigma$. If $\Phi=\set a$, then $\Phi(\sigma)$ denotes the finite sequence $\set a^\sigma$. For $P,Q\incl\pre\omega\omega$, $\Phi:Q\to P$ means that  for all $g\in Q$, $\Phi(g)$ is a total function belonging to $P$. In such a case, $\Phi(Q):=\setof{\Phi(g)}{g\in Q}$. With slight imprecision, we also think of \functionof{\set a}{a\in\omega} as an enumeration of the partial recursive \textit{functions} and in particular, \functionof{W_a}{a\in\omega}, with $W_a:={\ssf Domain}(\set a)$ an enumeration of the recursively enumerable (r.e.) sets along with their finite \textit{stage approximations} \functionof{W_{a,s}}{a,s\in\omega}.

Many of the results below involve partial orderings, lattices and Boolean algebras; although these will be familiar to almost all readers, we introduce here some or our conventions and notations. A \dff{partial ordering} is always a structure ${\mathstr P}=(P,\leq)$ such that $\leq$ is a reflexive, transitive and antisymmetric binary relation. A partial ordering $\mathstr L$ is a \dff{lattice} iff each pair $a,b\in L$ has a greatest lower bound or \dff{meet} $a\meet b$:
\[(\forall x\in L)\;[x\leq a\text{ and }x\leq b\qIff x\leq a\meet b],\]
and a least upper bound or \dff{join} $a\join b$:
\[(\forall x\in L)\;[a\leq x\text{ and }b\leq x\qIff  a\join b\leq x.]\]
In this case we may expand the signature and write ${\mathstr L}=(L,\leq,\join,\meet)$, but often we write simply ${\mathstr L}=(L,\join,\meet)$ with the understanding that 
$a\leq b$ is the relation defined by the equivalent conditions
\[a=a\meet b\qIff a\join b=b.\]
$\mathstr L$ is an \dff{upper} (\dff{lower}) \dff{semilattice} iff joins (meets) but not necessarily meets (joins) always exist.

A lattice ${\mathstr L}$ is \dff{distributive} iff it satisfies, for all $a,b,c\in L$,
\[a\join(b\meet c)=(a\join b)\meet (a\join c)\qand a\meet(b\join c)=(a\meet b)\join(a\meet c),\]
and \dff{bounded} iff it has a least element $\latz$ and a greatest element $\lato$. For any finite set $A=\set{a_0,\ldots,a_{k-1}}\incl L$, we write
\[\bigmeet A\text{ for }a_0\meet\cdots \meet a_{k-1}\qand\bigjoin A\text{ for }a_0\join\cdots \join a_{k-1}\]
with the convention that $\bigmeet\emptyset=\lato$ and $\bigjoin\emptyset=\latz$.

A Boolean algebra is a bounded distributive lattice on which there exists a unary operation \dff{complement} such that
\[a\meet\latneg a=\latz\qand a\join\latneg a=\lato.\]
\end{section}

\begin{section}{Global degree structures}
For comparison, we think of Turing reducibility in the following form: for $f,g\in{}^\omega\omega$,
\[f\leqt g\Longleftrightarrow\exists\Phi\;[f=\Phi(g)],\]
where $\Phi$ ranges over the partial recursive functionals. Then

\begin{definition}For $P,Q\incl{}^\omega\omega$,
\begin{align*}
P\leqw Q&\Longleftrightarrow(\forall g\in Q)(\exists f\in P)\;f\leq_Tg\\
&\Longleftrightarrow(\forall g\in Q)\exists\Phi[\Phi(g)\in P],\\
\noalign{and}
P\leqs Q&\Longleftrightarrow\exists\Phi\;[\Phi:Q\to P]\\
&\Longleftrightarrow\exists\Phi(\forall g\in Q)[\Phi(g)\in P].
\end{align*}
\end{definition}

The relation $\leqw$ is known both as \dff{Mu\v cnik} or \dff{weak} reducibility and $\leqs$ is known as \dff{Medvedev} or \dff{strong} reducibility. Recent authors have tended to favor the terms weak and strong because of the unfortunate fact that the two names begin with the same letter, and we shall follow this lead. The two notions are related by the observation that strong reducibility is the \dff{uniform} version of weak reducibility.

The intuition behind these relations is to regard $P\incl{}^\omega\omega$ as a ``problem" and each $f\in P$ as a ``solution" to the problem. Then $P\leqw Q$ iff every solution to $Q$ computes a solution to $P$ and $P\leqs Q$ iff there is a uniform effective method $\Phi$ to compute from any solution to $Q$ a solution to $P$. In each case $Q$ is at least as hard to solve as $P$. 

For example, the empty set is the unsolvable problem, any set $P$ with a recursive element is an effectively solvable problem, and a singleton set \set f is a problem with a unique solution. More interesting examples that will be of interest below are:
\begin{itemize}
\item for any disjoint sets $A,B\incl\omega$, 
\[{\ssf Sep}(A,B):=\setof C{A\incl C\incl\overline B},\]
the problem of separating $A$ and $B$. Here, $\overline B=\omega\setminus B$ and as usual, we identify a subset of $\omega$ with its characteristic function;
\item for a first-order theory $\cal T$ in a G\"odel-numbered first-order language, 
\[{\ssf CpEx}({\cal T}):=\setof{{\cal U}}{{\cal U}\text{ is a complete extension of }{\cal T}},\]
where we identify ${\cal U}$ with \setof{{\ssf gn}(\phi)}{\phi\in {\cal U}};
\item for a graph ${\cal G}=(\omega,E)$,
\[{\ssf Col}^k({\cal G}):=\setof{f\in\pre\omega k}{f\text{ is a $k$-coloring of 
{$\cal G$}}}.\]
\end{itemize}

For $P\incl\pre\omega\omega$, let 
\[P^{\geqt}:=\setof g{(\exists f\in P)\;f\leqt g},\] 
the \dff{upward Turing closure} of $P$.
Then directly from the definitions we have

\begin{lemma}\label{wk=include}
For any $P,Q\incl\pre\omega\omega$,
\[P\supseteq Q\Implies P\leqs Q\Implies P\leqw Q\Iff P^{\geqt}\supseteq Q.\noproof\]
\end{lemma}

Recall that \dff{Turing degrees} are the equivalence classes of $f\in\pre\omega\omega$ under \dff{Turing equivalence}:
\begin{eqnarray*}
f\equiv_Tg&\dIff &f\leqt g\qand g\leqt f;\\
\degt(f)&:=&\setof g{f\equiv_Tg};\\
\degt(f)\leq\degt(g)&\dIff& f\leqt g;\\
\Dgt&:=&\setof{\degt (f)}{f\in\pre\omega\omega}.
\end{eqnarray*}
Simply because $\leqt$ is a preordering (transitive and reflexive), it follows that $\equiv_T$ is an equivalence relation and the ordering induced on $\Dgt$ is a well-defined partial ordering. Since each of $\leqw$ and $\leqs$ is also a preordering, the same considerations apply to
\begin{eqnarray*}
P\equiv_\bullet Q&\dIff& P\leq_\bullet Q\;\qand\; Q\leq_\bullet P\\
\deg_\bullet (P)&:=&\setof Q{P\equiv_\bullet Q}\\
\deg_\bullet (P)\leq_\bullet\deg_\bullet (Q)&\dIff& P\leq_\bullet Q\\
{\mathbb D}_\bullet& :=&\setof{\deg_\bullet (P)}{P\incl\pre\omega\omega}.
\end{eqnarray*}
where $P,Q\incl\pre\omega\omega$ and $\bullet$ may be either {\ssf w} or {\ssf s}. 

$\Dgt$ is an upper semi-lattice with the join operation
\[\degt(f)\join\degt(g):=\degt(f\join g)\]
where
\[(f\join g)(2x+i):=\cases{f(x),&if $i=0$;\cr \noalign{\smallskip}g(x),&if $i=1$.\cr}\]

$\Dgt$ has smallest element ${\bf 0}_T=\degt(\emptyset)$ but has no largest element and is not a lattice.

In contrast we have
\begin{proposition} \Dgw\ and \Dgs\ are bounded distributive lattices.
\end{proposition}

\begin{proof} For $P,Q\incl\pre\omega\omega$, set
\[\deg_\bullet(P)\join\deg_\bullet(Q):=\deg_\bullet(P\join Q)
\qand
\deg_\bullet(P)\meet\deg_\bullet(Q):=\deg_\bullet(P\meet Q)\]
where
\[P\join Q:=\setof{f\join g}{f\in P\text{ and }g\in Q}
\qand
P\meet Q:=(0)^\frown P\cup(1)^\frown Q.\]
Proofs that these are least upper bound (\dff{join}) and greatest lower bound (\dff{meet}) operations are all simple computations, so we give only two examples. To see that $P\meet Q$ is the greatest lower bound for $P$ and $Q$, note first that the recursive functionals $f\mapsto(0)^\frown f$ and $g\mapsto(1)^\frown g$ witness that $P\meet Q\leqs P$ and $P\meet Q\leqs Q$. If $R$ is any other $\leqs$-lower bound to $P$ and $Q$ witnessed by recursive functionals $\Phi:P\to R$ and $\Psi:Q\to R$, then if $\Theta((0)^\frown f):=\Phi(f)$ and $\Theta((1)^\frown g):=\Psi(g)$, $\Theta:P\meet Q\to R$, so $R\leqs P\meet Q$. The argument for $\leqw$ is similar.

Both lattices have smallest element
\begin{align*}
{\bf 0}_\bullet
&:=\deg_\bullet(\setof f{f \text{ is recursive}})\\
&=\deg_\bullet(P)\text{ for any $P$ with a recursive element,}
\end{align*}
largest element $\infty_\bullet:=\deg_\bullet(\emptyset)$, and satisfy the distributive laws. Again the proof is entirely straightforward; for example, for any $\dgq\in\Dgs$, $\zeros\leq\dgq$, because if $P\in\zeros$ has a recursive element $f$, then the recursive functional with constant value $f$ maps any $Q\in\dgq$ to $P$. \QED
\end{proof}

In the language of problems, members of $P\join Q$ are functions which encode solutions to both of $P$ and $Q$, while members of  $P\meet Q$ are solutions to one or the other of $P$ and $Q$. Note that for weak reducibility we have also 
$P\meet Q\eqw P\cup Q$, but the proof for $\leqs$ above breaks down without the 0-1 coding.

Recall that $\Dgt$ has cardinality $2^{\aleph_0}$ because there are only countably many partial recursive functionals, so Turing degrees are countable sets. However,

\begin{proposition}
\Dgw\ and \Dgs\ have cardinality $2^{2^{\aleph_0}}$ and indeed contain an antichain of that size.
\end{proposition}

\begin{proof}
This follows from two elementary facts from recursion theory and combinatorial set theory:
\begin{enumerate}
\item $(\exists R\incl\pre\omega\omega)\;{\ssf
Card}(R)=2^{\alef0} \qand f\not=g\in R\Implies f\;\vert_T\;g$;
\item $\forall\kappa\;(\exists {\cal
X}\incl\power(\kappa))\; {\ssf Card}({\cal X})=2^\kappa\qand(\forall X,Y\in{\cal X})\;X\;\vert_{\incl}\;Y$.
\end{enumerate}
Then by (i) $P,Q\incl R\Implies [P\leq_\bullet Q\Ifff Q\incl P]$, so by (ii) there exists ${\cal X}\incl\power(R)$ with ${\ssf
Card}({\cal X})=2^{2^{\alef0}}$
such that $P,Q\in{\cal X}\Impliess P\;\vert_\subseteq\;Q\Implies P\vert_\bullet Q$.
\QED
\end{proof}

There are several simple relationships among these (semi-)lattices.

\begin{proposition}\mbox{}\label{embedings}
\begin{enumerate}
\item[\textup{(i)}] There is an upper semi-lattice embedding of \Dgt\  into each of \Dgw\ and \Dgs\ that respects the least element;
\item[\textup{(ii)}]  there is an upper semi-lattice embedding of \Dgw\ into \Dgs\ that respects both least and greatest element;
\item[\textup{(iii)}] \cite[Remark 3.9]{Si1} \Dgw\ is isomorphic as a bounded lattice to the class of upward Turing-closed subsets of \pre\omega\omega\ in the following precise sense:
\[\bigl(\Dgw,\,\leqw,\,\join,\,\meet,\,\zerow,\,\inftyw\bigr)\iso\bigl(\power(\pre\omega\omega)^{\geqt},\,\supseteq,\,\cap,\,\cup,\,\pre\omega\omega,\,\emptyset\bigr).\]
\end{enumerate}
\end{proposition}

\begin{proof}
For $f\in\pre\omega\omega$ and $\dga=\degt(f)$, set $\dga_\bullet:=\deg_\bullet(\set f)$. It is easily checked that the mapping $\dga\mapsto\dga_\bullet$ is  a well-defined mapping of \Dgt\ into ${\mathbb D}_\bullet$ such that
$$\dga\leq\dgb\Iff\dga_\bullet\leq\dgb_\bullet,\quad(\dga\join\dgb)_\bullet=\dga_\bullet\join\dgb_\bullet\qand({\bf 0}_T)_\bullet={\bf 0}_\bullet.$$

For the second part, by Lemma \ref{wk=include}, for any $P$, $Q$,
\[P\leqw Q\Iff P^{\geqt}\supseteq Q^{\geqt}\Iff P^{\geqt}\leqs Q^{\geqt}.\]
Hence the mapping $P\mapsto P^{\geqt}$ induces a well-defined mapping $\dgp\mapsto\dgp^{\ssf s}$ of \Dgw\ into \Dgs\ such that $\dgp\leq\dgq\Iff\dgp^{\ssf s}\leq\dgq^{\ssf s}$. Easily
\[(\zerow)^{\ssf s}=\degs(\pre\omega\omega)=\zeros\qand
(\inftyw)^{\ssf s}=\degs(\emptyset)=\inftys.\]
Furthermore, for any $P$ and $Q$,
\[(P\join Q)^{\geqt}=P^{\geqt}\cap Q^{\geqt}\eqs P^{\geqt}\join Q^{\geqt},\]
so $(\dgp\join\dgq)^{\ssf s}=\dgp^{\ssf s}\join\dgq^{\ssf s}$. Note that this is not a lattice embedding because generally
\[(P\meet Q)^{\geqt}=P^{\geqt}\cup Q^{\geqt}\not\eqs P^{\geqt}\meet Q^{\geqt},\]
so meet is not respected.

For the third part, again by Lemma \ref{wk=include}, the mapping $P\mapsto P^{\geqt}$ induces a well-defined mapping $\dgp\mapsto \dgp^{\geqt}$ of \Dgw\ into $\power(\pre\omega\omega)^{\geqt}$. It is then straightforward to verify that this is the claimed isomorphism.
\QED
\end{proof}

\begin{remark}\label{genmeet}
By the isomorphism of part (iii), \Dgw\ is actually a complete lattice: given a family ${\bf P}\incl\Dgw$, choose ${\cal P}\incl\power(\pre\omega\omega)$ so ${\bf P}=\setof{\degw(P)}{P\in{\cal P}}$. Then
\begin{align*}
\bigmeet{\bf P}&:=\degw\bigl(\bigcup\setof{P^{\geqt}}{P\in{\cal P}}\bigr)
=\degw\bigl(\bigcup\setof P{P\in{\cal P}}\bigr)\quad\text{and}\\
\bigjoin{\bf P}&:=\degw\bigl(\bigcap\setof{P^{\geqt}}{P\in{\cal P}}\bigr)
\end{align*}
are easily seen to be respectively the greatest lower bound and least upper bound of {\bf P}. For countable families, the binary meet operation has the natural generalization
\[\bigmeet_{m\in\omega}P_m:=\bigcup_{m\in\omega}(m)^\frown P_m\]
which has the same weak degree as above. The strong degree is a lower bound for \setof{\degs(P_m)}{m\in\omega} but generally not a greatest lower bound. Nevertheless, this operation will be useful below.
\end{remark}

Of course, there is a huge literature on the general structure of the upper semi-lattice \Dgt; as noted, the lattices \Dgw\ and \Dgs\ have received less attention, but there is a growing body of information. An early contribution is \cite{Dy}; excellent more recent examples are \cite{Sor2} and \cite{Te2}. Most of these results are beyond the scope of this article and we mention here only a few examples as a sample to give the flavor. 

For $f\in\pre\omega\omega$, set 
\[\set f^+_{\ssf w}:=\setof g{f<_Tg}.\]
Easily $\set f\leqw\set f^+_{\ssf w}$, and $\degw(\set f^+_{\ssf w})$ is an immediate successor in the ordering \Dgw\ to $\degw(\set f)$:

\begin{proposition} For any $f\in\pre\omega\omega$ and $P\incl\pre\omega\omega$, $\set f\lew P\Implies\set f^+_{\ssf w}\leqw P$.
\end{proposition}

\begin{proof}
From $\set f\leqw P$ we have $P\incl\setof g{f\leqt g}$. Since also $P\not\leqw\set f$, in fact $P\incl\set f^+_{\ssf w}$, so in particular $\set f^+_{\ssf w}\leqw P$.
\QED
\end{proof}

Matters are slightly more complicated for \Dgs, since it is clearly unreasonable to expect even that $\set f\leqs\set f^+_{\ssf w}$ --- no single reduction procedure can compute $f$ from all $g\in\set f^+_{\ssf w}$. A typical strategy in such cases is to induce the needed uniformity by indexing:

\begin{definition} For $f\in\pre\omega\omega$, 
\[\set f^+:=\setof{(a)^\frown g}{g\in\pre\omega\omega\qand f=\set a^g\qand g\not\leqt f}.\]
\end{definition}

\begin{proposition} For any $f\in\pre\omega\omega$, $\degs(\set f^+)$ is the immediate successor of $\degs(\set f)$ --- explicitly, $\degs(\set f)\leq\degs(\set f)^+$ and for any $P\incl\pre\omega\omega$, $\set f\les P\Implies\set f^+\leqs P$.
\end{proposition}

\begin{proof} Clearly there exists a partial recursive functional $\Phi$ such that for all $g$, 
\[\Phi((a)^\frown g)=\set a^g\]
and thus $\Phi:\set f^+\to\set f$. If $\set f\leqs P$, then for some $\Psi$, $\Psi:P\to\set f$ --- that is, if $a$ is an index for $\Psi$, for all $g\in P$, $\set a^g=f$. If also $P\not\leqs\set f$, then for each $g\in P$, $g\not\leqt f$. Hence if $\set f\les P$, the recursive functional $g\mapsto(a)^\frown g$ witnesses that $\set f^+\leqs P$.
\QED
\end{proof}

It is worth noting that $\set f^+\eqw\set f^+_{\ssf w}$, although $\set f^+\not\eqs\set f^+_{\ssf w}$, so $\set f^+$ serves the same role in \Dgw\ as $\set f^+_{\ssf w}$. 

\begin{corollary} Neither \Dgw\ nor \Dgs\ is densely ordered. \noproof
\end{corollary}

Of course, there are also many counterexamples to density, known as \dff{minimal covers}, in \Dgt, but they are considerably harder to analyze. In fact, a simple extension of these ideas allows us in \Dgw, and with more effort in \Dgs, to completely characterize intervals $(P,Q)_\bullet:=\setof R{P<_\bullet R<_\bullet Q}$ that are empty.

\begin{proposition}[\cite{Dy}]
For any $P<_\bullet Q$, 
\[(P,Q)_\bullet=\emptyset\Iff(\exists f\in P)\bigl[P\equiv_\bullet Q\meet\set f\qand Q\meet\set f^+\equiv_\bullet Q\bigr].\]
\end{proposition}

\begin{proof}
We give the proof for weak degrees; the version for strong degrees is considerably more complicated and may be found in \cite[Theorem 2.5]{Te2}. Easily, if $P\leqw Q$, then for any $f\in P$,
\[P\leqw Q\meet\set f\leqw Q\meet\set f^+\leq Q,\]
and if $Q\not\leqw\set f$,
\[P\leqw Q\meet\set f\lew Q\meet\set f^+\leq Q.\]
Furthermore, for any $R$ such that $Q\meet\set f\leqw R\leqw Q$, either
\[(\exists h\in R)\;h\leqt f\quad\text{so}\quad R\leqw Q\meet\set f,\]
or
\[(\forall h\in R)\;[Q\leqw\set h\qor f<_Th]\quad\text{so}\quad Q\meet\set f^+\leq R.\]
Hence, for any $f\in P$,
\[(P,Q)_{\ssf w}=\setof R{P\lew R\leqw Q\meet\set f\qor Q\meet\set f^+\leqw R\lew Q}.\]
Now the implication $(\Larrow)$ of the statement is immediate and $(\Rarrow)$ follows because if $P\lew Q$, there exists $f\in P$ such that $Q\not\leqw\set f$.
\QED
\end{proof}

\end{section}

\begin{section}{Local degree structures}
Early in the development of the theory of Turing degrees attention was focused on the degrees of (characteristic functions of) special sets that seemed of more interest than arbitrary sets. By far the most intensively studied of these are the degrees of recursively enumerable (r.e.) sets:
\[\Dgpt :=\setof{\degt (A)}{A\incl\omega\text{ is recursively enumerable}}.\]
R.e. sets arise naturally in many contexts, most notably as the sets of G\"odel numbers of theorems of recursively axiomatizable theories in a first-order language.  \Dgpt\ has lattice properties similar to those of \Dgt\ except that it has a largest element

\begin{proposition} 
$\Dgpt$ is a countable bounded upper semi-lattice.
\end{proposition}

\begin{proof} The join operation and ${\bf 0}_T$ are the same as for \Dgt;  the largest element of \Dgpt\ is ${\bf 1}_T= {\bf 0}_T'=\degt(\setof{a\in\omega}{\set a(a)\downarrow})$, the \dff{jump} of $\zerot$.
\QED
\end{proof}

In a posting to the online discussion group FOM in 1999 \cite{SiFOM}, Stephen Simpson suggested that the Mu\v cnik degrees of \pzo\ subsets of \pre\omega2\ might provide an interesting alternative to the r.e.~degrees. In many papers since then, Simpson and other authors have developed this analogy. This development is the central, although not exclusive, focus of the present survey. 

We start with a quick review of the notion of \pzo\ subsets of \pre\omega\omega\ and \pre\omega2; Section 7 has further background. In topological terms they are the sets that are effectively closed. This can be made precise in several equivalent ways, but the one that is most useful here is the following. A \dff{tree} is a subset $T$ of \pre{<\omega}\omega, the set of finite sequences of natural numbers, that is closed under subsequence. A \dff{path} through a tree $T$ is a function $f$ such that all initial segments of $f$ belong to $T$. $[T]$ denotes the set of all paths through $T$.

\begin{definition} $P\incl\pre\omega\omega$ is a \pzo\ set iff $P=[T]$ for some recursive tree $T\incl\pre{<\omega}\omega$. 
\end{definition}

In accord with common usage we sometimes call a \pzo\ subset of \pre\omega 2 a  \pzo\ \dff{class}. Another useful characterization of the \pzo\ sets is as sets definable with only universal number quantifiers over a recursive matrix (see, for example, \cite[Definition III.1.2]{RTH}). \cite{Ce} and \cite{CeRe} are extensive surveys of \pzo\ sets.

There are many naturally arising examples of \pzo\ sets; for further examples and significance see \cite{CeRe}:
\begin{proposition}\label{pzoexamples}
The following are \pzo\ sets.
\begin{itemize}
\item[\textup{(i)}]$\set f$ for any recursive $f\in\pre\omega\omega$ (and many others);
\item[\textup{(ii)}] for disjoint r.e.~$A,B\incl\omega$, ${\ssf Sep}(A,B):=\setof C{A\incl C\incl\overline B}$;
\item[\textup{(iii)}] for an r.e.~graph $\cal G$, ${\ssf Col}^k({\cal G}):=\setof{f\in\pre\omega k}{f\text{ is a $k$-coloring of $\cal G$}}$;
\item[\textup{(iv)}] for a recursively axiomatizable first-order theory $\cal T$, 
\[{\ssf CpEx}({\cal T}):=\setof{{\cal U}}{{\cal U} \text{ is a complete extension of $\cal T$}};\]
\item[\textup{(v)}]${\ssf DNR}:=\setof{f\in\pre\omega\omega}{\forall a[f(a)\not=\set a(a)]}$ and $\dnr k:={\ssf DNR}\cap\pre\omega k$; members of {\ssf DNR} are called \dff{diagonally non-recursive}. \noproof
\end{itemize}
\end{proposition}

For reasons that will be explained in Section 7,  the basic definition is restricted to \pzo\ subsets of \pre\omega2:

\begin{definition}\label{defpzodg}
\begin{align*}
\Dgps &:=\setof{\degs (P)}{P\incl\pre\omega2\text{ is a nonempty \pzo\ class}}\\
\Dgpw &:=\setof{\degw (P)}{P\incl\pre\omega2\text{ is a nonempty \pzo\ class}}
\end{align*}
\end{definition}

\begin{proposition} 
\Dgps\ and \Dgpw\ are countable bounded distributive lattices.
\end{proposition}

\begin{proof}
The meet and join operations and ${\bf 0}_\bullet$ are the same as for ${\mathbb D}_\bullet$. However, establishing the existence of a largest element is much more complicated and forms the content of our first main result.
\QED
\end{proof}

\begin{thma}[{\cite[Theorem 3.20]{Si2}}]
The largest element of ${\mathbb P}_\bullet$ is
\[{\bf 1}_\bullet:=\deg_\bullet(\dnr2)
=\deg_\bullet({\ssf CpEx}({\cal T}))\]
for ${\cal T}={}$Peano Arithmetic or any standard first-order theory of arithmetic or sets.
\end{thma}

As noted in the introduction, proofs of main theorems are postponed to the latter sections of the paper.

Although as discussed in Section 2, none of the structures \Dgt, \Dgw\ or \Dgs\ is densely ordered, one of the landmark results of the theory of r.e.~degrees was Sacks' Density Theorem \cite[Theorem VIII.4.1]{Soa} that \Dgpt\ is a dense ordering. Other authors extended this to show that in any nontrivial interval $[\dga,\dgb]$ every countable partial order embeds preserving existing joins and if \dga\ is low, then every countable partial order embeds preserving existing joins and meets \cite[Exercise VIII.4.10]{Soa}. For \Dgw\ and \Dgs, we have the following results.

\begin{thmb}[{\cite[Theorem 14]{CeHi} and \cite[Theorem 1.1]{Co}}] \Dgps\ is densely ordered; in fact, every finite distributive lattice embeds in any nontrivial interval $[\dgp,\dgq]$.
\end{thmb}

\begin{thmc}[{\cite[Theorem 4.9]{BiSi1}}] \Dgpw\ is downward-densely ordered; in fact, every countable distributive lattice embeds in any nontrivial initial interval $[{\bf 0},\dgq]$, .
\end{thmc}

Both of these results have fairly difficult proofs which will not be included here; the reader is referred to the references. One of the major open questions in the area is

\begin{question} Is \Dgpw\ densely ordered?
\end{question}

Recall from Proposition \ref{embedings} that \Dgt\ embeds in \Dgw\ as a semi-lattice via the mapping 
\[\dga=\degt(A)\mapsto\degw(\set A)=:\dgaw,\]
so it is natural to ask if a similar mapping will provide an embedding of \Dgpt\ into \Dgpw. If $A$ is r.e., $\set A$ is $\Pi^0_2$ but not generally \pzo\ so it does not follow that $\dgaw\in\Dgpw$. It turns out that this problem has a solution that is simple to state, although some work to prove. 

\begin{definition} For any r.e.~set $A$, with $\dga=\degt(A)$, $\dgawstar:=\onew\meet\degw(\set A)$.
\end{definition}

\begin{thmd}[{\cite[Theorem 5.5]{Si3}}]
The mapping $\dga\mapsto\dgawstar$ is a semi-lattice embedding of \Dgpt\ into \Dgpw\ that respects the least and greatest elements --- that is, for all $\dga,\dgb\in\Dgpt$,
\begin{itemize}
\item [\textup{(i)}]$\dgawstar\in\Dgpw$;
\item [\textup{(ii)}]$\dga\leq\dgb\Iff\dgawstar\leq\dgbwstar$;
\item [\textup{(iii)}]$({\bf 0}_T)_{\ssf w}^*=\zerow$\qand $({\bf 1}_T)_{\ssf w}^*=\onew$;
\item [\textup{(iv)}]$(\dga\join\dgb)_{\ssf w}^*=\dgawstar\join\dgbwstar$.
\end{itemize}
\end{thmd}

\begin{question} Is there a similar embedding of \Dgpt\ into \Dgps?
\end{question}

We shall see below that this technique leads to several other results, but the immediate one here is that it suggests regarding \Dgpw\ as a natural extension of \Dgpt. We shall investigate several properties of \Dgpw\ below including ones that may in some sense mark it as a ``nicer" structure than \Dgpt, but only further development will establish the proper relationship between these two structures.
\end{section}

\begin{section}{Implicative Lattices}
One of the original motivations for the study of the Medvedev, and later Mu\v cnik lattices was the hope that they would provide interpretations for intuitionistic or other propositional logics. For our purposes it will be convenient to consider propositional logics with a set {\ssf PS} of \dff{propositional sentences} built in the usual inductive way from a denumerable set \At\ of atomic propositional sentences using the connectives `and' $\land$, `or' $\lor$, `not' $\lnot$ and `implies' $\implies$. A propositional logic is a subset ${\ssf VS}\incl{\ssf PS}$ of \dff{valid sentences} described in some interesting mathematical way, usually via a deductive system or semantical considerations. For example, the \dff{classical  propositional calculus} \CPC\ is the set of tautologies. 

To interpret a logic in a lattice $\mathstr L=(L,\,\leq,\,\meet,\,\join,\,\latz,\,\lato)$, we aim to define a family of mappings $v:{\ssf PS}\to L$ which relate the connectives in natural ways to operations on the lattice and serve to distinguish exactly the valid sentences of the logic. Historically this has been done in the following steps: 
\begin{enumerate}
\item distinguish a unary operation $\latneg$ and a binary operation $\latimpl$ on $\mathstr L$;
\item call $v$ an $\mathstr L$-\dff{valuation} iff for all $\phi,\psi\in{\ssf PS}$
\begin{align*}
v(\phi\land\psi)&=v(\phi)\meet v(\psi),  &v(\phi\lor\psi)&=v(\phi)\join v(\psi),\\
v(\lnot\phi)&=\latneg v(\phi) \qquad\qqand &v(\phi\implies\psi)&=v(\phi)\latimpl v(\psi).
\end{align*}
\item declare a sentence $\phi$ $\mathstr L$-\dff{valid} --- in symbols $\modelof{{\mathstr L}}\phi$ --- iff $v(\phi)=\lato$ for all $\mathstr L$-valuations $v$ and set 
\[\Th({\mathstr L}):=\setof\phi{\modelof{{\mathstr L}}\phi},\]
the \dff{theory of} $\mathstr L$.
\end{enumerate}

The easiest and best-known case is when $\mathstr L$ is a Boolean algebra; $\latneg$ exists by the definition (see Introduction) and $\latimpl$ is defined by $a\latimpl b:=\latneg a\join b$. Then the $\mathstr L$-valid sentences are exactly the tautologies --- that is, $\Th({\mathstr L})=\CPC$.
However, in the present context we have

\begin{proposition} \label{notBA}
\Dgw, \Dgpw, \Dgs, and \Dgps\ are not Boolean algebras.
\end{proposition}

\begin{proof}
For $\dgp,\dgq\in{\mathbb D}_\bullet$, 
\[\dgp\meet\dgq={\bf 0}_\bullet\Iff\dgp={\bf 0}_\bullet\qor\dgq={\bf 0}_\bullet\] 
(${\bf 0}_\bullet$ is \dff{meet-irreducible})
since $P\meet Q$ has a recursive element iff one of $P$, $Q$ has a recursive element. Hence, for ${\bf 0}_\bullet<\dgp<{\bf \infty}_\bullet$,
\[\dgp\meet\dgq={\bf 0}_\bullet\Implies\dgq={\bf 0}_\bullet\Implies\dgp\join\dgq=\dgp\not={\bf \infty}_\bullet\]
and thus \dgp\ has no complement. Since by the same argument ${\bf 0}_\bullet$ is still meet-irreducible in \Dgpw\ and \Dgps, we have also in these lattices that for ${\bf 0}_\bullet<\dgp<{\bf 1}_\bullet$,
\[\dgp\meet\dgq={\bf 0}_\bullet\Implies\dgq={\bf 0}_\bullet\Implies\dgp\join\dgq=\dgp\not={\bf 1}_\bullet.\]
For future reference note that in \Dgw\ and \Dgs, we have also
\[\dgp\join\dgq={\bf \infty}_\bullet\Iff\dgp={\bf \infty}_\bullet\qor\dgq={\bf \infty}_\bullet\]
(${\bf \infty}_\bullet$ is \dff{join-irreducible})
since $P\join Q=\emptyset$ iff one of $P$, $Q=\emptyset$.
Hence, for ${\bf 0}_\bullet<\dgp<{\bf \infty}_\bullet$, we could also argue
\[\dgp\join\dgq={\bf \infty}_\bullet\Implies\dgq={\bf \infty}_\bullet\Implies\dgp\meet\dgq=\dgp\not={\bf 0}_\bullet.\qedhere\]
\end{proof}

Some of the virtues of Boolean algebras are shared by the following class of structures.

\begin{definition} A bounded distributive lattice $\mathstr L$ is an \dff{implicative lattice} (\dff{Heyting algebra}) iff for all $a,b\in L$ there exists $e\in L$ such that 
\[(\forall x\in L)\;[a\meet x\leq b\qIff x\leq e].\]
If such $e$ exists, it is denoted by $a\latimpl b$; 
$a\latimpl b$ is the largest element $x$ such that $a\meet x\leq b$. $\latimpl$ is called an \dff{implication operator}. In any implicative lattice we set $\latneg a:=a\latimpl\latz$.
\end{definition}

It is easy to see that in a Boolean algebra, this notation is consistent with that above: $\latneg a\join b$ is the largest element $x$ such that $a\meet x\leq b$ and $\latneg a\join\latz=\latneg a$. Note that although by definition in any implicative lattice $a\meet\latneg a=\latz$ it is not generally true that $a\join\latneg a=\lato$. $\latneg$ is sometimes called a \dff{pseudo-complement}.

The standard example of an implicative lattice that is not a Boolean algebra is the lattice of open sets of a topological space $\code{T,{\cal O}}$. For $A,B\in{\cal O}$ (open sets),
\[A\leq B\Iff A\incl B,\quad A\join B:=A\cup B,\qand A\meet B:=A\cap B.\]
$\cal O$ is not in general a Boolean algebra, since the set complement of an open set is not generally open, but is easily seen to be an implicative lattice with
\[A\latimpl B:={\ssf Interior}\;\bigl((T\setminus A)\cup B\bigr).\]

\begin{proposition}\label{DgwIsImpl}
\Dgw\ is an implicative lattice.
\end{proposition}

\begin{proof}
For any $P,Q\incl\pre\omega\omega$, set
\begin{align*}
P\latimpl Q&:=\setof{g\in Q}{P\not\leqw\set g}\\
&\phantom{:}=\setof{g\in Q}{(\forall f\in P)\;f\not\leq_Tg}.
\end{align*}
First we show that this is well-defined on weak degrees --- that is,
\[[P\eqw P'\qand Q\eqw Q']\Implies (P\latimpl Q)\eqw (P'\latimpl Q').\]
Assume $P\eqw P'$ and $Q\eqw Q'$ and fix $g'\in (P'\latimpl Q')$. Then $P\not\leqw\set{g'}$ (since $P'\leqw P$), and there exists $g\in Q$ with $g\leqt g'$ (since $Q\leqw Q'$). Then $P\not\leqw\set g$ so $g\in(P\latimpl Q)$, and we conclude that $(P\latimpl Q)\leqw(P'\latimpl Q')$; the converse holds by symmetry.

That $P\latimpl Q$ has the required degree follows from the following equivalences:
\begin{align*}
P\meet X\leqw Q&\Iff(\forall g\in Q)\,[P\cup X\leqw\set g]\\
&\Iff(\forall g\in Q)\,[P\not\leqw\set g\Implies X\leqw\set g]\\
&\Iff(\forall g\in Q)\,[g\in(P\latimpl Q)\Implies X\leqw\set g]\\
&\Iff X\leqw P\latimpl Q.\qedhere
\end{align*}
\end{proof}

Interestingly, however, with considerably more effort we have

\begin{thme}[{\cite[Theorem 5.4]{Sor0}}]
\Dgs\ is not implicative.
\end{thme}

In the next section we shall consider just what $\Th(\Dgw)$ is, but here we address further questions of implicativity of lattices. We have the following recent results.

\begin{thmf}[{\cite[Theorem 3.2]{Te1}}] \Dgps\ is not implicative.
\end{thmf}

\begin{thmg}[{\cite[Theorem 2]{Hig}}]  \Dgpw\ is not implicative.
\end{thmg}

The scope of our inquiry is greatly expanded by the notion of duality:

\begin{definition}
The \dff{dual} of a lattice ${\mathstr L}=\code{L,\,\leq,\,\join,\,\meet}$ is the structure 
\[{\mathstr L}^\dualmarker:=\code{L,\,\dleq,\,\djoin,\,\dmeet},\]
where
\[\djoin{}:={}\meet,\qquad\dmeet{}:={}\join\qqand\dleq{}:={}\geq.\]
\end{definition}

\begin{proposition}\label{basicdual}
\begin{enumerate}
\item[\textup{(i)}] The dual ${\mathstr L}^\dualmarker$ of a (distributive) (bounded) lattice ${\mathstr L}$ is a (distributive) (bounded) lattice;
\item[\textup{(ii)}] if ${\mathstr L}$ is a Boolean algebra then ${\mathstr L}^\dualmarker$ is again a Boolean algebra; in fact,
$\mathstr L$ and ${\mathstr L}^\dualmarker$ are isomorphic via the mapping $a\mapsto\latneg a$;
\item[\textup{(iii)}]in general ${\mathstr L}$ and ${\mathstr L}^\dualmarker$ are not isomorphic. \hfill\noproof
\end{enumerate}
\end{proposition}

\begin{definition}
A bounded distributive lattice $\mathstr L$ is \dff{dual-implicative}\hfill\break
(a \dff{Brouwer algebra}) iff its dual ${\mathstr L}^\dualmarker$ is an implicative lattice. 
Unfolding the definition: $\mathstr L$ is dual-implicative iff for all $a,b\in L$ there exists $e\in L$ such that
\[(\forall x\in L)\;[b\leq a\join x\qIff e\leq x]\]
If such $e$ exists it is denoted by $a\dlatimpl b$; 
$a\dlatimpl b$ is the smallest element $x$ such that $b\leq a\join x$. When $\mathstr L$ is dual-implicative we set $\dlatneg a:=a\dlatimpl\lato$ and write $\Thd({\mathstr L})$ for $\Th({\mathstr L}^\dualmarker)$.
\end{definition}

Note that in a Boolean algebra, $a\dlatimpl b=\latneg a\meet b=b-a$
(\dff{relative complement}) and $\dlatneg a=\latneg a$.

\begin{proposition}[\cite{Med55}]
Both \Dgw\ and \Dgs\ are dual-implicative.
\end{proposition}

\begin{proof}
For \Dgw\ we set 
\[P\dlatimpl Q:=\setof h{(\forall f\in P)(\exists g\in Q)\;g\leq f\oplus h}\]
and first check that this is defined on weak degrees --- in fact,
\[[P\eqw P'\qand Q\eqw Q']\Implies (P\dlatimpl Q)= (P'\dlatimpl Q').\]
Assume the hypothesis and fix $h'\in(P'\dlatimpl Q')$. For any $f\in P$, there exists $f'\in P'$ with $f'\leqt f$ (since $P'\leqw P$). Choose $g'\in Q'$ such that $g'\leqt f'\oplus h'$ (since $h'\in(P'\latimpl Q')$) and $g\leqt g'$ with $g\in Q$ (since $Q\leqw Q'$).  Then
$g\leqt f\oplus h'$ and we conclude that $h'\in (P\dlatimpl Q)$. By symmetry, $(P\dlatimpl Q)= (P'\dlatimpl Q')$.

Now we claim that for any $X$,
\[Q\leqw P\join X\Iff(P\dlatimpl Q)\leqw X.\]
$(\Longrightarrow)$: if $(\forall f\in P)(\forall h\in X)(\exists g\in Q)\;g\leq_Tf\oplus h$, then $X\incl(P\dlatimpl Q)$ so $(P\dlatimpl Q)\leqw X$.

\smallskip
\noindent$(\Longleftarrow)$: if $(\forall k\in X)(\exists h\in (P\dlatimpl Q))\,h\leq_T k$,
then for any $f\in P$ and $k\in X$ there is $g\in Q$ with $g\leq_Tf\oplus h\leq_T f\oplus k$. Thus $Q\leqw P\join X$.

For \Dgs\ we take an indexed version of the same set:
\[P\dlatimpl Q:=\setof{(a)^\frown h}{(\forall f\in P)\;\set a^{f\oplus h}\in Q}\]
Here we need to establish that this is defined on strong degrees (by an indexed version of the above argument) and that
\[Q\leqs P\join X\qIff(P\dlatimpl Q)\leqs X.\]
$(\Longrightarrow)$: Suppose $\set a:P\join X\to Q$. Then
$(\forall f\in P)(\forall h\in X)\;\set a^{f\oplus h}\in Q$, so
$(\forall h\in X)\;(a)^\frown h\in(P\dlatimpl Q)$ and thus
$h\mapsto (a)^\frown h$ witnesses that $(P\dlatimpl Q)\leqs X$.

\smallskip
\noindent$(\Longleftarrow)$: Suppose $k\mapsto(a_k)^\frown h_k$ witnesses that $(P\dlatimpl Q)\leqs X$. Then $f\oplus k\mapsto \set{a_k}^{f\oplus h_k}$ witnesses that $Q\leqs P\join X$.
\QED
\end{proof}

For the local versions we know only

\begin{thmh}[{\cite[Theorem 1]{Si4}}]
\Dgpw\ is not dual-implicative
\end{thmh}

\begin{question}
Is \Dgps\ dual-implicative?
\end{question}

The following table summarizes the current state of knowledge:
\begin{center}
\begin{tabular}{ l | c  c }
    & Implicative & Dual-implicative \\[5pt]
    \hline\\
  \Dgs& No & Yes \\[10pt]
  \Dgw& Yes & Yes \\[10pt]
  \Dgps&No&{??}\\[10pt]
  \Dgpw&No&No\\
\end{tabular}
\end{center}

Several other recent papers study other aspects of these lattices. \cite{LSS} considers (among others)
\begin{align*}
{\mathbb D}_{\ssf s}^{\ssf cl}
&:=\setof{\degs(P)}{P\incl\pre\omega\omega\text{ and }P\text{ is closed}}\\
{\mathbb D}_{\ssf s}^{\ssf de}
&:=\setof{\degs(P)}{P\incl\pre\omega\omega\text{ and }P\text{ is dense in }\pre\omega\omega}\\
{\mathbb D}_{\ssf s}^{\ssf di}
&:=\setof{\degs(P)}{P\incl\pre\omega\omega\text{ and }P\text{ is discrete}}
\end{align*}
Each of these forms a sublattice of \Dgs. ${\mathbb D}_{\ssf s}^{\ssf cl}$ is shown not to be dual-implicative, but this question is left open for the other structures. Many ideals and filters of \Dgs\ are also considered as are the corresponding sets of degrees of subsets of \pre\omega2.

\cite{LNS} establishes that the first-order theories of both \Dgs\ and \Dgw\ are as complicated as possible in that they are recursively isomorphic to the third-order theory of the natural numbers, or equivalently the second-order theory of the real numbers. \cite{BCRW} introduces sublattices ${\mathbb P}_{\ssf s}^K$ and ${\mathbb P}_{\ssf w}^K$ (the K-{\it trivial} degrees) of \Dgps\ and \Dgpw\ and shows that they are unbounded and that ${\mathbb P}_{\ssf s}^K$ is densely ordered.
\end{section}

\begin{section}{Lattice Logics}
For any implicative lattice $\mathstr L$, $\Th({\mathstr L})$ is a set of propositional sentences that we call the \dff{theory of} $\mathstr L$. To justify this terminology we have directly from the definitions
 
 \begin{proposition}\label{closedMP}
 For any implicative lattice $\mathstr L$, 
 for all $a,b\in L$, $a\leq b\Iff a\latimpl b=\lato$;
hence $\Th({\mathstr L})$ is closed under \dff{modus ponens}; in fact, for all propositional sentences $\phi$ and  $\mathstr L$-valuations $v$,
 \[\text{if\quad }v(\phi)=\lato \qand v(\phi)\latimpl v(\psi)=\lato,
\quad\text{then}\quad v(\psi)=\lato.\noproof\]
\end{proposition}

We first note that $\Th({\mathstr L})$ will generally not coincide with \dff{classical propositional calculus}, the set of tautologies denoted by \CPC: 

\begin{proposition}
For any implicative lattice $\mathstr L$ and any propositional sentence $\phi$,  $\lnot(\phi\land\lnot\phi)\in\Th({\mathstr L})$, but generally $(\phi\lor\lnot\phi)\notin\Th({\mathstr L})$
\end{proposition}

\begin{proof}
Since for any $a\in L$, $a\meet\latneg a=\latz$, for any $\mathstr L$-valuation $v$,
\[v(\phi\land\lnot\phi)=\latz\quad\text{so}\quad v(\lnot(\phi\land\lnot\phi))=\latneg\latz=\lato\]
However, generally $v(\phi\lor\lnot\phi)=v(\phi)\join \latneg v(\phi)\not=\lato$.
\QED
\end{proof}

To begin to characterize $\Th({\mathstr L})$ for various lattices, we need to remind the reader of some facts of propositional logic.  

\begin{definition} 
\begin{itemize}
\item[\textup{(i)}] \dff{Intuitionistic Propositional Calculus} \IPC\ is the set of propositional sentences generated by {\it modus ponens} from the following axiom schemas:
\begin{itemize}
\item[]$\phi\implies(\psi\implies\phi)$
\item[]$(\phi\implies\psi)\implies\bigl[\bigl(\phi\implies(\psi\implies\theta)\bigr)\implies(\phi\implies\theta)\bigr]$
\item[]$(\phi\implies\theta)\implies\bigl[(\psi\implies\theta)\implies\bigl((\phi\lor\psi)\implies\theta\bigr)\bigr]$
\item[]$\phi\implies(\psi\implies\phi\land\psi)$
\item[]$\phi\implies(\phi\lor\psi)\qquad\qquad\qquad(\phi\land\psi)\implies\phi$
\item[]$\psi\implies(\phi\lor\psi)\qquad\qquad\qquad(\phi\land\psi)\implies\psi$
\item[]$(\phi\implies\psi)\implies\bigl[(\phi\implies\lnot\psi)\implies\lnot\phi\bigr]$
\item[]$\lnot\phi\implies(\phi\implies\psi)$.
\end{itemize}
\item[\textup{(ii)}] \WEM, the logic of the \dff{weak excluded middle}, is the set of sentences generated by these schemas together with the schema $\lnot\phi\lor\lnot\lnot\phi$.
\end{itemize}
\end{definition}

The following is well-known; for general information on intuitionistic logic and its extensions, we refer the reader to \cite{SEP}

\begin{lemma}\label{IPCvsCPC}
\begin{itemize}
\item[\textup{(i)}] $\IPC\subset\WEM\subset\CPC$;
\item[\textup{(ii)}] \CPC\ is  the set of sentences generated by \IPC\ together with the schema $\lnot\lnot\phi\implies\phi$. \noproof
\end{itemize}
\end{lemma}

\begin{proposition}\label{IPCinclThL}
For any implicative lattice $\mathstr L$, 
\begin{itemize}
\item[\textup{(i)}] $\Th(\mathstr L)$ is an \dff{intermediate logic} --- that is, $\IPC\incl\Th(\mathstr L)\incl\CPC$;
\item[\textup{(ii)}] $\Th(\mathstr L)=\CPC$ iff $\mathstr L$ is a Boolean algebra.
\end{itemize}
\end{proposition}

\begin{proof}
That $\IPC\incl\Th(\mathstr L)$ is straightforward to check from the axioms and Proposition \ref{closedMP}; we give three examples.

It is immediate that $a\latimpl b=\lato\Iff a\leq b$, so
\[\modelof{{\mathstr L}}\phi\implies\psi\qIff \forall v[v(\phi)\leq v(\psi)].\]
Since $v(\phi)\leq v(\phi)\join v(\psi),\quad \modelof{{\mathstr L}}\phi\implies (\phi\lor\psi)$. Next, since
\begin{align*}
v(\psi\implies\phi\land\psi)
&=v(\psi)\latimpl(v(\phi)\meet v(\psi))\\
&=\hbox{largest }x\;[v(\psi)\meet x\leq v(\phi)\meet v(\psi)]
\end{align*}
and $v(\phi)$ is such an $x$, we have
$v(\phi)\leq v(\psi\implies\phi\land\psi)$
\quad so\quad 
$v(\phi\implies(\psi\implies\phi\land\psi))=\lato$.

Finally, since $a\latimpl(b\latimpl c)=\lato\Iff a\leq(b\latimpl c)\Iff a\meet b\leq c$
\begin{align*}
\modelof{{\mathstr L}}(\phi\implies\psi)&\implies\bigl[(\phi\implies\lnot\psi)\implies\lnot\phi\bigr]\\
&\Iff\forall v[v(\phi\implies\psi)\meet v(\phi\implies\lnot\psi)\leq{\latneg v(\phi)}]\\
\noalign{\medskip}
&\Iff\forall v[v(\phi)\meet v(\phi\implies\psi)\meet v(\phi\implies\lnot\psi)=\latz].
\end{align*}
But $v(\phi)\meet v(\phi\implies\psi)\leq v(\psi)\qand $
$v(\phi)\meet v(\phi\implies\lnot\psi)\leq{\latneg v(\psi)}$
so
\[v(\phi)\meet v(\phi\implies\psi)\meet v(\phi\implies\lnot\psi)\leq v(\psi)\meet\latneg v(\psi)=\latz.\]
The second inclusion of (i) follows from the observation that among the $\Th(\mathstr L)$-valuations are ones that take on only values \latz\ and \lato; easily a sentence $\phi$ is a tautology iff it is assigned value \lato\ by all such valuations. Part (ii) is immediate from (i) and (ii) of the preceding lemma.
\QED
\end{proof}

From the table at the end of the preceding section we have at least three structures of Medvedev or Mu\v cnik degrees that are implicative lattices and therefore have well-defined propositional theories: 
\[\Thd(\Dgs), \quad\Th(\Dgw)\qand\Thd(\Dgw).\]
By the preceding proposition, Proposition \ref{notBA} and the observation that a lattice is a Boolean algebra iff its dual is one (Proposition \ref{basicdual} (ii)), we know that these theories are proper subsets of \CPC. It is also easy to see that they are proper supersets of \IPC. 

\begin{proposition} In each of $\Dgs ^\dualmarker$, \Dgw, and $\Dgw^\dualmarker$, $\latz$ is meet-irreducible.
\end{proposition}

\begin{proof} In the proof of Proposition \ref{notBA} we noted that $\latz$ is meet-irreducible in \Dgw. That the same is true in $\Dgs^\dualmarker$ and $\Dgw^\dualmarker$ follows from the fact, also noted in that proof, that in \Dgs\ and \Dgw, $\lato={\bf \infty}$ is join-irreducible.
\QED
\end{proof}

\begin{proposition}\label{WEMinclThL}
For any implicative lattice $\mathstr L$, if $\latz$ is meet-irreducible, then $\WEM\incl\Th(\mathstr L)$. 
\end{proposition}

\begin{proof}
Under the hypothesis, for any $a\in L$,
\begin{align*}
\latneg a=\hbox{largest }x[a\meet x=\latz]
&=\cases{\lato,&if $a=\latz$;\cr\latz,&if $a\not=\latz$;\cr}\\
\noalign{\smallskip}
\latneg \latneg a
&=\cases{\latz,&if $a=\latz$;\cr\lato,&if $a\not=\latz$.\cr}
\end{align*}
Hence, for any $a\in{\mathstr L}$,
\[\latneg a=\lato\qor\latneg \latneg a=\lato\quad\hbox{so}\quad\latneg a\join\latneg \latneg a=\lato\]
and for any sentence $\phi$ and $\mathstr L$-valuation $v$, $v(\lnot\phi\lor\lnot\lnot\phi)=\lato$.
\QED
\end{proof}

\begin{corollary}\label{latzimpliesWEM}
Each of $\Thd(\Dgs)$, $\Th(\Dgw)$, and $\Thd(\Dgw)$ includes \WEM. \noproof
\end{corollary}

In fact,

\begin{thmi}[\cite{Med55,Jan,Sor1,Sor3,SorTe1}]
$\Th(\Dgw)=\WEM=\Thd(\Dgw)=\Thd(\Dgs)$.
\end{thmi}

This result was seen by some as a disappointment in view of the early hopes that some of these degree structures would serve as models for \IPC. This motivated the study of theories of sublattices described as follows.

\begin{definition}\label{segments}
For any lattice \l\ and any $d,e\in L$,
\[\l[d,e]:=\bigl(L[d,e],\leq,\join,\meet\bigr)\]
where
\[L[d,e]:=\setof{a\in L}{d\leq a\leq e}\]
and $\leq$, $\join$ and $\meet$ are the obvious restrictions to $L[d,e]$.
\end{definition}

\begin{lemma}
For any lattice \l\ and $d,e\in L$, $\l[d,e]$ is a lattice, and if $d<e$ and \l\ is (dual-) implicative, so is $\l[d,e]$ with the (dual-) implication
\[a\latimpl_eb:=(a\latimpl b)\meet e\qqand
a\dlatimpl_db:=d\join(a\dlatimpl b). \noproof\]

\end{lemma}

Many authors contributed to the following theorem; references, further results and related open questions can be found in the cited work. (An error in this article will be corrected in an addendum to appear).

\begin{thmj}[\cite{SorTe}] 
There exist strong degrees \dgr, \dgs, {{\bf t}} and $\functionof{\dgu_n}{n\in\omega}$ such that\begin{enumerate}
\item[\textup{(i)}] $\Thd(\Dgs[\zeros,\dgr])=\IPC$;
\item[\textup{(ii)}] $\Thd(\Dgs[\zeros,\dgs])=\WEM$;
\item[\textup{(iii)}] $\Thd(\Dgs[\zeros,{{\bf t}}])=\CPC$;
\item[\textup{(iv)}] for all $m\not=n\in\omega$, 
$\Thd(\Dgs[\zeros,\dgu_m])\not=\Thd(\Dgs[\zeros,\dgu_n])$.
\end{enumerate}
\end{thmj}

Sorbi and Terwijn have recently announced in \cite{SorTe1} analogous results for \Dgw\ and $\Dgw^\dualmarker$:

\begin{proposition}
There exist weak degrees \dgr\ and \dgs\ such that
\[\Thd(\Dgw[\zerow,\dgr])=\IPC=\Th(\Dgw[\dgs,\infty_{\ssf w}]).\ \noproof\]
\end{proposition}

\begin{question} What are $\Thd(\Dgs[\zeros,\ones])$ and $\Thd(\Dgw[\zerow,\onew])$?
\end{question}
\end{section}

\begin{section}{Special Degrees}
In July and August 1999 there was a somewhat contentious debate on the FOM discussion group (see \cite{SiFOM} or more generally the FOM archives at \cite{FOM}) concerning, among other topics, the assertion that the theory of the recursively enumerable degrees \Dgpt\ has a history rather different from that of many, if not most mathematical theories. In the beginning only two r.e.~degrees, \latz\ and \lato\  were known, and many years elapsed between the formulation of the basic definitions and the result that showed that the theory was non-trivial --- the Friedberg-Mu\v cnik Theorem establishing that other r.e.~degrees exist. 

Stephen Simpson, Harvey Friedman and others noted that most theories are motivated by the existence of a plethora of examples with the abstract theory serving to organize and explain the examples; Group Theory is a prime example of this process. Whatever the merits of this distinction, it is an odd feature of the theory of r.e.~degrees that even as a mature theory, there are still few examples of r.e.~degrees that arise naturally in contexts beyond the theory itself. Simpson in \cite{SiFOM} expressed what at that time was merely a hope that the theory of (now called) weak or Mu\v cnik degrees of \pzo\ classes --- \Dgpw\ in the notation here --- might prove to be a richer theory in this regard. 

In the years since, several people, especially Simpson, have developed this idea. We have already covered above several of these results including that there is a semi-lattice embedding of \Dgpt\ into \Dgpw\ via the mapping $\dga\mapsto\dgawstar$ (Theorem D). As we noted at the end of Section 3 this suggests that we regard \Dgpw\ as an extension of \Dgpt. In this section we mention some examples fulfilling Simpson's hopes for the existence of ``natural" members of \Dgpw\ strictly between $\zerow$ and $\onew$.

Of course, from the perspective of this survey, it would also be natural to ask if \Dgps\ has some of the same properties. Indeed, \Dgps\ is known to be a dense ordering (Theorem B), one of the signature properties of \Dgpt. The density of \Dgpw\ is a major open question as is the existence of an embedding of \Dgpt\ into \Dgps, so a full understanding of the relationships must await further research. Indeed, by Remark \ref{truthtable} below, it may be more reasonable to expect an embedding of the r.e.~truth-table degrees into \Dgps.

In the following definition, a $\Pi^0_1[A]$ class is the set of paths through an $A$-recursive tree and $\zerot^{(n)}$ is the Turing degree of the $n$-th iterated jump of the empty (or any recursive) set.

In Proposition \ref{pzoexamples} we introduced the \pzo\ sets
\[{\ssf DNR}:=\setof{f\in\pre\omega\omega}{\forall a[f(a)\not=\set a(a)]}
 \qand \dnr k:={\ssf DNR}\cap\pre\omega k.\]
Let 
\[\dgd:=\degw({\ssf DNR}), \quad {\bf d}_{k,{\ssf w}}:=\degw(\dnr k),\qand 
 	{\bf d}_{k,{\ssf s}}:=\degs(\dnr k).\]

\begin{definition}
\begin{enumerate}
\item[\textup{(i)}] $\mu$ denotes the usual probability measure on \pre\omega2 defined by
\[\mu(\setof f{f(0)=i_0,\ldots,f(n-1)=i_{n-1}})=2^{-n};\]
\item[\textup{(ii)}] for $A\incl\omega$, $P\incl\pre\omega2$ is \dff{$A$-full} iff $P=\bigcup_{n\in\omega}P_n$ for some $A$-recursive sequence of $\Pi^0_1[A]$ classes $P_0\incl P_1\incl\ldots$ such that for all $n$, $\mu(P_n)\geq 1-2^{-n}$;
\item[\textup{(iii)}] ${\ssf R}^A:=\bigcap\setof P{P\hbox{ is $A$-full}}$, the set of \dff{$A$-random} reals;
\item[\textup{(iv)}] $\random n:={\ssf R}^{\zerot^{(n-1)}}$;
\item[\textup{(v)}] ${\bf r}_n:=\degw(\random n);\qquad {\bf r}_n^*:=\onew\meet{\bf r}_n$.
\end{enumerate}
\end{definition}

For much more information on algorithmic randomness see \cite{DoHi} or for a briefer summary \cite{DoRe}.

\begin{thmk}[{\cite[Theorems 4.3 and 5.6]{Si3}} and {\cite[Theorem 8.10]{Si1}}]
\mbox{}
\begin{enumerate}
\item[\textup{(i)}] \dgd, ${\bf r}_1$ and ${\bf r}_2^*\in\Dgpw$;
\item[\textup{(ii)}] $\zerow<\dgd<{\bf r}_1<{\bf r}_2^*<\onew$;
\item[\textup{(iii)}] these degrees are incomparable with all $\dgawstar$ for $\dga\in\Dgt\not=\zerot,\onet$;
\item[\textup{(iv)}] $\dgr_1$ is the largest element of \Dgpw\ that contains a \pzo\ set of positive measure.
\end{enumerate}
\end{thmk}
 
\begin{thml}[{\cite[Theorems 5 and 6]{Jo1}} and {\cite[Corollary 2.11]{CeHi2}}]
 For all $2\leq\ell<k$, 
 \begin{enumerate}
\item[\textup{(i)}] ${\bf d}_{k,{\ssf w}}\in\Dgpw$ and ${\bf d}_{k,{\ssf s}}\in\Dgps$;
\item[\textup{(ii)}] ${\bf d}_{k,{\ssf w}}={\bf d}_{\ell,{\ssf w}}=\onew$;
\item[\textup{(iii)}] ${\bf d}_{k,{\ssf s}}<{\bf d}_{\ell,{\ssf s}}\leq{\bf d}_{2,{\ssf s}}=\ones$.
 \end{enumerate}
 \end{thml}
 
 See also \cite{CeHi2} for more examples of strong degrees.
 \end{section}

\begin{section}{\pzo\ Sets and Classes}
In preparation for the proofs to follow, we collect here some general information on \pzo\ sets. This is a large topic with an extensive literature, and we shall discuss here only those aspects of the theory required in the sequel. Excellent general references are \cite{Ce} and \cite{CeRe}. 

In Definition \ref{defpzodg} we focussed attention on \pzo\ classes --- that is, \pzo\ subsets of $\pre\omega2$:
\begin{align*}
{\mathbb P}_\bullet&:=\setof{\deg_\bullet(P)}{P\incl\pre\omega2\text{ is a nonempty \pzo\ class}}.\\
\noalign{\noindent\text{Other natural possibilities that come to mind are}}
{\mathbb P}^k_\bullet&:=\setof{\deg_\bullet(P)}{P\incl\pre\omega k\text{ is a nonempty \pzo\ set}}\text{\quad and}\\
{\mathbb P}^\omega_\bullet&:=\setof{\deg_\bullet(P)}{P\incl\pre\omega\omega\text{ is a nonempty \pzo\ set}},\\
\noalign{\noindent\text{but it turns out that the key alternative is}\smallskip}
{\mathbb P}^{\ssf bd}_\bullet&:=\setof{\deg_\bullet(P)}{P\incl\pre\omega\omega\text{ is a nonempty recursively bounded \pzo\ set}}.
\end{align*}
Here we use the following notions:

\begin{definition}\label{recbd}
For and $f,g\in\pre\omega\omega$, $\sigma\in\pre{<\omega}\omega$ and $P\incl\pre\omega\omega$,
\begin{itemize}
\item[\textup{(i)}]$f$ is \dff{ bounded by} $g$ iff $\forall m[f(m)\leq g(m)]$;
\[\pre\omega g:=\setof f{f\text{ is bounded by }g};\]
\item[\textup{(ii)}]$\sigma$ is \dff{bounded by} $g$ iff $(\forall m<|\sigma|)[\sigma(m)\leq g(m)]$;
\begin{align*}
\pre{<\omega} g&:=\setof {\sigma\in\pre{<\omega}\omega}{\sigma\text{ is bounded by }g};\\
\pre{m} g&:=\setof {\sigma\in\pre{m}\omega}{\sigma\text{ is bounded by }g};
\end{align*}
\item[\textup{(iii)}]$f$ is \dff{recursively bounded} iff $f\in\pre\omega g$ for some recursive function $g$;
\item[\textup{(iv)}]$P$ is \dff{recursively bounded} iff $P\incl\pre\omega g$ for some recursive function $g$.
\end{itemize}
\end{definition}

Obviously, ${\mathbb P}_\bullet\incl{\mathbb P}^k_\bullet\incl{\mathbb P}^{k+1}_\bullet\incl{\mathbb P}^{\ssf bd}_\bullet\incl{\mathbb P}^{\omega}_\bullet$. Any \pzo\ set $P\incl\pre\omega\omega$ maps naturally onto a set $P^*\incl\pre\omega2$ via the recursive functional  $f\mapsto f^*$, where
\begin{align*}
f^*(\code{m,n})&:=\cases{1,&if $f(m)=n$;\cr0,&otherwise;}\\
P^*&:=\setof{f^*}{f\in P};
\end{align*}
in such a way that $P\equiv_\bullet P^*$. However, in general $P^*$ may not be \pzo. The following two lemmas clarify when this is the case and have other interesting consequences.

\begin{lemma}\label{Konig}
For any $h\in\pre\omega\omega$ and any tree $U\incl\pre{<\omega}h$,
\begin{align*}
(\exists f\in\pre\omega h)\forall m\,[f\restrict m\in U]
&\qIff\forall m(\exists\sigma\in\pre mh)\;\sigma\in U;\tag{i}\\
(\forall f\in\pre\omega h)\exists m\,[f\restrict m\notin U]
&\qIff\exists m(\forall\sigma\in\pre mh)\;\sigma\notin U.\tag{ii}
\end{align*}
\end{lemma}

\begin{proof}This result is generally known as the K\"onig Infinity Lemma and is an old and familiar fact; however, for completeness we give here a proof. It is also an expression of the topological compactness of the space \pre\omega h together with the fact that \pzo\ sets are closed sets. We prove (i); (ii) follows immediately. The implication (\Rarrow) is clear; assume the right-hand side. Call a sequence $\sigma$ \dff{infinitely extendable} iff \setof{\tau\in U}{\sigma\incl\tau} is infinite. The assumption is exactly that $U$ is infinite --- that is, that $\emptyset$ is infinitely extendable. Furthermore, since
\[\sigma\incl\tau\qIff\sigma=\tau\qor\bigl(\exists n<h(|\sigma|)\bigr)\;\sigma^\frown(n)\incl\tau,\]
if $\sigma$ is infinitely extendable, so is $\sigma^\frown(n)$ for some $n<h(|\sigma|)$. Hence there is a unique function $f$ such that for all $m$,
\[f(m)=\text{least }n<h(m)\;[(f\restrict m)^\frown(n)\text{ is infinitely extendable}].\]
Clearly $f$ witnesses the left-hand side.
\QED
\end{proof}

\begin{proposition}\label{compactness}
For any recursive function $h$, any \pzo\ class $P\incl\pre\omega h$ and any partial recursive functional $\Phi:P\to\pre\omega\omega$,
\begin{enumerate}
\item $\Phi(P)\in\pzo$;
\item $\Phi(P)$ is recursively bounded;
\item there exists a total recursive functional $\bar\Phi:\pre\omega\omega\to\pre\omega\omega$ extending $\Phi\restrict P$.
\end{enumerate}
\end{proposition}

\begin{proof}
Fix a recursive tree $T$ such that $P=[T]$. For (i), by the preceding lemma,
\begin{align*}
g\in\Phi(P)
&\Iff (\exists f\in\pre\omega h)\,[f\in P\qand\Phi(f)=g]\\
&\Iff  (\exists f\in\pre\omega h)\forall m[f\restrict m\in T\qand\Phi(f\restrict m)\incl g]\\
&\Iff\forall m(\exists\sigma\in\pre mh)\,[\sigma\in T\qand\Phi(\sigma)\incl g].
\end{align*}
The quantifier $(\exists\sigma\in\pre mh)$ is a bounded quantifier --- formally there is a recursively (because $h$ is recursive) calculable upper bound for the codes of all finite sequences $\sigma$ of length $m$ with all $\sigma(i)<h(i)$ --- and therefore does not increase the complexity of the expression. Hence this gives a characteriztion of $\Phi(P)$ with only one universal quantifier and thus shows that $\Phi(P)\in\pzo$.
For (ii), note that since $\Phi(f)$ is a total function for all $f\in P$,
\[\forall n(\forall f\in\pre\omega h)\exists m\;\bigl[f\restrict m\notin T\qor\Phi(f\restrict m)(n)\downarrow\bigr],\]
so by (ii) of the preceding lemma,
\[\forall n\exists m(\forall\sigma\in\pre mh)\,[\sigma\in T\Implies\Phi(\sigma)(n)\downarrow].\]
For each $n$, let $m_n$ denote the least such $m$. Then for all $f\in P$,
\[\Phi(f)(n)\leq\max\setof{\Phi(\sigma)(n)}{\sigma\in\pre{(m_n)}2\cap T}.\]
For (iii), set for all $f\in\pre\omega\omega$,
\[\bar\Phi(f)(n):=\cases{\Phi(f\restrict m_n)(n),&if $f\restrict m_n\in T$;\cr
0,&otherwise. \qedhere\cr}\]
\end{proof}

\begin{remark}\label{truthtable}
$\bar\Phi$ is called a \dff{truth-table functional} because each value \hfill\break 
$\bar\Phi(f)(n)$ can be recursively calculated from an initial segment $f\restrict m_n$ of $f$ of recursively computable length not depending on $f$. Hence for \pzo\ classes $P,Q\incl\pre\omega2$, when $P\leqs Q$, each element $g$ of $Q$ actually truth-table (not merely Turing) computes an element $f$ of $P$ (see \cite[Exercise V.2.12]{Soa} or \cite[Exercises 8.1.37-41]{Hi}) written $f\leq_{\ssf tt}g$.
\end{remark}

\begin{corollary}
For all $k\geq2$, ${\mathbb P}_\bullet={\mathbb P}^k_\bullet={\mathbb P}^{\ssf bd}_\bullet.$
\end{corollary}

\begin{proof}
In the discussion above, the map $f\mapsto f^*$ is recursive, so for any recursively bounded \pzo\ set $P$, the image $P^*$ of $P$ is also \pzo. Since by definition $P^*\incl\pre\omega 2$, this establishes that ${\mathbb P}^{\ssf bd}_\bullet\incl{\mathbb P}_\bullet$.
\QED
\end{proof}

We may thus focus on ${\mathbb P}_\bullet$; we shall see below that in several ways these classes are ``better behaved" than ${\mathbb P}^\omega_\bullet$, which to date have been relatively little studied. 

Note for comparison with (ii) of the preceding proposition that

\begin{proposition}\label{inversepzo}
For any \pzo\ sets $P,Q\incl\pre\omega\omega$ and any partial recursive functional $\Phi:P\to\pre\omega\omega$, $P\cap\Phi^{-1}(Q)\in\pzo$.
\end{proposition}

\begin{proof}
Let $U$ be a recursive tree such that $Q=[U]$. Then
\begin{align*}
P\cap\Phi^{-1}(Q)&=\setof{f\in P}{\Phi(f)\in Q}\\
&=\setof{f\in P}{\forall n\;\Phi(f)\restrict n\in U}\\
&=\setof{f\in P}{\forall m\;\Phi(f\restrict m)\in U}
\end{align*}
which gives the result.
\QED
\end{proof}

A \pzo\ set $P\incl\pre\omega\omega$ is by definition of the form $P=[T]$ for some recursive tree $T\incl\pre{<\omega}\omega$, but $T$ is generally not uniquely determined, since when $\sigma$ is such that $P$ has no element $f\supseteq\sigma$, the inclusion or omission from $T$ of $\tau\supseteq\sigma$ will have no effect on $[T]$. Furthermore, by a standard quantifier calculation, if $T$ non-recursive but is \pzo\ (co-r.e) as a set of (codes for) finite sequences, $[T]$ is still a \pzo\ set.  This leads to a natural effective enumeration \functionof{R_a}{a\in\omega} of all \pzo\ classes based on the standard enumeration of the r.e.\ sets:
\[T_a:=\setof{\sigma\in\pre{<\omega}{\omega}}{(\forall\tau\incl\sigma)\;\tau\notin W_a}\qand R_a:=[T_a].\]
It is also useful on occasion to refine this enumeration by setting
\[T_{a,s}:=\setof{\sigma\in\pre{<\omega}{\omega}}{(\forall\tau\incl\sigma)\;\tau\notin W_{a,s}}\qand R_{a,s}:=[T_{a,s}],\]
so 
\begin{align*}
\pre{<\omega}2&=T_{a,0}\supseteq\cdots\supseteq T_{a,s}\supseteq T_{a,s+1}\supseteq\cdots,\qquad &T_a&=\bigcap_{s\in\omega}T_{a,s},\\
\pre{\omega}2&=R_{a,0}\supseteq\cdots\supseteq R_{a,s}\supseteq R_{a,s+1}\supseteq\cdots
\qqand &R_a&=\bigcap_{s\in\omega}R_{a,s}.
\end{align*}
Topologically, each $R_{a,s}$ is a clopen set.
There is also a canonical tree associated with each \pzo\ set $P\incl\pre\omega\omega$:
\[T_P:=\setof\sigma{(\exists f\in P)\;\sigma\incl f}.\]
Clearly, $P=[T_P]$ and $T_P$ is distinguished by the property of having no ``leaves" or ``dead ends" --- $\sigma\in T_P$ such that no extension $\sigma^\frown(i)\in T_P$.  $T_P$ is not generally recursive or even \pzo, but we have

\begin{proposition}\label{TPispzo}
For any $P\incl\pre\omega2$ (or any \pre\omega h for a recursive $h$), $T_P$ is \pzo.
\end{proposition}

\begin{proof}
Using the Lemma \ref{Konig}, if $P=[T]$ with $T$ recursive, we have
\begin{align*}
T_P&=\setof\sigma{(\exists f\in P)\;\sigma\incl f}\\
&=\setof\sigma{(\exists f\in\pre\omega 2)(\forall m\geq n)\;\sigma\incl f\restrict m\in T}\\
&=\setof\sigma{\forall m(\exists \tau\in\pre m2)[m\geq|\sigma|\Implies\sigma\incl\tau\in T]},
\end{align*}
which establishes that $T_P$ is \pzo.
\QED
\end{proof}

An element $f\in P\incl\pre\omega\omega$ is called \dff{isolated} iff for some $\sigma\incl f$ there is no $g\not=f$ such that $\sigma\incl g\in P$. A set with no isolated elements is called \dff{perfect}, and a simple standard argument shows that and perfect set has the cardinality of the continuum and is thus uncountable. 

\begin{proposition}
Any isolated member of a \pzo\ class $P\incl\pre\omega2$ is recursive.
\end{proposition}

\begin{proof}
Suppose that $f$ is the unique function such that $\sigma\incl f\in P$. Then for any $\tau$ such that $\sigma\incl\tau\in\pre\omega2$,
\[\tau\incl f\Iff\tau\in T_P\Iff(\forall \upsilon\in\pre{|\tau|}2)\;[\tau\not=\upsilon\Implies\upsilon\notin T_P].\]
It follows from the preceding proposition that \setof\tau{\tau\incl f} is r.e.~and hence that $f$ is recursive.
\QED
\end{proof}

\begin{corollary}\label{nonzeroperfect}
Any \pzo\ class $P\incl\pre\omega2$ with no recursive element is uncountable. \noproof
\end{corollary}

In particular, this implies that the only functions $f\in\pre\omega2$ such that the singleton \set f is \pzo\ are the recursive functions. This differs dramatically from the situation for \pre\omega\omega; for example,  \cite[Corollary IV.2.22]{RTH} establishes that every hyperarithmetical ($\Delta^1_1$) set of natural numbers is recursive in a function $f\in\pre\omega\omega$ such that  \set f is \pzo. 

Even \pzo\ classes ${}\incl\pre\omega2$ with no recursive member have a relatively simple element.

\begin{definition}
A function $g\in\pre\omega\omega$ is \dff{almost recursive} iff every $f\leq_Tg$ is recursively bounded.
\end{definition}

\begin{lemma} For any sequence \functionof{P_a}{a\in\omega} of\ \pzo\ classes such that for all $a$, $P_{a+1}\incl P_a\incl\pre\omega2$, if for all $a$, $P_a\not=\emptyset$, then also $\bigcap_{a\in\omega}P_a\not=\emptyset$.
\end{lemma}

\begin{proof}
Fix recursive trees $U_a$ such that $P_a=[U_a]$; we may assume that also $U_{a+1}\incl U_a$. Then using Lemma \ref{Konig},
\begin{align*}
\forall a\;P_a\not=\emptyset&\Implies\forall a\;(\exists f\in\pre\omega2)\;\forall m\;[f\restrict m\in U_a]\\
&\Implies\forall a\;\forall m\;(\exists\sigma\in\pre m2)\;[\sigma\in U_a]\\
&\Implies\forall m\;(\exists \sigma\in \pre m2)\;\forall^\infty a\;[\sigma\in U_a]\\
&\Implies\forall m\;(\exists \sigma\in \pre m2)\;\forall a\;[\sigma\in U_a]\\
&\Implies(\exists f\in\pre\omega2)\;\forall a\;[f\restrict m\in U_a]\Implies\bigcap_{a\in\omega}P_a\not=\emptyset.
\end{align*}
The third implication uses the fact that ${}^m2$ is finite.
\QED
\end{proof}

\begin{proposition}\label{pzobases}
For any \pzo\ class $P\incl\pre\omega2$,
\begin{enumerate}
\item there exists $f\in P$ such that $f\leqt\zerot'$ --- equivalently, $f\in\Delta^0_2$;
\item there exists $g\in P$ such that $g$ is almost recursive.
\end{enumerate}
\end{proposition}

\begin{proof}
For (i) we take $f={\ssf LMB}(P)$,  the \dff{left-most branch} of $T_P$:
\[{\ssf LMB}(P)(m):=\text{least }n\;[({\ssf LMB}(P)\restrict m)^\frown(n)\in T_P].\]
Clearly there is a unique such function, ${\ssf LMB}(P)\in P$ and since $T_P$ is \pzo, ${\ssf LMB}(P)\leqt\zerot'$.

For (ii), set $P_0:=P$. Given $P_a$, if
\[\exists m\;(\exists f\in P_a)\;\set a^f(m)\uparrow,\]
fix the least such $\overline m$ and set $P_{a+1}:=\bigsetof{f\in P_a}{\set a^f(\overline m)\uparrow}$. Otherwise, set $P_{a+1}:=P_a$. Each $P_a$ is \pzo\ and by the preceding lemma, $\overline P:=\bigcap_{a\in\omega}P_a\not=\emptyset$. For any $g\in\overline P$ and $f\leqt g$, say $f=\set a^g$, $\set a^g$ is total on $P_a$, so by Proposition \ref{compactness}, $\set a(P_a)$ and in particular $\set a^g$ is recursively bounded. 
\QED
\end{proof}

\begin{remark}Part (i) of the proposition is known as the Kreisel Basis Theorem (\cite{Kr}) improved in \cite[Theorem 1]{Sh} and \cite[Theorem 2.1]{JoSoa} to the fact that $f$ may be chosen so that $f'\eqt\zerot'$, the Low Basis Theorem. Part (ii) is \cite[Theorem 2.4]{JoSoa} and is known as the Hyperimmune-free Basis Theorem.
\end{remark}

\begin{lemma}[{\cite[Theorem 4.18]{Si1}}]
For any almost recursive function $\bar g$ and any $f\leqt \bar g$, there exists a total recursive functional $\Phi$ such that $\Phi(\bar g)=f$. Hence $f\leq_{\ssf tt}g$.
\end{lemma}

\begin{proof}
If $f=\set a^{\bar g}$, set $f^*(m):=\text{least k }[\set a^{\bar g\restrict k}(m)\downarrow]$. Since $f^*\leqt \bar g$, there exists a recursive function $h$ that bounds $f^*$ and thus
\[\Phi(g):=\cases{\set a^{g\restrict h(m)}(m),&if defined;\cr0,&otherwise;\cr}\]
is a total recursive functional such that $\Phi(\bar g)=f$. The last assertion follows from Remark \ref{truthtable}.
\QED
\end{proof}

\begin{proposition}[{\cite[Lemma 6.9]{Si1}}]\label{subpzostrong}
For any \pzo\ classes $P,Q\incl\pre\omega 2$,
\[P\leqw Q\Implies(\exists R\incl Q)\;[\emptyset\not=R\in\pzo\qand P\leqs R].\]
\end{proposition}

\begin{proof}
Given $P\leqw Q$, by Proposition \ref{pzobases} fix an almost recursive $\bar g\in Q$ and some $f\in P$ such that $f\leqt\bar g$. By the preceding lemma, there exists a total recursive function $\Phi$ such that $\Phi(\bar g)=f$ and by Proposition \ref{inversepzo} it suffices to set $R:=\Phi^{-1}(P)\cap Q$.
\QED
\end{proof}
\end{section}

\begin{section}{Proof of Theorem A}
In this section we present a proof of the existence of a top element for each of \Dgps\ and \Dgpw:

\begin{thma}[{\cite[Theorem 3.20]{Si2}}]
The largest element of ${\mathbb P}_\bullet$ is
\[{\bf 1}_\bullet:=\deg_\bullet(\dnr2)
=\deg_\bullet({\ssf CpEx}({\cal T}))\]
for ${\cal T}={}$Peano Arithmetic or any standard first-order theory of arithmetic or sets.
\end{thma}

We begin with an alternative representation for \pzo\ classes in terms of {\em propositional} logic. In the propositional language described in Section 4 we fix an enumeration $\p0$, $\p1$, \dots\ of the atomic sentences and extend this to a G\"odel numbering of the set {\ssf PS} of propositional sentences. Consider a set $X\incl\omega$ as a classical truth-assignment $X: {\ssf PS}\to\set{{\ssf truth},{\ssf falsity}}$ defined by setting
\[X(\p n)={\ssf truth}\Iff n\in X\]
and extending in the usual way to all propositional sentences.
A classical \dff{propositional theory} $\cal T$ is a set of sentences closed under tautological consequence, and the set of \dff{models} of a theory is
\[{\Mod}({\cal T}):=\setof X{(\forall \phi\in{\cal T})X(\phi)={\ssf truth}}.\]
As usual, call a theory $\cal T$ \dff{recursively enumerable} or r.e.~iff the set of G\"odel numbers of elements of $\cal T$ is an r.e.~set. It is a standard result of propositional logic that for any r.e.~set $\Gamma$ of sentences, the set
\[\Th(\Gamma):=\setof\phi{\phi\text{ is a tautological consequence of }\Gamma}\]
is an r.e.~theory. Hence from the standard enumeration $W_0, W_1,\,\ldots$ of all r.e.~sets of numbers we can derive an effective enumeration ${\cal T}_0,\,{\cal T}_1,\,\ldots$ of all r.e.~theories. It will also be convenient to set for $\sigma\in{}^{<\omega}2$,
\[\q\sigma:=\lAnd_{i<{\ssf lg}(\sigma)}\p i^{\sigma(i)},\]
where $\p i^1$ is $\p i$, $\p i^0$ is $\lnot\p i$ and $\q\emptyset:=\p0\lor\lnot\p0$.

\begin{lemma}\label{pzorep}
$P\incl\pre\omega2$ is a \pzo\ class iff $P={\Mod}({\cal T})$ for some r.e.~propositional theory $\cal T$.
\end{lemma}

\begin{proof} 
Each $\Mod({\cal T})$ is a \pzo\ class since
\[X\in{\Mod}({\cal T})\Iff \forall \phi[\phi\not\in{\cal T}\qor X(\phi)={\ssf truth}].\]
Suppose that $P$ is a \pzo\ class and set 
\[{\cal T}:=\setof\phi{P\incl\Mod(\phi)}.\] 
Easily $\cal T$ is a theory and $P\incl\Mod({\cal T})$ by
definition. Conversely, if $X\in\Mod({\cal T})$, to show that
$X\in P$ it suffices to show that for all $n$, there is
some $Y\in P$ such that $\sigma_n:=X\restrict n=Y\restrict
n$. If for some $n$ this fails, then no $Y\in P$ is a
model of $\q{\sigma_n}$, so every $Y\in P$ is a model of
$\lnot\q{\sigma_n}$, and hence $\lnot\q{\sigma_n}\in\cal
T$. But this is impossible, since clearly
$X(\lnot\q{\sigma_n})={\ssf falsity}$. 
\QED
\end{proof}

\begin{definition} 
A propositional theory $\cal U$ is \dff{effectively incompletable} iff $\cal U$ is consistent and there exists a recursive mapping $a\mapsto\theta_a$ (formally it is the function $a\mapsto\text{ the G\"odel number of }\theta_a$ that is recursive) of $\omega\to{\ssf PS}$ such that for all $a$,
\[{\cal U}\incl{\cal T}_a\hbox{ consistent}\Implies\hbox{ both }{\cal T}_a\cup\set{\theta_a}\hbox{ and }{\cal T}_a\cup\set{\lnot\theta_a}\hbox{ are consistent}\]
\end{definition}

This is of course an effective propositional version of the property that the First Incompleteness Theorem establishes for sufficiently strong first-order theories.

\begin{lemma}
For any effectively incompletable r.e.~propositional theory \;$\cal U$ and any r.e.~propositional theory $\cal T$ there exists a recursive mapping
$(\phi,\psi,\chi)\mapsto\theta_{\phi,\psi,\chi}$
such that if both ${\cal T}\cup\set\phi$ and ${\cal
U}\cup\set\psi$ are consistent, then
\begin{enumerate}
\item[\textup{(i)}]${\cal T}\cup\set{\phi,\chi}\hbox{ is consistent}\Iff
{\cal U}\cup\set{\psi,\theta_{\phi,\psi,\chi}}\hbox{ is consistent}$;
\item[\textup{(ii)}]${\cal T}\cup\set{\phi,\lnot\chi}\hbox{ is consistent}\Iff {\cal
U}\cup\set{\psi,\lnot\theta_{\phi,\psi,\chi}}\hbox{ is consistent}.$
\end{enumerate}
\end{lemma}

\begin{proof}
With $a\mapsto\theta_a$ witnessing the effective
incompletability of
$\cal U$, by the Recursion Theorem there exists an index
$\bar a$ effectively computable from
$(\phi,\psi,\chi)$ such that
\[\openup2\jot{\cal T}_{\bar a}=
\cases{{\ssf Th}\bigl({\cal U}\cup\set{\psi,\theta_{\bar
a}}\bigr), &if ${\cal T}\cup\set{\phi,\chi}$ is
inconsistent;\cr {\ssf Th}\bigl({\cal
U}\cup\set{\psi,\lnot\theta_{\bar a}}\bigr), &if ${\cal
T}\cup\set{\phi,\lnot\chi}$ is inconsistent;\cr {\ssf
Th}\bigl({\cal U}\cup\set{\psi}\bigr), &otherwise.}\]
Note that the assumed consistency of ${\cal
T}\cup\set\phi$ ensures that exactly one of
these cases holds. 

Let $\theta_{\phi,\psi,\chi}:=\theta_{\bar a}$. For (i)(\Larrow), if
${\cal T}\cup\set{\phi,\chi}$ is inconsistent, we have
that ${\cal U}\incl{\cal T}_{\bar a}$, but ${\cal T}_{\bar
a}\cup\set{\lnot\theta_{\bar a}}$ is inconsistent. It
follows that ${\cal T}_{\bar a}$ is inconsistent as required.
The second case gives similarly (ii)(\Larrow).

In the third case, if both if ${\cal
T}\cup\set{\phi,\chi}$ and 
${\cal T}\cup\set{\phi,\lnot\chi}$ are consistent, we have
${\cal T}_{\bar a}={\ssf Th}({\cal U}\cup\set\psi)$, which
is consistent by hypothesis, so 
\[{\cal U}\cup\set{\psi,\theta_{\bar a}}\incl{\cal
T}_{\bar a}\cup\set{\theta_{\bar a}}
\qand
{\cal U}\cup\set{\psi,\lnot\theta_{\bar a}}\incl{\cal
T}_{\bar a}\cup\set{\lnot\theta_{\bar a}}\]
are both consistent by the properties of $\theta_{\bar a}$ which gives (i)(\Rarrow) and (ii)(\Rarrow).
\QED
\end{proof}

\begin{lemma}
For any effectively incompletable r.e.~ propositional theory $\cal U$, 
${\ssf deg}_\bullet\bigl({\ssf Mod}({\cal U})\bigr)={\bf 1}_\bullet$. In fact, for any consistent r.e.~propositional  theory $\cal T$,
\begin{enumerate}
\item[\textup{(i)}] there exists a recursive surjection ${\ssf Mod}({\cal U})\to{\ssf Mod}({\cal T})$;
\item[\textup{(ii)}] if also $\cal T$ is recursively incompletable, then this is a recursive isomorphism.
\end{enumerate}
\end{lemma}

\begin{proof}
Fix $\theta^{\cal U}$ for $\cal U$ as in the preceding Lemma and 
define recursive mappings $\sigma\mapsto\phi_\sigma$ and $\sigma\mapsto\psi_\sigma$ by
\[\psi_\emptyset:={\ssf p}_0\lor\lnot{\ssf p}_0,\qquad\phi_\sigma:=\q\sigma,\]
and
\[\psi_{\sigma^\frown(i)}:=\cases{\psi_\sigma\land\lnot\theta^{\cal U}_{\phi_\sigma,\psi_\sigma,{\ssf p}_{{\ssf lg}(\sigma)}},
&if $i=0$;\cr 
\psi_\sigma\land\theta^{\cal U}_{\phi_\sigma,\psi_\sigma,{\ssf p}_{{\ssf lg}(\sigma)}},
&if $i=1$.\cr}\]
Then easily for all $\sigma\in\pre{<\omega}2$,
\[{\cal T}\cup\set{\phi_\sigma}\hbox{ is consistent}\Iff
{\cal U}\cup\set{\psi_\sigma}\hbox{ is consistent}.\]
For any $f\in\pre\omega2$, set
\[{\cal T}^f:={\cal T}\cup{\setof{\phi_{f\restrict n}}{n\in\omega}}
\qand
{\cal U}^f:={\cal U}\cup{\setof{\psi_{f\restrict n}}{n\in\omega}}.\]
Then easily
\begin{enumerate}
\item[(1)] $X\in{\ssf Mod}({\cal T})\Iff \hbox{there exists (a unique) }f\in\pre\omega2, X\in{\ssf Mod}({\cal T}^f)$;
\item[(2)] $Y\in{\ssf Mod}({\cal U})\Iff \hbox{there exists (a unique) }f\in\pre\omega2, Y\in{\ssf Mod}({\cal U}^f)$;
\item[(3)] ${\cal T}^f\hbox{ is consistent}\Iff {\cal U}^f\hbox{ is consistent}$;
\item[(4)] if ${\cal T}^f$ is consistent, then it has exactly one model.
\end{enumerate}

For $Y\in{\ssf Mod}({\cal U})$, set
$$\Phi(Y):=\hbox{the unique }X\in{\ssf Mod}({\cal T}^f)$$
for the unique $f$ such that $Y\in{\ssf Mod}({\cal U}^f)$. This is well-defined and surjective, which establishes (i).

If also $\cal T$ is effectively incompletable, we modify the definitions as follows.
\[\phi_\emptyset:={\ssf p}_0\lor\lnot{\ssf p}_0=:\psi_\emptyset;\]
when ${\ssf lg}(\sigma)=2n$,
\begin{align*}
\phi_{\sigma^\frown(i)}
&:=\cases{\phi_\sigma\land\lnot{\ssf p}_n,
&if $i=0$;\cr
\phi_\sigma\land{\ssf p}_n,
&if $i=1$;\cr}\\
\psi_{\sigma^\frown(i)}
&:=\cases{\psi_\sigma\land\lnot\theta^{\cal U}_{\phi_\sigma,\psi_\sigma,{\ssf p}_n},
&if $i=0$;\cr 
\psi_\sigma\land\theta^{\cal U}_{\phi_\sigma,\psi_\sigma,{\ssf p}_n},
&if $i=1$;\cr}
\end{align*}
and when ${\ssf lg}(\sigma)=2n+1$,
\begin{align*}
\phi_{\sigma^\frown(i)}
&:=\cases{\phi_\sigma\land\lnot\theta^{\cal T}_{\psi_\sigma,\phi_\sigma,{\ssf p}_n},
&if $i=0$;\cr 
\phi_\sigma\land\theta^{\cal T}_{\psi_\sigma,\phi_\sigma,{\ssf p}_n},
&if $i=1$;\cr}\\
\psi_{\sigma^\frown(i)}
&:=\cases{\psi_\sigma\land\lnot{\ssf p}_n,
&if $i=0$;\cr
\psi_\sigma\land{\ssf p}_n,
&if $i=1.$\cr}
\end{align*}
Now (1) -- (4) follow as before, but in addition
\begin{enumerate}
\item[(5)] if ${\cal U}^f$ is consistent, then it has exactly one model,
\end{enumerate}
from which it follows that the functional $\Phi$ is injective.
\QED
\end{proof}

It remains to verify that the \pzo\ classes ${\ssf DNR}_2$ and ${\ssf CpEx}({\cal T})$ are indeed of the form $\Mod({\cal U})$ for an effectively incompletable theory $\cal U$.

\begin{definition}
\begin{enumerate}
\item[(i)] Disjoint sets $A,B\incl\omega$ are \dff{effectively inseparable} iff there exists a recursive function $h$ such that for any r.e.~sets $W_c$ and $W_d$, if
\[A\incl W_c, \qquad B\incl W_d\qqand W_c\cap W_d=\emptyset\]
then $h(c,d)\notin W_c\cup W_d$;
\item[(ii)]${\ssf K}_i:=\setof a{\set a(a)\simeq i}$.
\end{enumerate}
\end{definition}

\begin{lemma}
The following pairs are effectively inseparable:
\begin{enumerate}
\item[\textup{(i)}] ${\ssf K}_0$ and ${\ssf K}_1$;
\item[\textup{(ii)}] $\cal T$ and ${\ssf Neg}\,{\cal T}:=\setof{\lnot\phi}{\phi\in{\cal T}}$ for $\cal T$ Peano Arithmetic or any standard first-order theory of arithmetic or sets.
\end{enumerate}
\end{lemma}

\begin{proof}
(ii) is an effective version of the First Incompleteness Theorem. For (i), by the $S^m_n$-Theorem there exists a recursive
function $h$ such that for all $c$ and $d$,
$$\openup2\jot\set{h(c,d)}(x)\simeq\cases{
1,&if $\exists s\;[x\in W_{c,s}-W_{d,s}]$;\cr
0,&if $\exists s\;[x\in W_{d,s}-W_{c,s}]$;\cr
\uparrow,&otherwise.\cr}$$
If ${\ssf K}_0\incl W_c$, ${\ssf K}_1\incl W_d$, and
$W_c\cap W_d=\emptyset$, then
\[h(c,d)\in W_c\Impliess
\set{h(c,d)}(h(c,d))\simeq1\Impliess h(c,d)\in W_d,\]
and 
\[h(c,d)\in W_d\Impliess
\set{h(c,d)}(h(c,d))\simeq0\Impliess h(c,d)\in W_c,\]
so $h(c,d)\notin W_c\cup W_d$. 
\QED
\end{proof}

\begin{lemma}
For any pair $A,B$ of effectively inseparable sets,
\[{\cal U}_{A,B}:={\ssf Th}\bigl(\setof{{\ssf p}_b}{b\in A}\cup\setof{\lnot p_b}{b\in B}\bigr)\]
is effectively incompletable.
\end{lemma}

\begin{proof}
Fix recursive $f$ and $g$ such that for all $a\in\omega$,
\[W_{f(a)}=\setof b{{\ssf p}_b\in{\cal T}_a}\qand 
W_{g(a)}=\setof b{\lnot{\ssf p}_b\in{\cal T}_a}.\]
Given effectively inseparable $A$ and $B$, for any
$a$ such that ${\cal U}_{A,B}\incl{\cal
T}_a$ we have 
\[A\incl W_{f(a)} \qand B\incl W_{g(a)},\]
and if also ${\cal T}_a$ is consistent, then 
\[W_{f(a)}\cap
W_{g(a)}=\emptyset.\]
Hence $\theta_a:={\ssf p}_{h(f(a),g(a))}$ witnesses the incompleteness of ${\cal T}_a$.
\QED
\end{proof}

Finally, note that for any disjoint $A,B$,
\[{\ssf Mod}({\cal U}_{A,B})={\ssf Sep}(A,B):=\setof X{A\incl X\qand X\cap B=\emptyset}.\]
Then Theorem A follows from the preceding three lemmas, since
\[{\ssf DNR}_2={\ssf Mod}\bigl({\cal U}_{{\ssf K}_0,{\ssf K}_1}\bigr)\]
and
\[{\ssf CpEx}({\cal T})\hbox{ is a \pzo\ subset of }{\ssf Mod}\bigl({\cal U}_{{\cal T},{\ssf Neg}\,{\cal T}}).\]
\end{section}

\begin{section}{Proof of Theorem D}
Recall that for a Turing degree $\dga\in\Dgt$ and $A\in\dga$,
\[\dgaw:=\degw(\set A)\qqand \dgawstar:=\onew\meet\dgaw.\]
We will show here that, as asserted in detail by Theorem D, the mapping $\dga\mapsto\dgawstar$ is an embedding of \Dgpt\ into \Dgpw\ respecting all of the structure of \Dgpt\ as a bounded upper semi-lattice. In fact, for use is later sections we will establish a bit more.

The first task is to verify that for any r.e.\ Turing degree \dga, $\dgawstar$ is actually a member of \Dgpw, since \set A is generally $\Pi^0_2$ but not \pzo. The following technique actually yields considerably more than we need here and will have other applications below.  

\begin{lemma}[{\cite[Lemma 3.3]{Si3}}]\label{existspzo}
For any $\Sigma^0_3\hbox{ set }S\incl\pre\omega\omega$ and \pzo\  set $\emptyset\not=R\incl\pre\omega2$, there exists a \pzo\ set\hfill\break
$S^*\incl\pre\omega2$ such that
\[S^*\eqw R\meet S.\]
\end{lemma}

\begin{proof}
We define \pzo\ sets $S_1$, $S_2\incl\pre\omega\omega$ and $S_3\incl\pre\omega2$ such that
\[S_3\eqw S_2\leqw S_1\eqw S\qand R\meet S_1\leqw S_2\]
and set $S^*:=R\meet S_3$, so that
\[S^*\eqw R\meet S_2\eqw R\meet S_1\eqw R\meet S.\]

A $\Sigma^0_3$ set $S$ may be represented by
a recursive map $(x,y)\mapsto T_{x,y}$ of pairs of natural numbers to trees such that
\[f\in S\Ifff\ex x\all y\ex z\left(f\restrict z\notin T_{x,y}\right).\]
Set
\[S_1:=\setof{\code{x,f,g}}
{\all y\left(f\restrict g(y)\notin T_{x,y}\right)}.\]
Clearly $S_1$ is \pzo,
$S\leqw S_1\text{ since }\code{x,f,g}\in S_1\text{ computes }f\in S$
and
$S_1\leqw S,\text{ since if }$\hfill\break
$f\in S\text{ with witness $x$, then $f$ computes}$
\[g(y):=\hbox{least }z\left(f\restrict z\notin T_{x,y}\right)\]
and hence computes $\code{x,f,g}\in S_1$.

Now fix a recursive trees $T_{S_1}\incl\pre{<\omega}\omega$ and $T_R\incl\pre{<\omega}2$ such that 
\[S_1=[T_{S_1}]\qqand R=[T_R].\]
Set
\begin{align*}
T_{S_2}:=
\biggl\{\,{\tau_0}^\frown(n_0)^\frown&\cdots^\frown{\tau_{k-1}}^\frown(n_{k-1})^\frown\tau_k:(\forall i<k)\Bigl[2\leq n_i\leq\sum_{j\leq i}{|\tau_j|}\Bigr],\\
&(n_0-2,\ldots,n_{k-1}-2)\in T_{S_1}\qand
(\forall i\leq k)\;\tau_i\in T_R\;\biggr\}
\end{align*}
and $S_2:=[T_{S_2}]$. The idea is that we use $\tau\in T_R$ to pad $\sigma\in T_{S_1}$ so that
\[(m_0,\ldots,m_l)\in T_{S_2}\Implies m_l\leq l+1.\]

$S_2\leqw S_1$ because any $g\in S_1$ computes some \functionof{\tau_i\in T_R}{i\in\omega} such that for all $i$, $g(i)\leq\sum_{j\leq i}{|\tau_j|}$ and hence a member of $S_2$. To see that $R\land S_1\leqw S_2$, for any $h\in S_2$, there are two cases:
\begin{enumerate}
\item[(1)] if $h$ has infinitely many values $\geq 2$,
these determine a path through $T_{S_1}$
hence a member of $S_1$;
\item[(2)] otherwise the tail of $h$ beyond its last value $\geq2$
is a path through $T_R$ so a member of $R$.
\end{enumerate}
Finally we set
\[T_{S_3}
:=\bigsetof{\,0^{m_0}10^{m_1}1\cdots0^{m_{l-1}}10^n}
{n\leq l+1\qand (\zlist ml)\in T_{S_2}}\]
and $S_3:=[T_{S_3}]$. $T_{S_3}$ is a tree since $m_l\leq l+1$ and easily $S_3\eqw S_2$.
\QED
\end{proof}

Now we can prove (a strengthening of) all but one piece (the implication $(\Larrow)$ of (ii)) of Theorem D:
\begin{proposition}
For all $\dga,\dgb\in\Dgpt$ and more generally $\dga,\dgb\leqt\onet$ ($\Delta^0_2$),
\begin{enumerate}
\item[\textup{(i)}] $\dgawstar\in\Dgpw$;
\item[\textup{(ii)}] $\dga\leq\dgb\Implies\dgawstar\leq\dgbwstar$;
\item[\textup{(iii)}] $({\bf 0}_T)_{\ssf w}^*=\zerow$\quad\hbox{and}\quad $({\bf 1}_T)_{\ssf w}^*=\onew$;
\item[\textup{(iv)}] $(\dga\join\dgb)_{\ssf w}^*=\dgawstar\join\dgbwstar$.
\end{enumerate}
\end{proposition}

\begin{proof}
Part (i) follows from the preceding Lemma and the observation that for $A\in\Delta^0_2$, $\set{A}\in\Pi^0_2\incl\Sigma^0_3$.
(ii) is immediate from the definitions. For (iii) we have
\begin{align*}
(\zerot)^*_{\ssf w}
&=\degw(\set\emptyset)\meet\onew=\zerow\meet\onew=\zerow;\\
\noalign{\vskip6pt}
(\onet)^*_{\ssf w}
&=\degw(\set{\emptyset'})\meet\degw(\dnr2)=\degw(\dnr2)=\onew,\qquad
\end{align*}
because $\emptyset'$ computes a function $f\in\dnr2$: 
\[f(a):=\cases{1-\set a(a),&if $\set a(a)\downarrow$;\cr 0,&otherwise.\cr}\]
Finally, for (iv),
\begin{align*}
(\dga\join\dgb)^*_{\ssf w}
 &=\degw(\set{A\oplus B})\meet\onew\\
 &=\bigl(\degw(\set A)\join\degw(\set{B})\bigr)\meet\onew=\dga^*_{\ssf w}\join\dgb^*_{\ssf w}.\qedhere
 \end{align*}
 \end{proof}
 
 To complete the proof of Theorem D, we note that from the definitions and the Proposition we have for $\dga,\dgb\in\Dgpt$,
 \begin{align*}
 \dgawstar\leq\dgbwstar
 &\Iff\dgawstar\leq\dgb_{\ssf w}\\
 &\Iff\dga_{\ssf w}\leq\dgb_{\ssf w}\qor\onew\leq\dgb_{\ssf w}\\
 &\Iff\dga\leq\dgb\qor\onew\leq\dgb_{\ssf w}
 \end{align*}
 because $\dgb_{\ssf w}$ is the weak degree of a singleton. Hence it suffices to prove

\begin{lemma} 
For all $\dgb\in\Dgpt$, $\onew\leq\dgb_{\ssf w}\Implies \dgb=\onet$, so in particular if $\onew\leq\dgb_{\ssf w}$, then for all $\dga\in\Dgpt$, $\dga\leq\dgb$.
\end{lemma}

\begin{proof}
Assume that $\dgb=\degt(B)$ for an r.e.\ set $B$ and $\onew\leq\dgb_{\ssf w}$ -- that is, 
(by Theorem A) there exists $g\leqt B$ with $g\in\dnr2$. Fix a recursive $h$ such that for all $a$,
\[W_a\not=\emptyset\Implies\forall x[\set{h(a)}(x)\in W_a]\]
 and $f\leqt g$ such that
\[W_{f(a)}=\set{g(h(a))}.\]
Then for all $a$,
\[W_a=W_{f(a)}\Implies W_a\not=\emptyset\Implies\set{h(a)}(h(a))=g(h(a))\]
contrary to $g\in\dnr2$. Hence $\forall a[W_a\not= W_{f(a)}]$; we say that $f$ is \dff{fixed-point free}. Note that one version of the Recursion Theorem asserts that no recursive function is fixed-point free.

Fix an index $a$ such that $f=\set a^B$ and a stage enumeration \functionof{B_s}{s\in\omega} of $B$, and set
\[f_s:=\set a^{B_s}_s\qand m^f(x):=\hbox{least }s\bigl[a^{B_s}_s(x)\downarrow\hbox{ correctly}\bigr],\]
so $m^f\leqt B$ and $f(x)=\lim f_s(x)=f_{m^f(x)}(x)$.
Fix a Turing complete set ${\ssf K}\in\zerot'$ with stage enumeration \functionof{{\ssf K}_s}{s\in\omega} and a partial recursive function $x\mapsto s_x$ such that
\[x\in{\ssf K}\Implies s_x\simeq\hbox{least }s[x\in{\ssf K}_s].\]
By the Recursion Theorem there exists a recursive function $h$ such that
\[W_{h(x)}=\cases{W_{f_{s_x}(h(x))},&if $x\in{\ssf K}$;\cr\emptyset,&otherwise.\cr}\] Since $f$ is fixed-point free, for all $x$, $W_{h(x)}\not=W_{f(h(x))}$, whence 
$f_{s_x}(h(x))\not=f(h(x))$ and thus $s_x<m^f(h(x))$.
Then ${\ssf K}\leqt B$, since $x\in{\ssf K}\Iff x\in{\ssf K}_{m^f(h(x))}$,
so $B$ is Turing complete and $\dgb=\onet$.
\QED
\end{proof}

It should be noted that this lemma is a version of the Arslanov Completeness Criterion; for extensions and complete references see \cite{GenArsl}.
\end{section}

\begin{section}{Proof of Theorem E}
To establish that \Dgs\ is not implicative we need to show that for some $\dgp,\dgq\in\Dgs$, $\dgp\latimpl\dgq$ does not exist --- that is that there is no largest ${\bf x}\in\Dgs$ such that $\dgp\meet{\bf x}\leq\dgq$. This follows immediately from the following

\begin{proposition}[{\cite[Theorem 5.4]{Sor0}}] \label{notimplicative}
For any non-recursive function $f$, there exists $Q\incl\pre\omega\omega$ such that for all $X\incl\pre\omega\omega$,
\[\set f\meet X\leqs Q\Implies\exists Y\bigl(X\les Y\qand\set f\meet Y\leqs Q\bigr).\]
Thus there is no greatest ${\bf x}\in\Dgs$ such that $\degs(\set f)\meet{\bf x}\leq\degs Q$, and\hfill\break
$\degs(\set f)\latimpl \degs Q$ does not exist.
\end{proposition}

Towards the proof, we establish two lemmas.
\begin{lemma}
For any functions $f<_T\zlist gm$ and any $\tau\in\pre{<\omega}2$, there exists a function $g\supset\tau$ such that $f<_Tg$ and for all $i<m$, $g_i$ is Turing incomparable with $g$.
\end{lemma}

\begin{proof}
This is a fairly standard so-called Kleene-Post construction; for completeness we provide a sketch of the proof. We define a strictly increasing sequence \functionof{\tau_s}{s\in\omega} of of finite sequences as follows. Given $\tau_s$,
\begin{itemize}
\item[$\bullet$]if for some $i<m$ and $n$, $s=2am+i$ and $\set a^{g_i}(\lg(\tau_s))\simeq n$, set $\tau_{s+1}:=\tau_s^\frown(1-n)$; otherwise, $\tau_{s+1}:=\tau_s^\frown(0)$;
\item[$\bullet$]if for some $i<m$, $s=(2a+1)m+i$ and there exist $\sigma,\sigma'\supset\tau_s$ such that $\set a^{f\oplus\sigma}$ and $\set a^{f\oplus\sigma'}$ are incomparable, choose $\tau_{s+1}$ to be one of these such that \hfill\break
$\set a^{f\oplus\tau_{s+1}}\not\incl g_i$; otherwise set $\tau_{s+1}:=\tau_s^\frown(0)$.
\end{itemize}
Set $h:=\bigcup_{s\in\omega}\tau_s$ and $g:=f\oplus h$. The action taken at stage $=2am+i$ guarantees that $\set a^{g_i}\not=h$; hence $h\not\leqt g_i$ and consequently $g\not\leqt g_i$. Suppose towards a contradiction that $\set a^g=g_i$. Then at stage $s=(2a+1)m+i$ the mentioned $\sigma$ and $\sigma'$ must exist as otherwise for all $n$,
\[g_i(n)=\set a^{f\oplus\sigma_n}(n)\quad\text{for}\quad\sigma_n:=(\text{least }\sigma\supseteq\tau_s)\,\set a^{f\oplus\sigma}(n)\downarrow,\]
and thus $g_i\leqt f$, contrary to hypothesis. Hence the action taken at these stages guarantees that $g_i\not=\set a^g$.
\QED
\end{proof}

\begin{lemma}
For any non-recursive function $f$ there exists a sequence  of functions\functionof{g_m}{m\in\omega} such that for all $m$ and $n$,
\begin{enumerate}
\item[\textup{(i)}] $f<_T g_m$;
\item[\textup{(ii)}] $m\not=n\Implies g_m$ and $g_n$ are Turing incomparable;
\item[\textup{(iii)}] $\set m^{(m)^\frown g_m}\not=f$.
\end{enumerate}
\end{lemma}

\begin{proof}
Given \zlist gm, by the preceding lemma, for every $\tau\in\pre{<\omega}2$, there exists $g\supset\tau$ such that $f<_Tg$ and $g$ is Turing incomparable with each of \zlist gm. Suppose that for all such $g$, $\set m^{(m)^\frown g}=f$. Then for all $k$,
\[f(k)=\set m^{(m)^\frown\sigma_k}(k)\quad\text{for}\quad\sigma_k:=\text{least }\tau\,\bigl[\set m^{(m)^\frown\tau}(k)\downarrow\bigr],\]
and thus $f$ is recursive contrary to hypothesis. Hence we may choose $g_m$ to be such a $g$ such that $\set m^{(m)^\frown g_m}\not=f$.
\QED
\end{proof}

\begin{proof}[\ifautoproof\else Proof \fi of Proposition \ref{notimplicative}]
Given a non-recursive function $f$, set 
\[Q:=\bigmeet_{m\in\omega}\set{g_m}\]
(see Remark \ref{genmeet}) with the $g_m$ from the preceding lemma.
Assume that $\set f\meet X\leqs Q$ and fix $\Phi:Q\to\set f\meet X$.
For $i=0,1$, set $Q^i:=\setof{h\in Q}{\Phi(h)(0)=i}$ so $Q^0$ and $Q^1$ partition $Q$ and 
\begin{enumerate}
\item[(1)]$\set f\leqs Q^0\qand X\leqs Q^1$.
\end{enumerate}
By (iii) of the lemma, $\set f\not\leqs Q$, so $Q^1\not=\emptyset$; 
fix $\bar m$ such that $(\bar m)^\frown g_{\bar m}\in Q^1$ and set 
\[Y:=Q^1\setminus\set{(\bar m)^\frown g_{\bar m}}.\]
\noindent We need to establish
\begin{enumerate}
\item[(2)] $X\leqs Y$;
\item[(3)] $Y\not\leqs X$;
\item[(4)] $\set f\meet Y\leqs Q$.
\end{enumerate}
(2) holds because $X\leqs Q^1$ and $Y\incl Q^1$ so $Q^1\leqs Y$. Towards (3), note that 
\[\setof{g_n}{\bar m\not=n}\leqs Y \text{ (via the mapping }(m)^\frown g\mapsto g)\]
and $X\leqs \set{g_{\bar m}}$ by (1) because $Q^1\leqs \set{g_{\bar m}}$ (via the mapping $g_{\bar m}\mapsto(\bar m)^\frown g_{\bar m}$). Hence $Y\leqs X$ would contradict (ii) of the lemma.

For (4), by (i) of the lemma, choose a recursive functional $\Theta$ such that\hfill\break
$\Theta((\bar m)^\frown g_{\bar m})=f$.
Then $\set f\meet Y\leqs Q$ via the functional $\Psi$ defined by
\[\openup2\jot
\Psi(h):=\cases{\Phi(h),&if $\Phi(h)(0)=0$;\cr
(1)^\frown h,&if $\Phi(h)(0)=1$ and $h(0)\not=\bar m$;\cr
(0)^\frown\Theta(h),&otherwise.\cr}\]
The first clause handles $h\in Q^0$ such that $\Phi(h)=(0)^\frown f$, the second handles $h\in Y$, and the third the only remaining $h\in Q$, namely $h=(\bar m)^\frown g_{\bar m}$.
\QED
\end{proof}
\end{section}

\begin{section}{Proof of Theorem F}
To establish that \Dgps\ is not implicative, we prove in this section the following effective version of Proposition \ref{notimplicative}:

\begin{proposition}[{\cite[Theorem 3.2]{Te1}}]\label{pzonotimplicative}
There exist \pzo\ classes $P$ and $Q$ such that for every \pzo\ class $X$
\[P\meet X\leqs Q\Implies(\exists Y\in\pzo)\bigl[X\les Y\qand 
P\meet Y\leqs Q\bigr].\]
Thus there is no greatest ${\bf x}\in\Dgps$ such that $\degs(P)\meet{\bf x}\leq\degs Q$, and $\degs(P)\latimpl \degs(Q)$ does not exist.
\end{proposition}

Again there are two lemmas. The first is a weak version of \cite[Theorem 4.1]{JoSoa} and will not be proved here.

\begin{lemma}
There exist uniformly r.e.\ sequences of sets \functionof{A_m}{m\in\omega} and\hfill\break
\functionof{B_m}{m\in\omega} such that for all $m$, $A_m\cap B_m=\emptyset$ and for all $m\not=n$, $C \in\sepset(A_m,B_m)$ and $D\incl\omega$,
\[D\leqt C\Implies D\notin\sepset(A_n, B_n).\]
In particular, for all $m$, 
\[\bigmeet_{m\not=n}\sepset(A_n,B_n)\not\leqw\sepset(A_m,B_m),\]
so also
\[\bigmeet_{m\not=n}\sepset(A_n,B_n)\not\leqs\sepset(A_m,B_m).\ \noproof\]
\end{lemma}

\begin{lemma}[{\cite[Lemma 3.1]{Te1}}]
There exist non-empty \pzo\ classes $P$ and \functionof{R_m}{m\in\omega} such that for all $m$,
\begin{enumerate}
\item[\textup{(i)}] $P\leqs R_m$;
\item [\textup{(ii)}] $\bigmeet_{m\not=n}R_n\not\leqs R_m$;
\item [\textup{(iii)}] $(\forall g\in R_m) \set m^{(m)^\frown g}\not\in P$.
\end{enumerate}
\end{lemma}

\begin{proof}
Fix \functionof{A_m}{m\in\omega} and \functionof{B_m}{m\in\omega}  as in the preceding lemma. Set $S_0=\emptyset$ and $S_{m+1}:=\sepset(A_m,B_m)$. 
We define a function $\alpha\in\pre\omega\omega$ and finite sequences\hfill\break
\functionof{\sigma_m}{m\in\omega} such that 
\[P:=\bigmeet_{n\in\omega} S_{\alpha(n)}\qand R_m:={\sigma_m}^\frown S_{m+1}\]
are \pzo\ classes and satisfy (i)--(iii). Indeed, property (ii) is immediate from the corresponding property of the $S_m$ and (i) follows immediately as long as we ensure that ${\ssf Im}(\alpha)=\omega$, since then modulo a prefix $R_m$ is a subset of $P$. To guarantee property (iii) it will suffice to ensure that for each $m$, if for some $\sigma$ and $n$, $\set m^{(m)^\frown\sigma}(0)\simeq n$, then $\sigma\incl\sigma_m$ and $\alpha(n)\not=m+1$, since then for $g\in R_m$, $\set m^{(m)^\frown g}(0)\simeq n$ but either $S_{\alpha(n)}=\emptyset$ or $S_{\alpha(n)}=S_{m'+1}$ for $m'\not=m$. It follows that $\set m^{(m)^\frown g}\notin P$ because neither of these sets has an element recursive in $g$.

We define in stages approximating sequences 
\[\functionof{\alpha_s}{s\in\omega} \qand \functionof{\sigma_{m,s}}{m,s\in\omega},\]
and set
\[\alpha:=\lim_{s\to\infty}\alpha_s\qand\sigma_m:=\lim_{s\to\infty}\sigma_{m,s}.\]
The $\alpha_s$ will be partial functions with 
\[\text{domain}(\alpha_s)=\setof i{i<d_s}\qand \text{image}(\alpha_s)= \setof j{j<s}. \]
Set $\alpha_0:=\emptyset$ and $\sigma_{m,0}:=\emptyset$. At stage $s+1$ set $\alpha_{s+1}(d_s):=s$, and if for some $m,\sigma\leq s$, 
\[\sigma_{m,s}=\emptyset\qand(\exists n<d_s)\bigl[\set m^{(m)^\frown\sigma}(0)\simeq n\qand\alpha_s(n)=m+1\bigr],\]
then choose the least such pair and set
\[\alpha_{s+1}(n):=0,\quad\alpha_{s+1}(d_s+1):=m+1\qand\sigma_{m,s+1}:=\sigma.\]
For all other $m'$, $\sigma_{m',s+1}:=\sigma_{m',s}$. Note that each $\alpha_s(n)$ and $\sigma_{m,s}$ changes at most once so the limits exist and have the required properties. The $R_m$ are clearly \pzo. To see that $P$ is \pzo, let
$\beta(n):=\text{least }s[n\leq d_s]$ and $\rho(n,s)$ be the condition
\[(\exists\ m,\sigma<s)\bigl[\sigma_{m,s}=\emptyset\qand\set m^{(m)^\frown\sigma}(0)\simeq n\qand\alpha_s(n)=m+1\bigr].\]
Then if \functionof{U_n}{n\in\omega} is a uniformly \pzo\ sequence of trees such that $S_n=[U_n]$,
\[T_n:=\setof\sigma{\sigma\in U_{\alpha_{\beta(n)}(n)}\qand\forall s\lnot\rho(n,s)}\]
is a uniformly \pzo\ sequence of trees such that $S_{\alpha(n)}=[T_n]$. 
\QED
\end{proof}

\begin{proof}[\ifautoproof\else Proof \fi of Proposition \ref{pzonotimplicative}]
With $P$ and $R_m$ as in the preceding lemma, set $Q:=\bigmeet_{m\in\omega}R_m$, assume that for some $X\in\pzo$, $P\meet X\leqs Q$ and fix a recursive functional $\Phi:Q\to P\meet X$. For $i=0,1$, set $Q^i:=\setof{h\in Q}{\Phi(h)(0)=i}$ so $Q^0$ and $Q^1$ partition $Q$ and 
\begin{enumerate}
\item[(1)] $P\leqs Q^0\qand X\leqs Q^1$.
\end{enumerate}
By (iii) of the lemma, $P\not\leqs Q$, so $Q^1\not=\emptyset$ and indeed $\exists\bar m[(\bar m)^\frown R_{\bar m}\incl Q^1]$, where $\bar m$ is an index for the functional $\Phi^+(f)(n):=\Phi(f)(n+1)$. Fix such $\bar m$ and set 
\[Y:=Q^1\setminus(\bar m)^\frown R_{\bar m}
=\setof{h\in Q}{\Phi(h)(0)=1\hbox{ and }h(0)\not=\bar m};\]
the second version makes it clear that $Y\in\pzo$. We need to establish
\begin{enumerate}
\item[(2)] $X\leqs Y$;
\item[(3)] $Y\not\leqs X$;
\item[(4)] $P\meet Y\leqs Q$.
\end{enumerate}
(2) holds because $X\leqs Q^1$ and $Y\incl Q^1$  so $Q^1\leqs Y$.
Towards (3), note that 
\[\bigmeet_{\bar m\not=n}R_n\leqs Y\text{ (via the mapping }(m)^\frown g\mapsto g)\]
and $X\leqs R_{\bar m}$ by (1) since $Q^1\leqs B_{\bar m}$ (via the mapping $g\mapsto(\bar m)^\frown g$). Hence  $Y\leqs X$ would contradict (ii) of the lemma.

For (4), by (i) of the lemma, choose a recursive functional $\Theta$ such that $\Theta:(\bar m)^\frown R_{\bar m}\to P$.
Then $P\meet Y\leqs Q$ via the functional $\Psi$ defined by
\[\openup2\jot
\Psi(h):=\cases{\Phi(h),&if $\Phi(h)(0)=0$;\cr
(1)^\frown h,&if $\Phi(h)(0)=1$ and $h(0)\not=\bar m$;\cr
(0)^\frown\Theta(h),&otherwise. \qedhere\cr}\]
\end{proof}
\end{section}

\begin{section}{Proof of Theorem G}
To establish that \Dgpw\ is not implicative, we prove in this section the following companion to Proposition \ref{pzonotimplicative}:

\begin{proposition}[{\cite[Theorem 2]{Hig}}]\label{pzownotimplicative}
There exist \pzo\ classes $P$ and $Q$ such that for every \pzo\ class $X$,
\[P\meet X\leqw Q\Implies(\exists Y\in\pzo)\bigl[Y\not\leqw X\qand 
P\meet Y\leqw Q\bigr].\]
Thus there is no greatest ${\bf x}\in\Dgpw$ such that $\degw(P)\meet{\bf x}\leq\degw(Q)$ and $\degw(P)\latimpl \degw(Q)$ does not exist.
\end{proposition}

Note that strengthening the conclusion to $X\lew Y$ is not required but easily accomplished by replacing $Y$ with $X\join Y$. The proof will require three lemmas.

\begin{lemma}[{\cite[Theorem 2.5]{JoSoa}}]\label{nocountableleqw}
For any nonempty \pzo\ class $Q\incl\pre\omega2$ and any nonempty set $P\incl\pre\omega2$, if $P\leqw Q$, then either $P$ has a recursive element or $P$ is uncountable. \noproof
\end{lemma}

\begin{lemma}[{\cite[Theorem 4.7]{JoSoa}}]\label{allTincomp}
There exists a nonempty \pzo\ class $Q\incl\pre\omega2$ such that any two distinct elements of $Q$ are Turing incomparable. \noproof
\end{lemma}

\begin{lemma}\label{takeaway}
For any \pzo\ class $Q\incl\pre\omega2$ with no recursive element and any $g\in Q$, $Q\setminus\set g\leqw\dnr2$.
\end{lemma}

\begin{proof}
By Corollary \ref{nonzeroperfect}, any such $Q$ is uncountable, so there exists  $g'\in Q$ with $g'\not= g$ and hence $m$ such that $g\restrict m\not=g'\restrict m$. Then 
\[Q':=\setof{h\in Q}{h\restrict m=g'\restrict m}\]
is a nonempty \pzo\ subclass of $Q\setminus\set g$ and by Theorem A, $Q\setminus\set g\leqw Q'\leqw\dnr2$. 
\QED
\end{proof}

\begin{proof}[\ifautoproof\else Proof \fi of Proposition \ref{pzownotimplicative}]
Fix $Q$ to be any class as in Lemma \ref{allTincomp}. Restating the defining property of $Q$, we have
\[(\forall g\in Q)\;Q\setminus\set g\not\leqw\set g.\tag{1}\]
From Lemma \ref{takeaway} we have immediately
\[(\forall g\in Q)\;\dnr2\not\leqw\set g.\tag{2}\]

Fix an effective enumeration\functionof{R_a}{a\in\omega} of all \pzo\ classes as in the discussion preceding Proposition \ref{TPispzo}. Using the notation of the proof of Proposition \ref{pzobases}(i), set
\[\overline Q:=\setof {{\ssf LMB}(Q\cap R_a)}{a\in\omega}.\]
A calculation as in that proof shows that $\overline Q$ is $\Sigma^0_3$, so by Lemma \ref{existspzo} there exists a \pzo\ class $P$ such that $P\eqw\dnr2\meet\overline Q$. Trivially from the definition we have
\[(\forall g\in\overline Q)\;P\meet(Q\setminus\set g)\leqw Q.\tag{3}\]
Now fix a \pzo\ class $X$ such that $P\meet X\leqw Q$; we shall construct a \pzo\ class $Y$ such that $Y\not\leqw X$ and $P\meet Y\leqw Q$.

If $Q\not\leqw X$, then $Y:=Q$ will suffice, so for the rest of the proof we assume that $Q\leqw X$. By Proposition \ref{subpzostrong}, there exists a \pzo\ class $R\incl Q$ such that $P\meet X\leqs R$. In fact, we then have
\[X\leqs R.\tag{4}\]
To see this, fix $\Phi:R\to P\meet X=(0)^\frown P\cup(1)^\frown X$; we claim that for all $g\in R$, $\Phi(g)\in (1)^\frown X$, from which (4) follows immediately. Otherwise, by Proposition \ref{inversepzo}, $R\cap\Phi^{-1}\bigl((0)^\frown P\bigr)$ is a nonempty \pzo\ class, and for any member $g$, either $\dnr2\leqw\set g$ or $\overline Q\leqw\set g$ so by (2), $\overline Q\leqw\set g$. Hence $\overline Q\leqw R\cap\Phi^{-1}\bigl((0)^\frown P\bigr)$ which, since $\overline Q$ is countable and has no recursive element, contradicts Lemma \ref{nocountableleqw}.

Now set $\bar g:={\ssf LMB}(R)$; since $\bar g\in\Delta^0_2$ by Proposition \ref{pzobases}, $Q\setminus\set{\bar g}$ is also $\Delta^0_2$ so by Lemma \ref{existspzo} there is a \pzo\ class $Y\eqw\dnr2\meet Q\setminus\set{\bar g}$ and hence by Lemma \ref{takeaway}, $Y\eqw Q\setminus\set{\bar g}$. 

Since $Q\leqw X\leqw R$, there exist $f\in X$ and $g\in Q$ such that $g\leqt f\leqt\bar g$. Then by the defining property of $Q$, $g=\bar g$ so $f\eqt\bar g$. But by (1), $Y\not\leqw\set{\bar g}$, so $Y\not\leqw\set f$ and hence $Y\not\leqw X$ as desired.

Finally, since $\bar g\in\overline Q$, $P\leqw\set{\bar g}$ and for $g\in Q$ with $g\not=\bar g$, $Y\eqw Q\setminus\set{\bar g}\leqw\set g$, we have $P\meet Y\leqw Q$.
\QED
\end{proof}
\end{section}

\begin{section}{Proof of Theorem H}
In the pattern of the three preceding sections, we establish that \Dgpw\ is not dual-\hfill\break
 implicative via the following

\begin{proposition}\label{notdualimplicative}
For every \pzo\ class $Q$ there exists a \pzo\ class $P$ such that for every \pzo\ class $X$
\[Q\leqw P\join X\Implies(\exists Y\in\pzo)[Y\lew X\qand Q\leqw P\join Y].\]
Thus there is no smallest ${\bf x}\in\Dgpw$ such that $\degw(Q)\leq\degw(P)\join{\bf x}$ so $\degw(P)\dlatimpl\degw(Q)$ does not exist.
\end{proposition}

Again we need two lemmas for the proof.

\begin{lemma} \label{notdominated}
For any $\zerot<\degt(f)\leq\zerot'$, there exists a function $g$ such that
\[f\equiv_T g\quad\hbox{but}\quad g\hbox{ is not recursively bounded}.\]
\end{lemma}

\begin{proof}
By the Limit Lemma \cite[3.3]{Soa} fix a recursive sequence $\functionof{f_s}{s\in\omega}$ with $f=\lim_{s\to\infty} f_s$. Set
\[g(x):=\hbox{least }s\geq x\,(\forall y\leq x)\,f_s(y)=f(y);\]
clearly $f\equiv_T g$. Suppose $g$ is bounded by a recursive function $h$. Set
\[t_x:=(\hbox{least }t\geq x)\;\forall s\,\bigl[t\leq s\leq h(t)\Implies f_t(x)=f_s(x)\bigr].\]
$t_x$ is well-defined since $f_s(x)$ is eventually constant;
$x\mapsto t_x$ is recursive and $t_x\leq g(t_x)\leq h(t_x)$, so
\[f_{t_x}(x)=f_{g(t_x)}(x)=f(x),\]
and thus $f$  is recursive contrary to hypothesis.
\QED
\end{proof}

\begin{lemma} \label{notleqwg}
For every \pzo\ class $P$.
\begin{itemize}
\item[\textup{(i)}] if $P$ has no recursive elements, then there exists $g\in\pre\omega\omega$ such that
\[\zerot<\degt(g)<\zerot'\qand P\not\leqw\set g;\]
\item[\textup{(ii)}] for any $f\in\pre\omega\omega$,  if $\zerot<\degt(f)<\zerot'$ and $P$ has no elements recursive in $f$, then there exists $g\in\pre\omega\omega$ such that
\[\zerot<\degt(g)<\zerot', \quad P\not\leqw\set g\qand \degt(f\oplus g)=\zerot'.\]
\end{itemize}
\end{lemma}

\begin{proof}
Fix $P$ with no recursive elements and a recursive tree $T$ such that $P=[T]$. Let $\Phi_a=\set a$ and set $\tau_0:=\emptyset$; given $\tau_a$, set $\sigma_{a,0}:=\tau_a$ and given $\sigma_{a,i}$; set
\[\sigma_{a,i+1}:\simeq\hbox{least }\sigma\bigl[\sigma_{a,i}\subset\sigma\qand\Phi_a(\sigma_{a,i})\subset\Phi_a(\sigma)\in T\bigr].\]
If $i\mapsto\sigma_{a,i}$ is total, then $h:=\bigcup_{i\in\omega}\sigma_{a,i}$ is recursive and $\Phi_a(h)$ is a recursive element of $P$ contrary to hypothesis. Hence there is a least $i$ such that $\sigma_{a,i+1}\uparrow$; set
$\tau_{a+1}:={\sigma_{a,i}}^\frown(0)$.

Thus $g:=\bigcup_{a\in\omega}\tau_a$ is a total function and $g\leqt\zerot'$ since its definition involves only one-quantifier questions. For any $a$, if $\Phi_a(g)=\bigcup_{a\in\omega}\Phi_a(\tau_a)$ is total, then by construction
\[\lnot\exists\sigma\bigl[\tau_a\subset\sigma\qand\Phi_a(\tau_a)\subset\Phi_a(\sigma)\in T\bigr],\] so $\Phi_a(g)\notin P$. Thus $P\not\leqw\set g$.

For (ii), fix $P=[T]$ and $f$ as in the hypothesis; by Lemma \ref{notdominated} we may assume that $f$ is not recursively bounded. As before, set $\tau_0:=\emptyset$; given $\tau_a$, set $\sigma_{a,0}:=\tau_a$, and given $\sigma_{a,i}$, set
\[\theta_{a,i}(n)\simeq\hbox{least }\sigma\bigl[\sigma_{a.i}^\frown(n)\incl\sigma\qand\Phi_a(\sigma_{a,i})\incl\Phi_a(\sigma)\in T\bigr].\]
Since $\theta_{a,i}$ is partial recursive it not a total function bounding $f$ so there exists
\[n_{a,i}:=\hbox{least }n\bigl[\theta_{a,i}(n)<f(n)\qor\theta_{a,i}(n)\uparrow\bigr].\]
If $\theta_{a,i}(n_{a,i})\downarrow$, set $\sigma_{a,i+1}:=\theta_{a,i}(n_{a,i})$;
otherwise, set $\tau_{a+1}:={\sigma_{a,i}}^\frown(n_{a,i},{\ssf K}(a))$.
The second alternative must occur for some (least) $i_a$, since otherwise $i\mapsto\sigma_{a,i}$ is total, $h:=\bigcup_{i\in\omega}\sigma_{a,i}$ is recursive in $f$ and $\Phi_a(h)$ is a member of $P$ recursive in $f$ contrary to hypothesis.
Set $g:=\bigcup_{a\in\omega}\tau_a$; we conclude that $g\leqt\zerot'$ and $P\not\leqw\set g$ as before. 

Now $\zerot'\leq\degt(\functionof{\tau_a}{a\in\omega})$ because ${\ssf K}(a)=\tau_{a+1}(|\tau_{a+1}|-1)$, and from $f\oplus g$ we can reconstruct this sequence as follows. Given $\tau_a$ and $\sigma_{a,i}$ for some $i<i_a$, $n_{a,i}=g(|\sigma_{a,i}|)$. If
\[\exists \sigma<f(n_{a,i})\bigl[{\sigma_{a,i}}^\frown(n_{a,i})\incl\sigma\qand\Phi_a(\sigma)\in T\bigr],\]
then $i+1<i_a$ and $\sigma_{a,i+1}$ is the least such $\sigma$; otherwise, $i+1=i_a$ and $\tau_{a+1}={\sigma_{a,i}}^\frown(n_{a,i},g(|\sigma_{a,i}|+1))$.
\QED
\end{proof}

\begin{proof}[\ifautoproof\else Proof \fi of Proposition \ref{notdualimplicative}]
Given $Q$, by part (i) of the lemma choose $f$ such that
\[\zerot<\degt(f)<\zerot'\qand Q\not\leqw\set f.\]
Since $\set f\in\Pi^0_2$, by Lemma \ref{existspzo} there exists a \pzo\ class $P$ such that $P\eqw Q\meet\set f$. Suppose that $X\in\pzo$ is such that $Q\leqw P\join X$. Since $P\leqw\set f$ but $Q\not\leqw\set f$, also $X\not\leqw\set f$, so by part (ii) of the lemma we may choose $g$ such that
\[\zerot<\degt(g)<\zerot',\quad X\not\leqw\set g\qand\degt(f\oplus g)=\zerot'.\]
Again by Lemma \ref{existspzo} there exists a \pzo\ class $Y$ such that $Y\eqw X\meet\set g$. Clearly $Y\lew X$. By Theorem D,
\[\set f\join\set g\eqw\set{\zerot'}\geqw\dnr2.\]
Hence, using distributivity,
\begin{align*}
P\join Y&\eqw(Q\meet\set f)\join(X\meet\set g)\\
&=(Q\join X)\meet(Q\join\set g)\meet(\set f\join X)\meet(\set f\join\set g)\\
&\geqw Q\meet Q\meet(P\join X)\meet\dnr2\\
&\eqw Q.\qedhere
\end{align*}
\end{proof}
\end{section}

\begin{section}{\IPC- and \WEM-Completeness Theorems}
The topics of this section are classical results on the relationships between lattices and their theories and do not directly concern the Mu\v cnik or Medvedev degrees but play a large role in the proof of Theorems I and J. As such, they might well be consigned to references to the literature. However, their proofs are somewhat tedious to dig out of that literature, so it seemed worthwhile to include versions here.

For convenience, we call a lattice \latz-\dff{irreducible} iff the least element $\latz$ is meet-irreducible and \lato-\dff{irreducible} iff the greatest element $\lato$ is join-irreducible.

The two theorems to be proved in this section are

\begin{IPCcplthm}
\[\IPC=\;\bigcap\setof{\Th({\mathstr L})}{{\mathstr L}\hbox{ is a 
finite \lato-irreducible implicative lattice}}.\]
\end{IPCcplthm}

\begin{WEMcplthm}[\cite{Jan}]
\[\WEM=\;\bigcap\setof{\Th({\mathstr L})}{{\mathstr L}\hbox{ is a 
finite \latz- and \lato-irreducible implicative lattice}}.\]
\end{WEMcplthm}

The inclusions (${}\incl{}$) are provided by Propositions \ref{IPCinclThL} and \ref{WEMinclThL} respectively, which establish these inclusions without the qualifiers finite or 1-irreducible. Hence the proofs below also establish versions of the Completeness Theorems without one or either of these. To complete the proofs we need, therefore, to show that for each sentence $\phi\notin\IPC$ ($\phi\notin\WEM$) there exists a finite \lato-irreducible (\latz- and \lato-irreducible) lattice ${\mathstr L}_\phi$ and an ${\mathstr L}_\phi$-valuation $v_\phi$ such that $v_\phi(\phi)\not=\lato$.

We begin with the \IPC-Completeness Theorem; in this case the lattices ${\mathstr L}_\phi$ will be finite sublattices of the Lindenbaum lattice ${\mathstr L}_\IPC$, which we proceed to describe. For clarity of exposition we shall sometimes write $\pripc\phi$ ($\prwem\phi$) instead of $\phi\in\IPC$ ($\phi\in\WEM$), particularly when $\phi$ is a long expression.

\begin{definition}
For any propositional sentences $\phi$ and $\psi$,
\begin{align*}
\phi\simeq_\IPC\psi&\dIff\pripc\phi\iff\psi\\
\dgipc\phi&\;:=\setof\psi{\phi\simeq_\IPC\psi}\\
{\mathstr L}_\IPC&\;:=\bigl(\setof{\dgipc\phi}{\phi\in {\ssf PS}},\,\meet,\,\join,\,\latz,\,\lato\bigr),
\end{align*}
where
\begin{align*}
\ipcdg\phi\meet\ipcdg\psi&:=\ipcdg{\phi\land\psi}\\
\ipcdg\phi\join\ipcdg\psi&:=\ipcdg{\phi\lor\psi}\\
\latz&:=\setof\phi{\lnot\phi\in\IPC}=\ipcdg\phi\text{ for any }\phi\text{ such that }\lnot\phi\in\IPC\\
\lato&:=\IPC=\ipcdg\phi\text{ for any }\phi\in\IPC.
\end{align*}
\end{definition}

Of course, the usual verifications are needed here: that ${}\simeq_\IPC{}$ is an equivalence relation and that $\meet$ and $\join$ are well-defined on equivalence classes; we leave these to the reader.

\begin{proposition}
${\mathstr L}_\IPC$ is an implicative lattice.
\end{proposition}

\begin{proof}
We need to check that (omitting the subscript \IPC) for all sentences $\phi$ and $\psi$,
\begin{enumerate} 
\item $\frdg\phi=\frdg\phi\meet\frdg\psi\qIff\frdg\psi=\frdg\phi\join\frdg\psi$;
\item the relation $\phi\leq\psi$ defined by this condition is a partial ordering;
\item $\join$ is the the join (least upper bound) operation and $\meet$ is the the meet (greatest lower bound) operation for this ordering;
\item $\latz$ ($\lato$) is the least (greatest) element for this ordering;
\item there exists an implication operation.
\end{enumerate}
These are all pretty straightforward; for example that $\join$ is a least upper bound requires that 
\begin{gather*}
\pripc\phi\implies\phi\lor\psi\text{\quad and\quad}\pripc\psi\implies\psi\lor\psi;\\
\pripc\phi\implies\theta\text{ and }\pripc \psi\implies\theta\qImplies\pripc\phi\lor\psi\implies\theta.
\end{gather*}
The implication operator is defined in the obvious way:
\[\frdg\phi\latimpl\frdg\psi:=\frdg{\phi\implies\psi}.\]
That this is an implication operator depends on the fact that
\[\pripc\phi\land\theta\implies\psi\qIff\pripc\theta\implies(\phi\implies\psi),\]
which is easy to check.
\QED
\end{proof}

\begin{proposition}
The function $v_\IPC(\phi):=\ipcdg\phi$ is an ${\mathstr L}_\IPC$-valuation.
\end{proposition}

\begin{proof}
The required conditions for $\land$, $\lor$ and $\implies$ are immediate from the definitions. For $\lnot$ we need that
\[v_\IPC(\lnot\phi):=\ipcdg{\lnot\phi}=\latneg\ipcdg\phi=:\ipcdg\phi\latimpl\latz=\ipcdg{\phi\implies({\ssf p}_0\land\lnot{\ssf p}_0)},\]
or equivalently,
\[\pripc\lnot\phi\iff\bigl(\phi\implies({\ssf p}_0\land\lnot{\ssf p}_0)\bigr),\]
which follows easily from the \IPC\ axioms.
\QED
\end{proof}

\begin{corollary}\label{wkIPC}
$\IPC=\;\bigcap\setof{\Th({\mathstr L})}
{{\mathstr L}\hbox{ is a \lato-irreducible implicative lattice}}$
\end{corollary}

\begin{proof}
Clearly for any $\phi\notin\IPC$, $v_\IPC(\phi)\not=\lato$ so $\phi\notin\Th({\mathstr L}_\IPC)$. That ${\mathstr L}_\IPC$ is \hfill\break
1-irreducible is a standard, although not trivial, result about \IPC\ (the \dff{disjunction property})
\[\pripc\phi\lor\psi\qIff\pripc\phi\text{ or }\pripc\psi,\]
(see \cite[Theorem 57, \S80]{Kl} or \cite[XI.6.1]{RaSi}).
\QED
\end{proof}

To complete the proof of the \IPC-Completeness Theorem we shall show that for $\phi\notin\IPC$ there exists a finite sublattice ${\mathstr L}_\phi$ of ${\mathstr L}_\IPC$ and an ${\mathstr L}_\phi$-valuation $v_\phi$ such that $v_\phi(\phi)\not=\lato$. The idea is to in some sense ``generate" ${\mathstr L}_\phi$ from 
\[\setof{\ipcdg\psi}{\psi\text{ is a subsentence of }\phi}.\]
Using a finite set to generate a finite substructure of a distributive lattice or a Boolean algebra, which is also a distributive lattice or a Boolean algebra, is a simple and familiar process. For example, if ${\mathstr B}=(B,\meet,\join,\latneg,\latz,\lato)$ is a Boolean algebra and $A$ is a finite subset of $B$, then we can describe a finite Boolean subalgebra of ${\mathstr B}$ as follows. For $U\incl A\cup\set{\latz,\lato}$, set 
\[U^{\join}:=\bigjoin\setof{\latneg a}{a\in A\setminus U}\join\bigjoin U\]
and for ${\cal U}\incl\power(A\cup\set{\latz,\lato})$, ${\cal U}^{\meet\join}:=\bigmeet\setof{U^{\join}}{U\in{\cal U}}$. Then it follows from the distributive and DeMorgan laws that 
$B_A:=\setof{{\cal U}^{\meet\join}}{{\cal U}\incl\power(A\cup\set{\latz,\lato})}$ is closed under $\meet$, $\join$ and $\latneg$ and with the restrictions of these operations is a Boolean subalgebra ${\mathstr B}_A$ of $\mathstr B$. Clearly $B_A$ includes $A$ and is finite with at most $2^{2^{|A|+2}}$ elements. This is just the construction of the \dff{conjunctive normal form} in propositional logic. A similar construction using disjunctive normal form works the same way.

For a finite subset $A$ of a distributive lattice ${\mathstr L}$, we may similarly construct a finite sublattice ${\mathstr L}_A$ by removing reference to the operation $\latneg$ and relying on the distributive laws. However, in the case at hand, ${\mathstr L}$ is also implicative and the desired finite sublattice must also be implicative with an implication closely enough related to that of ${\mathstr L}$ to achieve the result $v_\phi(\phi)\not=\lato$. In the cases described above, we  simply closed the set $A$ under the operations $\meet$, $\join$ and in the case of a Boolean algebra also $\latneg$ and could achieve this in ``one step", thus preserving finiteness. However, a parallel attempt to close a set $A$ also under $\latimpl$ does not succeed in one step because there are in general no distributive laws relating $\meet$ and $\join$ to $\latimpl$, so we would seemingly need to iterate the map $(a,b)\mapsto a\latimpl b$ infinitely often to reach a set that is closed. This would in general fail to produce a finite sublattice. The solution to this problem we outline here is essentially that presented in \cite{RaSi} collecting the precursors to the proof of IX.3.1 of that text.

\begin{definition}\label{interior}
For any Boolean algebra ${\mathstr B}=(B,\meet,\join,\latneg,\latz,\lato)$, an \dff{interior operator} on $\mathstr B$ is a function $I:B\to B$ such that for all $a,b\in B$,
\begin{enumerate}
\item $I(a\meet b)=I(a)\meet I(b)$;
\item $I(a)\leq a$;
\item $I(I(a))=I(a)$;
\item $I(\lato)=\lato$.
\end{enumerate}
\end{definition}

\begin{remark}
For a topological space $(T,{\cal O})$, the (topological) interior operator defined by
\[I(X):=\bigcup\setof{A\in{\cal O}}{A\incl X}\]
is an interior operator (in the current sense) on the Boolean algebra 
\[(\power(T),\cup,\cap,-,\emptyset,T).\]
Conversely, for any interior operator on this Boolean algebra, $(T,{\cal O}_I)$, where
\[{\cal O}_I:={\ssf Im}_I(\power(T)):=\setof{I(X)}{X\incl T},\]
is a topological space. The following lemma is the version of this that we need here.
\end{remark}

\begin{lemma}\label{intimpl}
For any Boolean algebra ${\mathstr B}=(B,\meet,\join,\latneg,\latz,\lato)$ and any interior operator $I$ on $\mathstr B$, 
\[{\cal O}_I({\cal B}):=({\ssf Im}_I(B),\meet,\join,\latz,\lato)\]
is an implicative lattice with implication operator
\[I(a)\latimpl_II(b):=I\bigl(\latneg I(a)\join I(b)\bigr).\]
\end{lemma}

\begin{proof}
Fix ${\mathstr B}$, $I$ and ${\cal O}_I({\cal B})$ as in the hypothesis. Clearly ${\ssf Im}_I(B)$ contains $\latz$ and $\lato$ by (ii) and (iv) of the definition and is closed under $\meet$ by (i). For closure under $\join$, note first that for any $a,b\in B$, using (i),
\begin{align*} 
a\leq b\Iff a=a\meet b&\Implies I(a)=I(a\meet b)=I(a)\meet I(b)\tag{v}\\
&\Implies I(a)\leq I(b).
\end{align*} 
Hence by (iii), $I(a)=I(I(a))\leq I(I(a)\join I(b))$ and similarly for $b$ so by (ii),
\[I(a)\join I(b)\leq I(I(a)\join I(b))\leq I(a)\join I(b),\]
and $I(a)\join I(b)=I(I(a)\join I(b))\in{\ssf Im}_I(B)$. 

Finally we verify that $\latimpl_I$ is an implication --- that is, for any $a,b,x\in B$,
\[I(a)\meet I(x)\leq I(b)\qIff I(x)\leq I(\latneg I(a)\join I(b)).\]
For (\Rarrow) we have by (v),
\begin{align*} 
I(a)\meet I(x)\leq I(b)&\Implies I(x)\leq \latneg I(a)\join I(b)\\
&\Implies I(x)=I(I(x))\leq I(\latneg I(a)\join I(b)).
\end{align*} 
For (\Larrow) it suffices to show that
\[I(a)\meet I(\latneg I(a)\join I(b))\leq I(b).\]
By (iii) and (i) the left-hand side is
\begin{align*}
I(I(a))\meet I(\latneg I(a)\join I(b))&=I((I(a)\meet\latneg I(a))\join(I(a)\meet I(b)))\\
&=I(\latz\join(I(a)\meet I(b)))\\
&\leq I(I(b))=I(b)
\end{align*}
using distribution, (v) and again (iii).
\QED
\end{proof}

\begin{lemma}
For any implicative lattice ${\mathstr L}=(L,\meet,\join,\latz,\lato)$ with implication $\latimpl$, there exists a Boolean algebra  ${\mathstr B}=(B,\meet,\join,\latneg,\latz,\lato)$ and an interior operator $I$ on $\mathstr B$ such that ${\mathstr L}={\cal O}_I({\mathstr B})$ and $\latimpl$ coincides with $\latimpl_I$.
\end{lemma}

\begin{proof}
Given ${\mathstr L}$ and $\latimpl$ we shall construct $\mathstr B$ and $I$ such that ${\mathstr L}\iso{\cal O}_I({\mathstr B})$ under an isomorphism that maps $\latimpl$ onto $\latimpl_I$; the statement as given follows by a standard procedure. We define first a larger Boolean algebra $\mathstr C$ as follows. A \dff{filter} on $\mathstr L$ is any non-empty proper subset $\Delta\subset L$ such that for $a,b\in L$,
\[a\in\Delta\text{ and }a\leq b\Implies b\in\Delta\qand a,b\in\Delta\Implies a\meet b\in\Delta.\]
A filter $\Delta$ is \dff{prime} iff additionally
\[a\join b\in\Delta\qIff a\in\Delta\text{ or }b\in\Delta.\]
Let {\ssf PF} denote the set of all prime filters on $\mathstr L$ and set
\[C:=\power({\ssf PF})\qand {\mathstr C}:=(C,\cap,\cup,-,\emptyset ,{\ssf PF}).\]
Define an embedding $\eta:L\to C$ by
\[\eta(a):=\setof{\Delta\in{\ssf PF}}{a\in\Delta}.\]
Easily $\eta(\latz)=\emptyset$ (since $\latz\in\Delta\Implies \Delta=L$) and $\eta(\lato)={\ssf PF}$. It also follows easily from the properties of prime filters that
\[\eta(a\meet b)=\eta(a)\cap\eta(b)\qand \eta(a\join b)=\eta(a)\cup\eta(b).\]
To see that $\eta$ is injective, suppose that $a\not=b$, say $a\not\leq b$, and consider
\[D:=\setof{\Delta}{\Delta\text{ is a filter on }L,\ a\in\Delta\text{ and }b\notin \Delta}.\]
$\Delta\not=\emptyset$ since $\Delta_a:=\setof b{a\leq b}\in D$. Easily the union of a chain of filters is a filter, so by Zorn's Lemma, $D$ has a maximal element $\overline\Delta$. $\overline\Delta$ is prime, since if $c_0\join c_1\in\overline\Delta$ but $c_0\not\in\overline\Delta$ and $c_1\not\in\overline\Delta$, then consider the filters (i=0,1)
\[\Delta_i:=\setof e{(\exists d\in\overline\Delta)
;c_i\meet d\leq e}.\]
Both $\Delta_0$ and $\Delta_1$ properly extend $\overline\Delta$ so do not belong to $D$ and thus $b$ belongs to both, say $c_i\meet d_i\leq b$. Then $(c_0\join c_1)\meet(d_0\meet d_1)\leq b$, so $b\in\overline\Delta$ contrary to the choice of $\overline\Delta\in D$. Hence $\overline\Delta\in\eta(a)\setminus\eta(b)$ so $\eta(a)\not=\eta(b)$.

Let $M:={\ssf Im}_\eta(L)$ and ${\mathstr M}:=(M,\cap,\cup,\emptyset,{\ssf PF})$; $\eta$ is an isomorphism ${\mathstr L}\iso{\mathstr M}$ and carries $\latimpl$ onto a relation $\latimpl_M$ that makes $\mathstr M$ an implicative lattice, since $\latimpl$ is definable from $\meet$. Now we define $\mathstr B$ to be the Boolean subalgebra of $\mathstr C$ generated by $M$. This could be described as in the discussion preceding Definition \ref{interior} (replacing $U$ and $A\setminus U$ by pairs of finite subsets of $M$), but using the additional information here that $\mathstr M$ is a lattice gives us the simpler definition ${\mathstr B}:=(B,\cap,\cup,-,\emptyset,{\ssf PF})$ where
\[B:=\setof{\bigcap_{i<k}(-u_i\cup v_i)}{k\in\omega\text{ and }(\forall i<k)\;u_i,v_i\in M}.\]

We define a function $I:B\to M$ by
\[I\Bigl(\bigcap_{i<k}(-u_i\cup v_i)\Bigr):=\bigcap_{i<k}(u_i\latimpl_M v_i).\]
We need first to verify that this is well-defined. Note that for $u,v\in M$, $u\cap(u\latimpl_Mv)\incl v$ (because $\latimpl_M$ is an implication) so
\[u\latimpl_Mv\incl-u\cup v.\tag{*}\]
Next we have for $u,u_i,v\text{ and }v_i\in M$,
\begin{align*}
\bigcap_{i<k}(-u_i\cup v_i)\incl -u\cup v
&\Implies\bigcap_{i<k}(-u_i\cup v_i)\cap u\incl v\\
&\Implies\bigcap_{i<k}(u_i\latimpl_M v_i)\cap u\incl v\\
&\Implies\bigcap_{i<k}(u_i\latimpl_M v_i)\incl u\latimpl_M v.
\end{align*}
It follows immediately that for $u_i,v_i,u'_i,v'_i\in M$,
\[\bigcap_{i<k}(-u_i\cup v_i)\incl\bigcap_{i<k'}(-u'_i\cup v'_i)\Implies
\bigcap_{i<k}(u_i\latimpl_M v_i)\incl\bigcap_{i<k'}(u'_i\latimpl_M v'_i),\]
so
\[\bigcap_{i<k}(-u_i\cup v_i)=\bigcap_{i<k'}(-u'_i\cup v'_i)\Implies
\bigcap_{i<k}(u_i\latimpl_M v_i)=\bigcap_{i<k'}(u'_i\latimpl_M v'_i),\]
which is exactly the statement that $I$ is well-defined.
Now ${\ssf Im}_I(B)\incl M$ (since $M$ is closed under $\latimpl_M$) and for $v\in M$,
\[I(v)=I(-{\ssf PF}\cup v)={\ssf PF}\latimpl_Mv=v,\]
so $M={\ssf Im}_I(B)$ and $I$ has property (iii) of Definition \ref{interior}. Property (i) is immediate and property (ii) follows from (*). $I$ has property (iv) because
\[I({\ssf PF})=I(-{\ssf PF}\cup {\ssf PF})={\ssf PF}\latimpl_M{\ssf PF}={\ssf PF}.\]
Thus $I$ is an interior operator and ${\mathstr M}={\cal O}_I({\mathstr B})$. That $\latimpl_M$ coincides with $\latimpl_I$ is immediate from the definitions.
\QED
\end{proof}

\begin{proposition}
For any implicative lattice ${\mathstr L}=(L,\meet,\join,\latz,\lato)$ with implication $\latimpl$ and any finite set $A\incl L$, there exists a finite set $L_A\incl L$ and an operator $\latimpl_A$ on $L_A$ such that $A\cup\set{\latz,\lato}\incl L_A$ and
\begin{enumerate}
\item${\mathstr L}_A:=(L_A,\meet,\join,\latz,\lato)$ is a sublattice of $\mathstr L$;
\item ${\mathstr L}_A$ is an implicative lattice with implication $\latimpl_A$;
\item for all $a,b\in A\cup\set{\latz,\lato}$,
\[a\latimpl b\in A\qImplies a\latimpl_A b=a\latimpl b.\]
\end{enumerate}
\end{proposition}

\begin{proof}
Fix $\mathstr L$, $\latimpl$ and $A$ as in the hypothesis. By the preceding lemma there exists a Boolean algebra $\mathstr B$ and an interior operator $I$ in $\mathstr B$ such that $\mathstr L={\cal O}_I({\mathstr B})$ and $\latimpl$ coincides with $\latimpl_I$. Let ${\mathstr B}_A$ be the finite Boolean subalgebra of $\mathstr B$ defined as in the discussion preceding Defnition \ref{interior}. For $u\in B_A$, set
\[J(u):=\bigjoin\setof{\bigmeet W}{W\incl A\cup\set{\latz,\lato}\text{ and }\bigmeet W\leq u}.\]
It is easy to check that $J$ is an interior operator on ${\mathstr B}_A$. Since $J(u)\in L$, $I(J(u))=J(u)$, so from $J(u)\leq u$ we deduce that $J(u)\leq I(u)$. On the other hand, if $I(u)\in  A\cup\set{\latz,\lato}$, since $I(u)\leq u$, we have $I(u)\leq J(u)$. Thus
\[(\forall u\in B_A)\;[\;I(u)\in A\cup\set{\latz,\lato}\qImplies I(u)=J(u)].\tag{**}\]
Set $L_A:={\ssf Im}_J(B_A)$. Then easily ${\mathstr L}_A={\cal O}_J({\mathstr B}_A)$ and is a finite sublattice of $\mathstr L$. As in the proof of Lemma \ref{intimpl}, ${\mathstr L}_A$ is implicative with implication
\[J(u)\latimpl_AJ(v):=J(-J(u)\join J(v)).\]
For any $u,v\in A\cup\set{\latz,\lato}$, if
\[J(u)\latimpl J(v)=I(-J(u)\join J(v))\in A\cup\set{\latz,\lato},\]
then by (**), $J(u)\latimpl_A J(v)=J(u)\latimpl J(v)$ as desired.
\QED
\end{proof}

\begin{proof}[\ifautoproof\else Proof \fi of the \IPC-Completeness Theorem]
As in the proof of Corollary \ref{wkIPC} and the following discussion, choose $\phi\notin\IPC$ and set
\[A_\phi:=\setof{\ipcdg\psi}{\psi\text{ is a subsentence of }\phi}.\]
$A_\phi$ is a finite subset of $L_\IPC$. Let ${\mathstr L}_\phi$ and $\latimpl_\phi$ be as in the proposition and $v_\phi$ the unique ${\mathstr L}_\phi$-valuation such that for atomic propositional sentences {\ssf p}.
\[v_\phi({\ssf p})=
\cases{
\ipcdg{\ssf p},&\text{if $\ipcdg{\ssf p}\in A_\phi$;}\cr
\noalign{\smallskip}
\lato_\IPC,&\text{otherwise.}\cr
}\]
A straightforward induction, using (iii) of the proposition, shows that for all subsentences $\psi$ of $\phi$,
$v_\phi(\psi)=v_\IPC(\psi)$. In particular, $v_\phi(\phi)\not=\lato_\IPC$ and hence $\phi\notin\Th({\mathstr L}_\phi)$.
\QED
\end{proof}

We begin now on the proof of the \WEM-Completeness Theorem.

\begin{definition}
For any bounded lattice $\l=(L.\leq,\meet,\join,\latz,\lato)$, \lz, \lo\ and \lzo\ denote the lattices that extend \l\ by adjoining a new least element $0^*$, a new greatest element $1^*$ or both.
\end{definition}

\begin{lemma}\label{lzo}
For any lattice \l,
\begin{enumerate}
\item \lz, \lo\ and \lzo\ are lattices;
\item \lz\ and \lzo\ are \latz-irreducible; \lo\ and \lzo\ are \lato-irreducible;
\item if \l\ is \lato-irreducible, so is \lz; if \l\ is \latz-irreducible, so is \lo;
\item if \l\ is (dual-) implicative, so are  \lz, \lo\ and \lzo.
\end{enumerate}
\end{lemma}

\begin{proof}
Clearly for $a,b\in L$, $a\meet b$ and $a\join b$ have the same values in each of the extensions and we have in the relevant extensions
\begin{align*}
a\join \latz^*&=a&a\meet \latz^*&=\latz^*\\
a\join\lato^*&=\lato^*&a\meet\lato^*&=a.
\end{align*}
Since for $a,b\not=\latz^*,\lato^*$ also $a\meet b, a\join b\in L$, they are not equal to $\latz^*$ or $\lato^*$ and thus $\latz^*$ is meet-irreducible and $\lato^*$ is join-irreducible.
Suppose now that \l\ is implicative so for all $a,b,x\in L$,
\[a\meet x\leq b\qIff x\leq a\latimpl b.\]
For $\l^*$ each of  \lz, \lo\ and \lzo, we need to define $\latimpl^*$ such that for all $a,b,x\in L^*$, 
\[a\meet x\leq b\qIff x\leq a\latimpl^* b.\]
We leave it to the reader to verify that the following definition suffices: for $a,b\in L$,
\[a\latimpl^*b:=\cases{\lato^*,&if $a\leq b$;\cr a\latimpl b,&otherwise;\cr}\]
\begin{align*}
\latz^*\latimpl^*b&=\lato\text{ or }\lato^*&a\latimpl^*\latz^*&=\latz^*&\latz^*\latimpl^*\latz^*&=\lato\text{ or }\lato^*\\
\lato^*\latimpl^*b&=b&a\latimpl^*\lato^*&=\lato^*&\lato^*\latimpl^*\lato^*&=\lato^*.
\end{align*} 
Of course, the calculations for dual-implication are dual!
\QED
\end{proof}

For any propositional sentence $\phi$, $\At(\phi)$ will denote the finite set of atomic sentences that occur in $\phi$. $\phi$ is called \dff{positive} iff the negation symbol $\lnot$ does not occur in $\phi$. We write $\chi\pripc\phi$ iff $\chi\implies\phi\in\IPC$.

\begin{lemma}\label{zirred}
For any implicative lattice \l, and $\l^*$ any of the extensions above, for any positive sentence $\phi$, any $\mathstr L$-valuation $v$ and any $\l^*$-valuation $w$, if $v(\pp)=w(\pp)$ for all $\pp\in\At(\phi)$, then $v(\phi)=w(\phi)$. 
\end{lemma}

\begin{proof}
This is immediate from the fact that for $a,b\in L$, all of the operations have the same values in $\l^*$ as in \l. Of course, the restriction to positive sentences is crucial, since generally
$\latneg^*a\not=\latneg a$.
\QED
\end{proof}

\begin{lemma}
For any sentence $\phi$, any $S\supseteq\At(\phi)$, any $X\incl S$ and
\[\chi_ X:=\lAnd_{\pp\in X}\lnot \pp\land\lAnd_{\pp\in S\setminus X}\lnot\lnot \pp,\]
one of the following holds:
\begin{enumerate}
\item $\chi_X\pripc\phi$;
\item $\chi_X\pripc\lnot\phi$;
\item $\chi_X\pripc\lnot\lnot\phi$ and there exists a positive sentence $\phi^X$ such that
\[\At(\phi^X)\incl S\setminus X\qand\chi_X\pripc\phi\iff\phi^X.\]
\end{enumerate}
\end{lemma}

\begin{proof}
We proceed by sentence induction and write $\prx\psi$ for $\chi_X\pripc\psi$. If $\phi$ is the atomic sentence $\pp$, then
\begin{align*}
\pp\in X&\qImplies\prx\lnot\phi;\\
\pp\notin X&\qImplies\prx\lnot\lnot\phi\qand\phi^X:=\phi\text{ is positive}.
\end{align*}
Suppose next that $\phi$ is $\psi\land\theta$ and assume as induction hypothesis that one of (i)--(iii) holds for each of $\psi$ and $\theta$. Of the nine resulting cases, six are immediate:
\begin{align*}
\prx\psi\qand\prx\theta&\qImplies\prx\phi;\\
\prx\lnot\psi\qor\prx\lnot\theta&\qImplies\prx\lnot\phi.
\end{align*}
If $\prx\psi$ while $\prx\lnot\lnot\theta$ and $\prx\theta\iff\theta^X$, then $\prx\phi\iff\theta$ so $ \prx\lnot\lnot\phi$ and $\prx\phi\iff\theta^X=:\phi^X$. The case with $\psi$ and $\theta$ interchanged is parallel. Finally if (iii) holds for both $\psi$ and $\theta$, then easily $\prx\lnot\lnot\phi$ and $\phi\iff\psi^X\land\theta^X=:\phi^X$, so case (iii) holds also for $\phi$.

The corresponding cases for $\phi=\psi\lor\theta$ are dual and are left to the reader. If $\phi$ is $\lnot\psi$, then the three cases for $\psi$ lead easily to $\prx\lnot\phi$, $\prx\phi$ and $\prx\lnot\phi$, respectively. Finally, suppose that $\phi$ is $\psi\implies\theta$. Again, six of the nine cases are immediate:
\begin{align*}
\prx\lnot\psi\qor\prx\theta&\qImplies\prx\phi;\\
\prx\psi\qand\prx\lnot\theta&\qImplies\prx\lnot\phi.
\end{align*}
If $\prx\psi$ while $\prx\lnot\lnot\theta$ and $\prx\theta\iff\theta^X$, then $\prx\phi\iff\theta$ so $ \prx\lnot\lnot\phi$ and $\prx\phi\iff\theta^X=:\phi^X$. If $\prx\lnot\theta$ while $\prx\lnot\lnot\psi$ and $\prx\psi\iff\psi^X$, then $\prx\phi\iff\lnot\psi$ so $ \prx\lnot\phi$.  Finally, if case (iii) holds for both $\psi$ and $\theta$, then case (iii) holds also for $\phi$ with $\phi^X:=(\psi^X\implies\theta^X)$.
\QED
\end{proof}

\begin{proof}[\ifautoproof\else Proof \fi of the \WEM-Completeness Theorem]
We shall show that for each propositional sentence $\phi$, either $\phi\in\WEM$ or there exists a finite \latz- and \lato-irreducible implicative lattice $\mathstr L$ and an $\mathstr L$-valuation $v$ such that $v(\phi)\not=\lato$. Fix $\phi$ and suppose first that $\chi_X\pripc\phi$ for all $X\incl\At(\phi)$. Then
\[\pripc\lOr_{X\incl\At(\phi)}\chi_X\implies\phi.\]
But $\displaystyle\prwem\lAnd_{\pp\in\At(\phi)}(\lnot \pp\lor\lnot\lnot \pp)$, so by the distributive law 
\[\prwem\lOr_{X\incl\At(\phi)}\chi_X\hbox{\qquad and thus\qquad }\prwem\phi.\]
Otherwise, we may fix $X\incl\At(\phi)$ such that $\chi_X\not\pripc\phi$, so one of cases (ii) or (iii) of the lemma holds for $X$ and $\phi$. If Case (ii) holds, let $\mathstr L$ be the 2-element lattice \set{\latz,\lato} and $v$ the valuation such that $v(\pp)=\latz$ for $\pp\in X$ and $v(\pp)=\lato$ for $\pp\notin X$. Then $v(\chi_X)=\lato$, so since $\chi_X\pripc\lnot\phi$, also $v(\lnot\phi)=\lato$ and thus $v(\phi)\not=\lato$. 

Finally, suppose that case (iii) holds for $X$ and $\phi$ so there exists a positive sentence $\phi^X$ such that 
\[\At(\phi^X)\incl \At(\phi)\setminus X\qand\chi_X\pripc\phi\iff\phi^X.\]
(The property $\chi_X\pripc\lnot\lnot\phi$ is not needed here and was present only to carry through the induction.) Since $\chi_X\not\pripc\phi$, also $\chi_X\not\pripc\phi^X$ so in particular $\not\pripc\phi^X$. Hence, by the \IPC-Completeness Theorem, there exists a finite \lato-irreducible implicative lattice $\mathstr L$ and an  $\mathstr L$-valuation $v$ such that $v(\phi^X)\not=\lato$. Let $\mathstr L^*=\lz$ be as in Lemma \ref{lzo} and $w$ the $\mathstr L^*$-valuation
\[w(\pp):=\cases{\latz^*,&for $\pp\in X$;\cr v(\pp),&for $\pp\notin X$.\cr}\]
Then
\begin{align*}
\pp\in X&\qImplies w(\lnot \pp)=\latz^*\latimpl^*\latz^*=\lato,\qand\\
\pp\notin X&\qImplies w(\lnot\lnot \pp)=(v(\pp)\latimpl^*\latz^*)\latimpl^*\latz^*=\latz^*\latimpl^*\latz^*=\lato,
\end{align*}
so $w(\chi_X)=\lato$. Since $\phi^X$ is positive and $w(\pp)=v(\pp)$ for all $\pp\in\At(\phi^X)$, $w(\phi^X)=v(\phi^X)\not=\lato$. But since $\chi_X\pripc\phi\implies\phi^X$, $w(\phi)=w(\chi_X\land\phi)\leq w(\phi_X)$, also $w(\phi)\not=\lato$.
\QED
\end{proof}
\end{section}

\begin{section}{Proof of Theorem I}
In this section we establish that $\Th(\Dgs^\dualmarker)=\WEM$. For the corresponding results for $\Th(\Dgw)$ and $\Th(\Dgw^\dualmarker)$ we refer the reader to \cite{SorTe1}. 

In the preceding section, we considered embeddings of one lattice $\mathstr L$ into another $\mathstr K$ that respected the lattice structure but generally did not respect the implication operator (when it exists) --- if both $\mathstr L$ and $\mathstr K$ are implicative, then so is the image of $\mathstr L$, but its implication may not be simply the restriction of the implication of $\mathstr K$. Here we shall be considering embeddings that do respect implication and we introduce the notation ${\mathstr L}\embeds{\mathstr K}$ to denote the existence of such an embedding.

\begin{lemma}
For any implicative lattices $\mathstr L$ and $\mathstr K$, if ${\mathstr L}\embeds{\mathstr K}$, then $\Th({\mathstr K})\incl \Th ({\mathstr L})$.
\end{lemma}

\begin{proof}
If ${\mathstr L}\iso{\mathstr L}'\incl{\mathstr K}$, the isomorphism respects implication and the implication of ${\mathstr L}'$ agrees with that of $\mathstr K$, then every $\mathstr K$-valuation is an ${\mathstr L}'$-valuation, so if $v(\phi)=\lato$ for every $\mathstr K$-valuation, then also $v(\phi)=\lato$ for every ${\mathstr L}'$-valuation and hence for every $\mathstr L$-valuation.
\QED
\end{proof}

 The main result of this section is the

\begin{embedthm}[{\cite[Theorem 2.6]{Sor1}}] For every finite \latz- and \hfill\break
\lato-irreducible implicative lattice $\mathstr L$,  ${\mathstr L}\embeds\Dgs^\dualmarker$.
\end{embedthm}

\begin{proof}[\ifautoproof\else Proof \fi of Theorem I (for $\Dgs^\dualmarker$)]
From these two results we have that for each finite \latz- and \lato-implicative lattice $\mathstr L$, $\Th(\Dgs^\dualmarker)\incl\Th({\mathstr L})$, so by the \hfill\break
\WEM-Completeness Theorem, $\Th(\Dgs^\dualmarker)\incl\WEM$. 
\QED
\end{proof}

It will be more convenient to deal directly with $\Dgs$ rather than its dual. For dual-implicative lattices 
$\mathstr L$ and $\mathstr K$, we write ${\mathstr L}\dembeds{\mathstr K}$ iff there is an injective function $\eta:L\to K$ which respects \latz, \lato, $\leq$, $\meet$, $\join$ and $\dlatimpl$. Then directly from the definitions
\begin{align*}
{\mathstr L}\dembeds{\mathstr K}&\qIff{\mathstr L}^\dualmarker\embeds{\mathstr K}^\dualmarker\qand\\
{\mathstr L}\text{ is } \latz\text{- and } \lato&\text{-irreducible and dual-implicative}\hskip 1.5in\\
&\qIff{\mathstr L}^\dualmarker\text{ is } \latz \text{- and } \lato\text{-irreducible and implicative},
\end{align*}
so it is equivalent to prove the

\begin{dembedthm}For every finite \latz- and \lato-irreducible \hfill\break
dual-implicative lattice $\mathstr L$,  ${\mathstr L}\dembeds\Dgs$.
\end{dembedthm}

The key tool in establishing this is the notion of a free lattice $\free(\calp)$ generated by a partial ordering \calp. Given \l\ as above we shall show that there exists a finite partial ordering \calp\ such that
\[\l\qdembeds\free(\calp)^{\ssf 1}_{\ssf 0}\qdembeds\free(\Dgt)^{\ssf 1}_{\ssf 0}\qdembeds\Dgs.\]
Here \Dgt\ is considered just as a partial ordering ignoring its join operation.

\begin{definition}
For any partial ordering $\calp=(P,\leq_P)$, the \dff{free lattice} $\free(\calp)$ is defined as follows. For $\cals,\calt\in\poweromm\poweromm(P)$ [finite non-empty sets of finite non-empty subsets of $P$], set
\begin{align*}
\cals\leq\calt&\dIff(\forall A\in\cals)(\exists B\in\calt)\;\set A\leq\set B\\
\noalign{\noindent where
${\set A}\leq{\set B}\dIff(\forall b\in B)(\exists a\in A)\;a\leq_Pb$;\smallskip}
\cals\equiv\calt&\dIff\cals\leq\calt\qand\calt\leq\cals;\\
\noalign{\smallskip}
\frdgs&\;:=\setof{\calt}{\cals\equiv\calt}.
\end{align*}
The elements of $\free(\calp)$ are all $\frdgs$ with the inherited partial ordering
\[\frdgs\leq\frdgt\qIff\cals\leq\calt.\]
We define the lattice operations by:
\begin{align*}
\cals\join\calt&:=\cals\cup\calt;\\
\noalign{\smallskip}
\cals\meet\calt&:=\setof{A\cup B}{A\in\cals\qand B\in\calt};\\
\noalign{\smallskip}
\cals\dlatimpl\calt&:=\setof{B\in\calt}{\set B\not\leq\cals}.
\end{align*}
Then, with the justification below, for each of these operations 
\[\frdgs*\frdgt:=\frdg{\cals*\calt}.\]
\end{definition}

\begin{remark}
Intuitively, $\set A$ represents $\bigmeet A$ and $\cals$ represents $\bigjoin_{A\in\cals}\bigmeet A$. $a\mapsto\frdg{\set{\set a}}$ is an order-preserving embedding $\calp\to\free(\calp)$. Many simple properties of $\free(\calp)$ that we shall use without explicit mention stem naturally from this intuition (or directly from the definition) --- for example,
\begin{align*}
A\incl B&\qImplies\set B\leq\set A;\\
\cals\incl\calt&\qImplies\cals\leq\calt;\\
a,b\in A\qand a<_Pb&\qImplies\set A\equiv\set{A\setminus\set b};\\
A,B\in\cals\qand A\incl B&\qImplies\cals\equiv\cals\setminus\set B.
\end{align*}
\end{remark}

\begin{lemma}
For any finite partial order \calp, $\free(\calp)$ is a dual-implicative lattice.
\end{lemma}

\begin{proof}
This is just a (somewhat tedious) verification that things work as expected and we will write out only a few cases.
\[\cals\leq\cals'\text{ and }\calt\leq\calt'\qImplies\cals\meet\calt\leq\cals'\meet\calt'\]
because if for each $A\in\cals$ there is $A'\in\cals'$ such that $\set A\leq\set{A'}$ and similarly for $B, B'\in\calt,\calt'$, then for a typical $A\cup B\in\cals\meet\cals'$, $\set{A\cup B}\leq\set{A'\cup B'}$.

For any $\calx\in\poweromm\poweromm(P)$,
\begin{align*}
\calt\leq\cals\join\calx
&\Iff(\forall B\in\calt)\bigl[(\exists A\in\cals)\set B\leq\set A\text{ or }(\exists A\in\calx)\set B\leq\set A\bigr]\\
&\Iff(\forall B\in\calt)\bigl[\set B\not\leq\cals\Implies\set B\leq\calx\bigr]\\
&\Iff(\forall B\in\calt)\bigl[B\in\cals\dlatimpl\calt\Implies\set B\leq\calx\bigr]\\
&\Iff\cals\dlatimpl\calt\leq\calx.
\end{align*}

Since $P$ is assumed finite, we have also least and greatest elements
\[\latz=\bigset{\set P}\qand\lato=\bigsetof{\set a}{a\in P}.\qedhere\]\\
\end{proof}

\vskip-25pt
The following lemma is a somewhat messy technical tool needed in the proof that every finite dual-implicative lattice can be embedded in some $\free(\calp)$

\begin{lemma}\label{cprodpos}
For any finite partial orderings $\calq=(Q,\leq_Q)$ and $\calr=(R,\leq_R)$, there exists a finite partial ordering $\calp=(P,\leq_P)$ such that
\[\free(\calq)\times\free(\calr)\dembeds\free(\calp).\]
\end{lemma}

\begin{proof}
Note first that as usual we endow the cartesian product with dual-implicative lattice structure by defining everything component-wise. We set
\begin{gather*}
P:=(\set 0\times Q)\cup(\set 1\times R);\\
\noalign{\smallskip}
(i,a)\leq_P(j,c)\dIff(i=j=0\text{ and }a\leq_Q c)\text{ or }(i=j=1\text{ and }a\leq_R c);\\
\noalign{\smallskip}
\eta(\cals,\calu):=\setof{A^0}{A\in\cals}\cup\setof{C^1}{C\in\calu},\\
\noalign{where}
A^0:=(\set 0\times A)\cup(\set 1\times R)\qand C^1:=(\set 0\times Q)\cup(\set 1\times C).
\end{gather*}
We aim to show that $\eta$ is well-defined on equivalence classes and determines a dual-implicative embedding of $\free(\calq)\times\free(\calr)$ into $\free(\calp)$. The idea is that among the many copies of \calq\ in $\free(\calp)$ --- for example, $q\mapsto \bigset{\set{(0,q)}\cup(\set 1\times C)}$ for any $C\incl R$ ---  we choose the one with $C=R$ and analogously for \calr. This choice ensures the following facts for all $A,B\in\poweromm(Q)$, $C,D\in\poweromm(R)$, $\calt\in\poweromm\poweromm(Q)$ and $\calv\in\poweromm\poweromm(R)$.
\begin{align*}
\set{A^0}\leq\set{B^0}&\qIff\set{A}\leq\set{B}\tag{1}\\
\set{C^1}\leq\set{D^1}&\qIff\set{C}\leq\set{D}\\
\noalign{\smallskip}
\set{A^0}\leq\set{C^1}&\qIff\set{A}\equiv\set{Q}\tag{2}\\
\set{C^1}\leq\set{A^0}&\qIff\set{C}\equiv\set{R}\\
\noalign{\smallskip}
\set{A^0}\leq\eta(\calt,\calv)&\qIff\set{A}\leq\calt\tag{3}\\
\set{C^1}\leq\eta(\calt,\calv)&\qIff\set{C}\leq\calv
\end{align*}
(1) and (2) are immediate from the definitions --- note that always $\set Q\leq\set A$ and $\set R\leq\set C$. For (3),
\begin{align*}
\set{A^0}\leq\eta(\calt,\calv)&\qIff(\exists B\in\calt)\set{A^0}\leq\set{B^0}\\
&\hskip1in\qor(\exists D\in\calv)\set {A^0}\leq\set{D^1}\\
&\qIff(\exists B\in\calt)\set A\leq\set B\qor\set A\equiv\set Q\\
&\qIff\set A\leq\calt,
\end{align*}
since for all $B\in\calt$, $\set Q\leq\set B$. The second clause of (3) is similar.

Next we have
\begin{align*}
\eta(\cals,\calu)\leq\eta(\calt,\calv)&\qIff(\forall A\in\cals)\set{A^0}\leq\eta(\calt,\calv)\\
&\hskip1in\qand(\forall C\in\calu)\set{C^1}\leq\eta(\calt,\calv)\\
&\qIff(\forall A\in\cals)\set A\leq\calt\qand(\forall C\in\calu)\set C\leq\calv\\
&\qIff\cals\leq\calt\qand\calu\leq\calv.
\end{align*}
It follows that $\eta$ defines an injective and order-preserving map $\free(\calq)\times\free(\calr)$ into $\free(\calp)$:
\[\eta\bigl(\frdgs,\frdgu\bigr):=\frdg{\eta(\cals,\calu)}.\]
That this is a dual-implicative embedding now follows by the following straightforward calculations.
\begin{align*}
\eta(\cals,\calu)\join\eta(\calt,\calv)&=\setof{A^0}{A\in\cals\cup\calt}\cup\setof{C^1}{C\in\calu\cup\calv}\\
&=\eta(\cals\join\calt,\calu\join\calv)\\
&=\eta\bigl((\cals,\calu)\join(\calt,\calv)\bigr).
\end{align*}
\begin{align*}
\eta(\cals,\calu)\meet\eta(\calt,\calv)&=\setof{A^0\cup B^0}{A\in\cals\qand B\in\calt}\\
&\qquad\cup\setof{A^0\cup D^1}{A\in\cals\qand D\in\calv}\\
&\qquad\cup\setof{B^0\cup C^1}{B\in\calt\qand C\in\calu}\\
&\qquad\cup\setof{C^1\cup D^1}{C\in\calu\qand D\in\calv}.
\end{align*}
Since $A,B\incl Q$ and $C,D\incl R$,
\begin{gather*}
A^0\cup D^1=B^0\cup C^1=Q^0\cup R^1\\
\noalign{and}
A^0\cup B^0,\; C^1\cup D^1\incl Q^0\cup R^1,
\end{gather*}
so
\begin{align*}
\eta(\cals,\calu)\meet\eta(\calt,\calv)&=\setof{E^0}{E\in\cals\meet\calu}\cup\setof{F^1}{F\in\calt\meet\calv}\\
&=\eta(\cals\meet\calu,\calt\meet\calv)\\
&=\eta\bigl((\cals,\calt)\meet(\calu,\calv)\bigr),
\end{align*}
and
\begin{align*}
\eta(\cals,\calu)\dlatimpl\eta&(\calt,\calv)=\setof{B^0}{B\in\calt\text{ and }\set{B^0}\not\leq\eta(\cals,\calu)}\\
&\hskip1in\cup\setof{D^1}{D\in\calv\text{ and }\set{D^1}\not\leq\eta(\cals,\calu)}\\
&=\setof{B^0}{B\in\calt\text{ and }\set B\not\leq\cals}\cup\setof{D^1}{D\in\calv\text{ and }\set D\not\leq\calu}\\
&=\setof{B^0}{B\in\cals\dlatimpl\calt}\cup\setof{D^1}{D\in\calu\dlatimpl\calv}\\
&=\eta(\cals\dlatimpl\calt,\calu\dlatimpl\calv)\\
&=\eta\bigl((\cals,\calu)\dlatimpl(\calt,\calv)\bigr).
\end{align*}
Finally we have
\[\eta\bigl(\latz_{\free(\calq)\times\free(\calr)}\bigr)
\equiv\eta\bigl(\set Q,\set R\bigr)
=\set{Q^0,R^1}
=\set P\equiv\latz_{\free(\calp)},\]
and
\begin{align*}
\eta\bigl(\lato_{\free(\calq)\times\free(\calr)}\bigr)
&=\eta\bigl(\setof{\set b}{b\in Q},\setof{\set c}{c\in R}\bigr)\\
&=\setof{\set{(0,b)}}{b\in Q}\cup(\set 1\times R)\\
&\phantom{{}=\setof{\set{(0,b)}}{b\in Q}}\cup(\set 0\times Q)\cup\setof{\set{(1,c)}}{c\in R}\\
&\equiv\setof{\set{(0,b)}}{b\in Q}\cup\setof{\set{(1,c)}}{c\in R}\equiv\lato_{\free(\calp)}.\qedhere
\end{align*}
\end{proof}

\begin{proposition}
For any finite dual-implicative lattice \l\, there exists a finite partial ordering \calp\ such that $\l\dembeds\free(\calp)$.
\end{proposition}

\begin{proof}
We proceed by induction on the size $|\l|$ of \l. The smallest dual-implicative lattice is the two-element lattice {\bf 2}; easily ${\bf 2}\dembeds\free({\bf 2})$. For $|\l|>2$, suppose first that \l\ is \latz-irreducible and set 
\[\latz':=\bigmeet L\setminus\set{\latz}.\] 
Then $\latz<\latz'$ and $\latz'$ is the immediate successor of \latz, so $\l\iso\l[\latz',\lato]_{\ssf 0}$ (Definition \ref{segments}). By the induction hypothesis, there exists a finite partial ordering \calq\ such that $\l[\latz',\lato]\dembeds\free(\calq)$ and it is straightforward to verify that $\l\dembeds\free(\calq_{\ssf 0})$, where $\calq_{\ssf 0}$ is \calq\ enriched with a new least element.

Otherwise, \l\ is not \latz-irreducible so there exist $d,e>\latz$ in \l\ such that $d\meet e=\latz$. We shall show that in this case 
\[\l\dembeds\l[d,\lato]\times\l[e,\lato]\dembeds\free(\calq)\times\free(\calr)\dembeds\free(\calp)\]
for finite partial orderings \calq\ and \calr\ from the induction hypothesis and \calp\ from Lemma \ref{cprodpos}. Define $\eta:L\to L[d,\lato]\times L[e,\lato]$ by
\[\eta(a):=(a\join d,a\join e).\]
Obviously
\[a\leq b\Implies a\join d\leq b\join d\text{\quad and\quad }a\join e\leq b\join e\Implies\eta(a)\leq \eta(b),\]
but also
\begin{align*}
\eta(a)\leq\eta(b)&\qImplies a\join d\leq b\join d\qand a\join e\leq b\join e\\
&\qImplies (a\join d)\meet (a\join e)\leq (b\join d)\meet(b\join e)\\
&\qImplies a=a\join\latz=a\join(d\meet e)\leq b\join(d\meet e)=b\join\latz=b,
\end{align*}
so $\eta$ is an order-preserving injection. To complete the proof we verify that $\eta$ respects \latz, \lato, $\join$, $\meet$ and $\dlatimpl$. Clearly
\begin{align*}
\eta(\latz)&=(d,e)=\latz_{\l[d,\lato]\times\l[e,\lato]};\\
\eta(\lato)&=(\lato,\lato)=\lato_{\l[d,\lato]\times\l[e,\lato]}.
\end{align*}
For $a,b\in L$,
\begin{align*}
\eta(a\join b)
&=\bigl((a\join b)\join d,(a\join b)\join e\bigr)\\
&=\bigl((a\join d)\join(b\join d),(a\join e)\join(b\join e)\bigr)\\
&=\eta(a)\join\eta(b);\\
\noalign{\medskip}
\eta(a\meet b)
&=\bigl((a\meet b)\join d,(a\meet b)\join e\bigr)\\
&=\bigl((a\join d)\meet(b\join d),(a\join e)\meet(b\join e)\bigr)\\
&=\eta(a)\meet\eta(b).
\end{align*}
Towards $\dlatimpl$, note first that for any $x$,
\[b\leq a\join x\qImplies b\join d\leq a\join d\join x\qImplies b\leq a\join(x\join d)\]
or equivalently
\[x\geq(a\dlatimpl b)\qImplies x\geq(a\join d\dlatimpl b\join d)\qImplies(x\join d)\geq(a\dlatimpl b).\]
From these we conclude
\[(a\join d\dlatimpl b\join d)\leq(a\dlatimpl b)\leq(a\join d\dlatimpl b\join d)\join d\]
so
\[(a\dlatimpl b)\join d=(a\join d\dlatimpl b\join d)\join d.\]
Then using also the corresponding equation for $e$,
\begin{align*}
\eta(a)\dlatimpl\eta(b)&=\bigl((a\join d)\dlatimpl_d (b\join d),(a\join e)\dlatimpl_e (b\join e)\bigr)\\
&=\bigl((a\join d\dlatimpl b\join d)\join d,(a\join e\dlatimpl b\join e)\join e\bigr)\\
&=\bigl((a\dlatimpl b)\join d,(a\dlatimpl b)\join e\bigr)=\eta(a\dlatimpl b).\qedhere
\end{align*}
\end{proof}

\begin{corollary}
For every finite \latz- and \lato-irreducible dual-implicative lattice \l, there exists a finite partial ordering \calp\ such that $\l\dembeds\free(\calp)^{\ssf 1}_{\ssf 0}$.
\end{corollary}

\begin{proof}
If \l\ has two or three elements, the result is clear. If \l\ has at least four elements, let
\[\latz':=\bigmeet L\setminus\set{\latz}\qand\lato':=\bigjoin L\setminus\set{\lato}.\]
Then $\l\iso\l[\latz',\lato']^{\ssf 1}_{\ssf 0}$. By the proposition there exists a finite partial ordering \calp\ such that $\l[\latz',\lato']\dembeds\free(\calp)$ so easily $\l\dembeds\free(\calp)^{\ssf 1}_{\ssf 0}$.
\QED
\end{proof}

The partial ordering \Dgt\ plays a role here via the following well-known

\begin{proposition}
Every finite partial ordering is embeddable in \Dgt. \noproof
\end{proposition}

For a proof of this and much more --- that every countable partial ordering is embeddable in the r.e. Turing degrees \Dgpt\ --- see, for example, \cite[Theorem 8.2.17]{Hi} or \cite[Exercise VII.2.2]{Soa}.

\begin{corollary}
For every finite \latz- and \lato-irreducible dual-implicative lattice \l, $\l\dembeds\free(\Dgt)^{\ssf 1}_{\ssf 0}$. \noproof
\end{corollary}

To complete the proof of the Embedding Theorem and Theorem I, it remains to establish

\begin{proposition}
$\free(\Dgt)^{\ssf 1}_{\ssf 0}\dembeds\Dgs$.
\end{proposition}

\begin{proof}
For $\dga\in\Dgt$ set
\[[\dga]:=\setof{h\in\pre\omega\omega}{\degt(h)\not\leqt\dga}\qand\sdg\dga:=\degs([\dga]).\]
Then for $A\in\poweromm(\Dgt)$ and $\cals\in\poweromm\poweromm(\Dgt)$ set
\[\eta(\set A):=\bigmeet_{\dga\in A}\sdg\dga\qand
\eta(\cals):=\bigjoin_{A\in\cals}\eta\bigl(\set A\bigr),\]
where these meets and joins are, of course, in \Dgs. Once we verify below that 
\[\eta(\cals)\leqs\eta(\calt)\qIff\cals\leq\calt,\]
we can extend $\eta$ to our final mapping $\free(\Dgt)^{\ssf 1}_{\ssf 0}\to\Dgs$ by
\[\eta\bigl(\frdgs\bigr):=\eta(\cals),\quad\eta(\latz^*):=\degs(\set\emptyset),\qand\eta(\lato^*):=\degs(\emptyset).\]
Before starting to establish that $\eta$ is a dual-implicative embedding,we note some properties of the sets $[\dga]$ and the strong degrees $\sdg\dga$: for all $\dga\in\Dgt$,
\begin{enumerate}
\item[(1)]$\dga\leqt\dgb\Iff \sdg\dga\leqs\sdg\dgb$;
\item[(2)]$[\dga]$ is \dff{homogeneous}: $(\forall\sigma\in\pre{<\omega}\omega)\;\forall h\;\bigl(h\in[\dga]\rImplies\sigma^\frown h\in[\dga]\bigr)$;
\item[(3)]$\sdg\dga$ is meet-irreducible;
\item[(4)]$\sdg\dga$ is join-irreducible.
\end{enumerate}
For (1), $\dga\leqt\dgb\Iff[\dgb]\incl[\dga]\Implies[\dga]\leqs[\dgb]$. Conversely, for a recursive functional $\Phi$, $\Phi(h)\leqt h$, so
\begin{align*}
\Phi:[\dgb]\to[\dga]&\qImplies(\forall h\in[\dgb])\;\Phi(h)\not\leqt a\\
&\qImplies(\forall h\in[\dgb])\;h\not\leqt\dga\\
&\qImplies[\dgb]\incl[\dga]\qImplies\dga\leqt\dgb,
\end{align*}
so also $[\dga]\leqs[\dgb]\Implies\dga\leqt\dgb$.
(2) is immediate, since $h\eqt\sigma^\frown h$. For (3), suppose that for some $P,Q\incl\pre\omega\omega$, $P\meet Q\leqs[\dga]$, say
\[\Phi:[\dga]\to((0)^\frown P)\cup( (1)^\frown Q).\]
Then if $P\not\leqs[\dga]$, some element of $[\dga]$ is mapped into $(1)^\frown Q$ so for some finite sequence $\sigma$, $\Phi(\sigma)(0)=1$, and thus by homogeneity,
\[Q\leqs\setof{h\in[\dga]}{\sigma\incl h}\eqs[\dga].\]
For (4), for any $\dga\in\Dgt$, let $g_\dga$ be a function with $\degt(g_\dga)=\dga$. Then $[\dga]\not\leqs\set{g_\dga}$, since for any recursive functional $\Phi$, $\Phi(g_\dga)\leqt g_\dga$ and thus $\Phi(g_\dga)\notin[\dga]$. But
\[[\dga]\not\leqs P\rImplies P\not\incl[\dga]\rImplies(\exists f\in P)\;\degt(f)\leqt\dga\rImplies P\leqs\set{g_\dga},\]
so 
\[P,Q\les[\dga]\qImplies P\join Q\leqs[\dga]\meet\set{g_\dga}\les[\dga].\]

We establish next a rather special instance of join-irreducibility:
\begin{enumerate}
\item[(5)]For any $A\in\poweromm(\Dgt)$, $\calt\in\poweromm\poweromm(\Dgt)$ and $X\incl\pre\omega\omega$,
\end{enumerate}
\[\eta\bigl(\set A\bigr)\leqs\eta(\calt)\join\degs(X)\rImplies\eta(\set A)\leqs\eta(\calt)\qor\eta(\set A)\leqs\degs(X).\]
Since $\eta(\calt)=\bigjoin_{B\in\calt}\bigmeet_{b\in B}\sdg\dgb$, by distributivity also
\[\eta(\calt)=\bigmeet_{F\in\prod\calt}\bigjoin_{B\in\calt}\sdg{F(B)}.\]
Assume the hypothesis of (5) and that $\eta(\set A)\not\leqs\eta(\calt)$. Then for some $F\in\prod\calt$, $\eta(\set A)\not\leqs\bigjoin_{B\in\calt}\sdg{F(B)}$. Set
\[Y:=\bigoplus_{B\in\calt}[F(B)]\quad\text{so}\quad\eta(\set A)\not\leqs\degs(Y),\]
and in particular, for each $\dga\in A$, $[\dga]\not\leqs Y$. But by hypothesis,
\[\eta(\set A)\leqs\degs(X)\join\degs(Y),\]
so there exists a recursive functional
\[\Phi:X\join Y\to\bigcup_{a\in A}(i_{\dga})^\frown[\dga].\]
For $f\in X$ set
\[\bigl(i(f),\sigma(f)\bigr):=\text{least }(i,\sigma)\;[\Phi(f\oplus\sigma)\simeq i],\]
and for $\dga\in A$,
\begin{align*}
X_\dga&:=\setof{f\in X}{i(f)=i_\dga};\\
\Phi_\dga(f\oplus g)(x)&:=\Phi\bigl(f\oplus\sigma(f)^\frown g\bigr)(x+1).
\end{align*}
Since by (2) $Y$ is homogeneous,
\[\Phi_\dga:X_\dga\join Y\to[\dga]\quad\text{so}\quad[\dga]\leqs X_\dga\join Y.\]
But we showed above that $[\dga]\not\leqs Y$ and by (4) $[\dga]$ is join-irreducible, so $[\dga]\leqs X_\dga$ and hence
\[\eta(\set A)\leqs\bigmeet_{\dga\in A}\degs(X_\dga)\leqs\degs(X),\]
since $\Psi(f):=\bigl(i(f)\bigr)^\frown f$ witnesses that $\bigmeet_{\dga\in A}X_\dga\leqs X$. This completes the proof of (5).

Now we have
\begin{align*}
\eta\bigl(\set A\bigr)\leqs\eta\bigl(\set B\bigr)
&\rIff\bigmeet_{\dga\in A}\sdg\dga\leqs\bigmeet_{b\in B}\sdg\dgb\\
&\rIff(\forall \dgb\in B)\;\Bigl(\bigmeet_{\dga\in A}\sdg\dga\Bigr)\leqs\sdg\dgb\\
&\rIff(\forall \dgb\in B)(\exists \dga\in A)\sdg\dga\leqs\sdg\dgb&&\text{by (3)}\\
&\rIff(\forall \dgb\in B)(\exists \dga\in A)\;\dga\leqt\dgb&&\text{by (1)}\\
&\rIff\set A\leq\set B
\end{align*}
and
\begin{align*}
\eta(\cals)\leqs\eta(\calt)
&\rIff\bigjoin_{A\in\cals}\eta\bigl(\set A\bigr)\leqs\eta(\calt)\\
&\rIff(\forall A\in\cals)\;\eta\bigl(\set A\bigr)\leqs\eta(\calt)\\
&\rIff(\forall A\in\cals)(\exists B\in\calt)\;\eta\bigl(\set A\bigr)\leqs\eta\bigl(\set B\bigr)&&\text{by (5) iterated}\\
&\rIff\cals\leq\calt.
\end{align*}
It follows as usual that $\eta$ is well-defined, injective and order-preserving on $\free(\Dgt)$. To verify that $\eta$ respects $\join$ and $\meet$ is straightforward and left to the reader, and we turn to $\dlatimpl$. We need to show that for all \cals\ and \calt,
\[\eta(\cals)\dlatimpl\eta(\calt)=\eta(\cals\dlatimpl\calt),\]
or equivalently, by the definition of $\dlatimpl$, for all $X$,
\[\eta(\calt)\leqs\eta(\cals)\join\degs(X)\rIff\eta(\cals\dlatimpl\calt)\leqs\degs(X).\]
Now by the definition of $\dlatimpl$ in $\free(\Dgt)$,
\begin{align*}
\eta(\cals\dlatimpl\calt)&=\eta\bigl(\setof{B\in\calt}{\set B\not\leqs\cals}\bigr)\\
&=\bigjoin\setof{\eta\bigl(\set B\bigr)}{B\in\calt\;\text{and}\;\set B\not\leq\cals}\\
&=\bigjoin\setof{\eta\bigl(\set B\bigr)}{B\in\calt\;\text{and}\;\eta\bigl(\set B\bigr)\not\leqs\eta(\cals)} &&\text{by (1)}.
\end{align*}
Hence it will suffice to show that for all $X$, 
\begin{multline*}
\eta(\calt)\leqs\eta(\cals)\join\degs(X)\\
\Iff(\forall B\in\calt)\bigl[\eta\bigl(\set B\bigr)\leqs\eta(\cals)\qor\eta\bigl(\set B\bigr)\leqs\degs(X)\bigr].
\end{multline*}
The implication (\Larrow) is immediate from the definition and (\Rarrow) follows directly from (5).
\QED
\end{proof}
\end{section}

\begin{section}{Proof of Theorem K}
The proof will require a substantial number of lemmas; the first is immediate from the definitions.

\begin{lemma} 
For any $A\incl\omega$ and any $A$-full set $P$, $P\in\Sigma^0_2[A]$ and $\mu(P)=1$. \noproof
\end{lemma}
 
\begin{lemma}\label{randomfull}
${\ssf R}_n$ is $\zerot^{(n-1)}$-full; hence ${\ssf R}_n\in\Sigma^0_{n+1}$ and $\mu({\ssf R}_n)=1$.
\end{lemma}

\begin{proof}
We give the proof for $n=1$ --- for the general case just relativize the proof to $\zerot^{(n-1)}$. 
As in the discussion preceding Proposition \ref{TPispzo}, fix effective enumerations \functionof{T_a}{a\in\omega} of all \pzo\ trees and  \functionof{T_{a,s}}{a\in\omega} of their recursive approximations. Set
\[U_{a,s}:=\bigsetof{\sigma\in\pre{<\omega}2}{\bigl(\set a(a)\downarrow\rand\mu([T_{\set a(a),s}])\geq1-2^{-a}\bigr)\Implies\sigma\in T_{\set a(a),s}}\]
and $Q_a^*:=\bigcap_{s\in\omega}[U_{a,s}]$, 
so $Q_a^*\in\pzo$, $\mu(Q_a^*)\geq1-2^{-a}$ and
\[\bigl(\set a(a)\downarrow\rand\mu([T_{\set a(a)}])\geq 1-2^{-a}\bigr)\qImplies Q_a^*=[T_{\set a(a)}].\]
Finally, set 
\[Q_n:=\bigcap_{a>n} Q_a^*\qqand Q:=\bigcup_{n\in\omega}Q_n.\]
Easily $Q$ is recursively full so ${\ssf R}_1\incl Q$. 
But also $Q\incl {\ssf R}_1$: if $f\notin {\ssf R}_1$ there exists a recursively full set $P=\bigcup_{n\in\omega}P_n$ such that $f\notin P$.
For each $n$, choose $a_n>n$ so that for all $k$, $P_k=[T_{\set {a_n}(k)}]$;
in particular, since $\mu(P_{a_n})\geq 1-2^{-a_n}$,
\[P_{a_n}=[T_{\set {a_n}(a_n)}]=Q_{a_n}^*\supseteq Q_n.\]
Hence for each $n$, $f\notin Q_n$ and thus $f\notin Q$. The other clauses follow from the preceding lemma.
\QED
\end{proof}

\begin{definition} 
$P\incl\pre\omega k$ is \dff{$k$-separating} iff there exist r.e.\ sets \hfill\break
$A_0,\ldots,A_{k-1}\incl\omega$ such that
\[P=\setof{f\in\pre\omega k}{(\forall n\in\omega)\,n\notin A_{f(n)}}.\]
Note that ordinary separating sets 
\[{\ssf Sep}(A_0,A_1)=\setof C{A_0\incl C\qand C\cap A_1=\emptyset},\]
in particular \dnr2, are 2-separating, and any $k$-separating set is \pzo.
\end{definition} 

\begin{lemma}[{\cite[Theorem 7.5]{Si1}}]\label{posmeas}
For any $k$-separating set $P$ and any \pzo\ set $Q$ of positive measure,
\[P\leqw Q\Implies P\hbox{ has a recursive element}.\]
Hence for the corresponding weak degrees, if $\dgp\leqw\dgq$ then $\dgp=\zerow$ and in particular $\onew\not\leq\dgq$.
\end{lemma}
 
\begin{proof} 
Assume that $P=\setof{f\in\pre\omega k}{(\forall n\in\omega)\,n\notin A_{f(n)}}\leqw Q$ and $Q$ is of positive measure, and for each index $a$, set 
\[Q_a:=\setof{g\in Q}{\set a^g\in P}.\]
 Then $Q=\bigcup_{a\in\omega}Q_a$, $P\leqs Q_a$, and by countable additivity of measure, for some $a$, $\mu(Q_a)>0$. Hence we may from the beginning assume that $P\leqs Q$ and fix a recursive functional $\Phi:Q\to P$; by Proposition \ref{compactness} we may also assume that $\Phi$ is total.

By standard arguments of measure theory there exist an open set $V\supseteq Q$ and a clopen set $U\incl V$ such that
\[\mu(V\setminus Q)<\frac{\mu(Q)}{k+1}\qand \mu(V\setminus U)<\frac{\mu(Q)}{k+1}.\]
Then
\[\mu(U\setminus Q)\leq\mu(V\setminus Q)<\frac{\mu(Q)}{k+1}\qand
\mu(Q\setminus U)\leq\mu(V\setminus U)<\frac{\mu(Q)}{k+1},\]
whence
\[\frac{k\cdot\mu(Q)}{k+1}<\mu(Q\cap U)\leq\mu(U)\]
so
\[\mu(U\setminus Q)<\frac{(k+1)\mu(U)}{k\cdot(k+1)}=\frac{\mu(U)}{k}.\]
Set $U^n_i:=\setof{f\in U}{\Phi(f)(n)=i}$. 
For each $n$ there exists $i<k$ such that
$\mu(U^n_i)\geq\frac{\mu(U)}{k}$ and for any such $i$, $U^n_i\not\incl U\setminus Q$
so there exists $f\in U^n_i\cap Q$, and since $\Phi(f)\in P$, $n\notin A_i$. Hence
\[g(n):=\hbox{least }i\,\left[\mu(U^n_i)\geq\frac{\mu(U)}{k}\right]\]
is a recursive element of $P$.
\QED
\end{proof}

\begin{lemma}
$\dgd, \dgr_1$ and $\dgr_2^*\in\Dgpw$ and $\zerow<{\bf r}_1\leq{\bf r}_2^*<\onew$
\end{lemma}

\begin{proof}
By Lemma \ref{existspzo}, $\dgd^*:=\dgd\meet\onew\in\Dgpw$, and by Lemma \ref{randomfull},
\[\dnr2\incl\dnr{}\Implies\dgd\leq\onew\Implies\dgd^*=\dgd.\]
Similarly $\dgr_1^*\in\Dgpw$, but ${\ssf R}_1$ is a union of non-empty \pzo\ sets and for any one $S$ of these, 
\[\dgr_1\leq\degw(S)\leq\onew\quad\hbox{so}\quad\dgr_1^*=\dgr_1.\]
${\ssf R}_2\in\Sigma^0_3$, so again by Lemma \ref{existspzo}, $\dgr_2^*\in\Dgpw$.
The ordering relationships are immediate from Lemmas \ref{randomfull} and \ref{posmeas}.
\QED
\end{proof}

\begin{lemma}\label{nullnotrandom}
For any $P\in\pzo$, if $\mu(P)=0$, then $\pre\omega 2\setminus P$ is recursively full so $P\cap\random1=\emptyset$.
\end{lemma}

\begin{proof}
If $P=[T]$, set $P_s:=[T_{a,s}]$ so $P=\bigcap_{s\in\omega} P_s$, an intersection of clopen sets. If $\mu(P)=0$, then $\lim_{s\to\infty}\mu(P_s)=0$ so if
$$h(n):=\hbox{least }s\left[\mu(P_s)<2^{-n}\right],$$
$\pre\omega2\setminus P=\bigcup_{n\in\omega}(\pre\omega 2\setminus P_{h(n)})$ is recursively full and hence $\random1\incl\pre\omega2\setminus P$.
\QED
\end{proof}

\begin{lemma}[{\cite[Theorem 4.19]{Si1}}]\label{almostrec}
For every $\emptyset\not=P\incl\pre\omega 2$, if $P\in\Sigma^0_2$, then $P$ has an almost recursive element.
\end{lemma}

\begin{proof}
Since $\Sigma^0_2$ sets are unions of \pzo\ sets this is immediate from Proposition \ref{pzobases} (ii).
\QED
\end{proof}

\begin{lemma}[{\cite[Remark 2.8]{DoSi}}]\label{almostrecprops}
For every almost recursive function $g$,
$g\notin\random2$.
\end{lemma}

\begin{proof} 
We construct an $\zerot'$-full set $P$ such that
\[g\in P\Implies(\exists f\leqt g)\hbox{ $f$ is not recursively bounded}.\]
For each $n$ and $a$, set $k^a_n:=2^{a+n+1}$ and partition \pre\omega2 into $k^a_n$-many pairwise disjoint clopen sets $Q^a_{n,0},\ldots,Q^a_{n,k_n^a-1}$ each of measure $1/k^a_n$.
Fix $n$ and suppress it in subscripts. For $i<k^a-1$ set
\[s^a_0:=a\qand s^a_{i+1}:\simeq\hbox{least }s>s^a_i\,\left[\set a_s(s^a_i)\downarrow\right].\]
The relation
\[s^a_i\simeq m\qIff(\forall j<i)\;s^a_j\simeq\hbox{least }s<m\;\left[\set a_s(s^a_j)\downarrow\right]\]
is recursive, hence the functional $\Phi$ defined by
\[\Phi(g)(m):\simeq\text{max }\setof{\set a(m)+1}{a\leq m\text{ and }(\exists i<m)\bigl[g\in Q^a_i\text{ and }s^a_i\simeq m\bigr]}\]
with the usual convention that $\text{max }\emptyset\simeq 0$ is partial recursive. Set
\[j^a:\simeq\text{largest }i\;[s^a_i\downarrow]\qand m^a:\simeq s^a_{j^a}.\]
Then $\Phi(g)(m)\downarrow$ unless for some $a$, $m\simeq m^a$ and $g\in Q^a_{j^a}$.
Hence, setting
\[P_n:=\setof g{\Phi(g)\text{ is total}},\]
we have $P_n\in\pzo[\zerot']$ because
\[P_n=\setof g{\forall a\forall i\bigl[ g\in Q^a_i\text{ and }\bigl( s^a_i\downarrow\Implies s^a_{i+1}\downarrow\bigr)\bigr]},\]
and $P_n$ has measure at least $1-2^{-n}$ because
\[\mu\bigl(\pre\omega\omega\setminus P_n\bigr)\leq\mu\Bigl(\bigcup_{a\in\omega} Q^a_{j^a}\Bigr)\leq\sum_{a\in\omega}\frac 1{k_a}=\frac 1{2^n}.\]
Furthermore, if $g\in P_n$ and \set a is total, then for the unique $i$ such that $g\in Q^a_i$,
\[\Phi(g)(s^a_i)>\set a(s^a_i),\]
so $\Phi(g)$ is not recursively bounded. Hence $P:=\bigcup_{n\in\omega}P_n$ is the desired
$\zerot'$-full set.
\QED
\end{proof}

\begin{lemma}
$\dgr_1<\dgr_2^*$.
\end{lemma}

\begin{proof}
By Lemma \ref{almostrec}, \random1 has an almost recursive element $g$.\hfill\break
If $\dgr_2^*\leq\degw(\set g)$, then since \set g is a singleton, either 
\[\dgr_2\leq\degw(\set g)\qor \onew\leq\degw(\set g).\]
The first alternative is impossible by Lemma \ref{almostrecprops} (ii). 
If the second holds, then by Lemma \ref{almostrecprops} (i) there exists a total $\Phi:\set g\to\dnr2$, so $\dnr2\leqw\Phi^{-1}(\dnr2)$, which is a \pzo\ class containing $g$, hence of positive measure by Lemma \ref{nullnotrandom}, contrary to Lemma \ref{posmeas}.
\QED
\end{proof}

\begin{lemma}
$\zerow<\dgd<\dgr_1$.
\end{lemma}

\begin{proof}
$\zerow<\dgd$ since obviously {\ssf DNR} has no recursive elements.
That $\dgr_1\not\leq\dgd$ follows from \cite{Ku}: there exists $f\in{\ssf DNR}$ such that $\degt(f)$ is minimal, together with the easy observation that for $f\in\random1$, $\degt(f)$ is not minimal because
\[\zerot<\degt(f^{\hbox{odd}}),\degt(f^{\hbox{even}})<\degt(f).\]
To show that $\dgd\leq\dgr_1$, we show as follows that
\[A\in\random1\Implies(\exists g\leqt A)\,g\in{\ssf DNR}.\]
Set $W_x\incl_k A$ iff $|W_x|>k$ and the first $k$ numbers enumerated into $W_x$ are in $A$,
and $P_{x,k}:=\setof A{W_x\not\incl_kA}$. Easily $\mu(P_{x,k})\geq 1-2^{-k}$.
With notation as in the proof of Lemma \ref{randomfull},
choose $a_{x,n}>n$ such that $\forall k\,\bigl(P_{x,k}=[T_{\set{a_{x,n}}(k)}]\bigr)$.
Then 
\[P_{x,a_{x,n}}=[T_{\set{a_{x,n}}(a_{x,n})}]=Q_{a_{x,n}}^*\supseteq Q_n,\]
so
\[A\in Q_n\qand W_x\incl A\Implies |W_x|<a_{x,n}.\]
For a (necessarily) infinite set $A\in\random1$, choose $\bar n$ and $h$ such that 
\[A\in Q_{\bar n} \qand h(x):= a_{x,\bar n}.\]
There exists $g\leqt A$ such that
\[W_{g(b)}=\cases{\hbox{the first $h(\set b(b))$ elements of }A,
&if $\set b(b)\downarrow$;\cr
\emptyset,&otherwise.\cr}\]
Then $g\in{\ssf DNR}$, since if $g(b)=\set b(b)$, $W_{\set b(b)}=W_{g(b)}\incl A$ so
\begin{align*}
|W_{g(b)}|&=h(g(b))\hbox{\quad by definition of }g, \text{ but}\\
|W_{g(b)}|&<h(g(b))\hbox{\quad by definition of }h,
\end{align*}
a contradiction. 
\QED
\end{proof}

\begin{lemma}
For any set $P\in\pzo$, if $\mu(P)>0$, then there exists a recursively full set $Q$ such that $P\eqw Q$.
\end{lemma}

\begin{proof}
Set $f^{(k)}(n):=f(k+n)$ and 
\[P^*:=\setof f{(\exists k\in\omega)\,f^{(k)}\in P}.\]
Clearly $P\incl P^*$ and $P\eqw P^*$ so it suffices to construct a recursively full set $Q$ with $P\incl Q\incl P^*$.
For trees $T$ and $U$, set
\[U+T:=U\cup\setof{\tau^\frown(i)^\frown\sigma}{\tau\in U\text{ and }\tau^\frown(i)\notin U\text{ and }\sigma\in T};\]
$U+T$ has a copy of $T$ attached to each leaf of $U$. Easily
\[1-\mu([U+T])=(1-\mu([U]))(1-\mu([T])).\]
Choose $U_0$ such that $P=[U_0]$, and for each $n$ set
\[U_{n+1}:=U_n+U_0\qqand Q_n:=[U_n],\]
so
\[Q_0\incl Q_1\incl\cdots\qand 1-\mu(Q_n)=(1-\mu(P))^n.\]
Since $\mu(P)>0$, there exists $\ell\in\omega$ such that $(1-\mu(P))^\ell\leq 1/2$, so for each $n$,
\[1-\mu(Q_{\ell n})\leq 2^{-n}\qand Q:=\bigcup_{n\in\omega}Q_n=\bigcup_{n\in\omega}Q_{\ell n}\]
is recursively full. Obviously $P\incl Q$, and $Q\incl P^*$ since each
$f\in Q$ is of the form $\sigma^\frown g$ for $g\in P$, so $f^{(|\sigma|)}\in P$ and hence $f\in P^*$.
\QED
\end{proof}

\begin{proposition}
$\dgr_1$ is the largest element of \Dgpw\ that contains a \pzo\ set of positive measure.
\end{proposition}

\begin{proof}
Since \random1 is of measure 1 and the union of \pzo\ sets, at least one of these $P$ is of positive measure. $P\leqw\random1$ by the preceding lemma and $\random1\leqw P$, since $P\incl\random1$.
For any other $\dgq\in\Dgpw$, $\dgq\leq\dgr_1$ also by the preceding lemma.
\QED
\end{proof}
\end{section}

\end{document}
\bye